\RequirePackage{luatex85}
\documentclass{amsart}
\usepackage{etex}
\usepackage{fixltx2e}

\usepackage[english,french]{babel}

\usepackage[usenames,dvipsnames]{xcolor}
\usepackage[bookmarksopen,bookmarksdepth=2]{hyperref}
\hypersetup{colorlinks=true,citecolor=NavyBlue,linkcolor=BrickRed,urlcolor=Green} 
\usepackage[mathscr]{eucal}
\usepackage{tikz}
\usepackage{tikz-cd}
\usepackage{tikzsymbols}
\usepackage{adjustbox}
\usepackage{enumitem}
\usepackage{dsfont}

\usepackage{microtype}

\renewcommand{\mathbb}{\mathbf}

\usepackage{amssymb,amsmath,amsfonts,amsthm,epsfig,amscd}
\usepackage{stmaryrd}
\usepackage[all,cmtip,poly]{xy}
\usepackage{color}

\newcommand{\Fpbartimes}{\overline{\F}^\times_p}

\newcommand{\gL}{\mathfrak{L}}

\newcommand{\irred}{\operatorname{irred}}
\newcommand{\red}{\operatorname{red}}

\newcommand{\fbar}{\overline{f}}
\newcommand{\Ann}{\operatorname{Ann}}
\newcommand{\Adist}{{A^{\operatorname{dist}}}}
\newcommand{\Akfree}{{A^{\operatorname{k-free}}}}
\newcommand{\Bdist}{{B^{\operatorname{dist}}}}
\newcommand{\Btwist}{{B^{\operatorname{twist}}}}
\newcommand{\Bkfree}{{B^{\operatorname{k-free}}}}
\newcommand{\Ckfree}{{C^{\operatorname{k-free}}}}
\newcommand{\gEdist}{{\widetilde{\gE}^{\operatorname{dist}}}}
\newcommand{\gEkfree}{{\widetilde{\gE}^{\operatorname{k-free}}}}
\newcommand{\dist}{{\operatorname{dist}}}
\newcommand{\kfree}{{{\operatorname{k-free}}}}

\newcommand{\To}{\longrightarrow}
\newcommand{\isoto}{\stackrel{\sim}{\To}}
\newcommand{\cotimes}{\, \widehat{\otimes}}
\newcommand{\pr}{\operatorname{pr}}

\newcommand{\BT}{\operatorname{BT}}
\newcommand{\ocZ}{\overline{\cZ}}
\newcommand{\cZbar}{\overline{\cZ}}

\newcommand{\Kbar}{\bar{K}}
\newcommand{\gP}{\mathfrak{P}}
\newcommand{\ghat}{\hat{g}}

\newcommand{\JH}{\operatorname{JH}}
 \newcommand{\sigmabar   }{\overline{\sigma}}   
\def\iso{\buildrel \sim \over \longrightarrow}

\newcommand{\xbar}{\bar{x}}
\newcommand{\ybar}{\bar{y}}

\newcommand{\id}{\operatorname{id}}

\newtheorem{thm}[subsubsection]{Theorem}
\newtheorem{lemma}[subsubsection]{Lemma}
\newtheorem{lem}[subsubsection]{Lemma}

\newtheorem{ithm}[subsection]{Theorem}
\newtheorem{cor}[subsubsection]{Corollary}

\newtheorem{prop}[subsubsection]{Proposition}

\newtheorem{alemma}[subsection]{Lemma}
\newtheorem{alem}[subsection]{Lemma}
\newtheorem{adefn}[subsection]{Definition}

\theoremstyle{definition}
\newtheorem{df}[subsubsection]{Definition}
\newtheorem{defn}[subsubsection]{Definition}

\theoremstyle{remark}
\newtheorem{remark}[subsubsection]{Remark}
\newtheorem{iremark}[subsection]{Remark}
\newtheorem{rem}[subsubsection]{Remark}
\newtheorem{aremark}[subsection]{Remark}

\newtheorem{example}[subsubsection]{Example}

\newtheorem{assumption}[subsubsection]{Assumption}

\def\numequation{\addtocounter{subsubsection}{1}\begin{equation}}
\def\nummultline{\addtocounter{subsubsection}{1}\begin{multline}}
\def\anumequation{\addtocounter{subsection}{1}\begin{equation}}
\def\anummultline{\addtocounter{subsection}{1}\begin{multline}}
\renewcommand{\theequation}{\arabic{section}.\arabic{subsection}.\arabic{subsubsection}}

\newif\iffinalrun
\finalruntrue
\iffinalrun
\else
\fi

\iffinalrun
  \newcommand{\need}[1]{}
  \newcommand{\mar}[1]{}
\else
  \newcommand{\need}[1]{{\tiny *** #1}}
  \newcommand{\mar}[1]{\marginpar{\raggedright\tiny fixme #1}}
\fi

\newcommand{\F}{\FF}

\newcommand{\Q}{\QQ}

\newcommand{\Z}{\ZZ}

\newcommand{\m}{\frakm}

\newcommand{\FF}{{\mathbb F}}
\newcommand{\GG}{{\mathbb G}}

\newcommand{\QQ}{{\mathbb Q}}

\newcommand{\ZZ}{{\mathbb Z}}

\renewcommand{\bf}{\ensuremath{\mathbf{f}}}

\newcommand{\cC}{{\mathcal C}}
\newcommand{\cD}{{\mathcal D}}
\newcommand{\cE}{{\mathcal E}}
\newcommand{\cF}{{\mathcal F}}
\newcommand{\cG}{{\mathcal G}}

\newcommand{\cI}{{\mathcal I}}

\newcommand{\cM}{{\mathcal M}}

\newcommand{\cO}{{\mathcal O}}
\newcommand{\cP}{{\mathcal P}}

\newcommand{\cR}{{\mathcal R}}
\newcommand{\cS}{{\mathcal S}}

\newcommand{\cU}{{\mathcal U}}

\newcommand{\cX}{{\mathcal X}}
\newcommand{\cY}{{\mathcal Y}}
\newcommand{\cZ}{{\mathcal Z}}

\newcommand{\frakm}{\mathfrak{m}}

\newcommand{\gE}{\mathfrak{E}}
\newcommand{\gM}{\mathfrak{M}}
\newcommand{\gN}{\mathfrak{N}}
\newcommand{\gS}{\mathfrak{S}}

\newcommand{\tB}{\mathrm{B}}

\newcommand{\Fbar}{\overline{\F}}
\newcommand{\Qbar}{\overline{\Q}}
\newcommand{\Zbar}{\overline{\Z}}

\newcommand{\Fp}{\F_p}

\newcommand{\Fpbar}{\Fbar_p}

\newcommand{\Fpbarx}{\Fpbar^{\times}}

\newcommand{\Zp}{\Z_p}

\newcommand{\Zpbar}{\Zbar_p}

\newcommand{\Qp}{\Q_p}

\newcommand{\Qpbar}{\Qbar_p}

\DeclareMathOperator{\Aut}{Aut}
\DeclareMathOperator{\coker}{coker}

\DeclareMathOperator{\End}{End}
\DeclareMathOperator{\Ext}{Ext}
\DeclareMathOperator{\kExt}{ker-Ext}
\DeclareMathOperator{\Fil}{Fil}
\DeclareMathOperator{\Gal}{Gal}
\DeclareMathOperator{\GL}{GL}

\DeclareMathOperator{\Hom}{Hom}
\DeclareMathOperator{\Tor}{Tor}
\DeclareMathOperator{\im}{im}
\DeclareMathOperator{\Ind}{Ind}

\DeclareMathOperator{\Spec}{Spec}

\DeclareMathOperator{\Spf}{Spf}

\DeclareMathOperator{\Sym}{Sym}

\DeclareMathOperator{\val}{val}

\newcommand{\ab}{\mathrm{ab}}
\newcommand{\cris}{\mathrm{cris}}

\newcommand{\un}{\mathrm{un}}
\newcommand{\ur}{\mathrm{ur}}

\newcommand{\Bcris}{\tB_{\cris}}

\newcommand{\Dpcris}{\operatorname{D_{pcris}}}

\newcommand{\rhobar}{\overline{\rho}}

\newcommand{\into}{\hookrightarrow}

\newcommand{\toisom}{\buildrel\sim\over\to}

\newcommand{\Ga}{\GG_a}
\newcommand{\Gm}{\GG_m}

\newcommand{\hchar}{h} 

\newcommand{\dd}{\mathrm{dd}}

\newcommand{\Art}{{\operatorname{Art}}}
\newcommand{\col}{\colon}

\newcommand{\varepsilonbar}{\overline{\varepsilon}}
\newcommand{\rbar}{\overline{r}}

\newcommand{\K}[1]{\mathcal{K}(#1)}

\newcommand{\scrC}{\mathscr{C}}
\newcommand{\Czero}{\scrC^0}
\newcommand{\Czerofrac}{\Czero_{u}}
\newcommand{\Cone}{\scrC^1}
\newcommand{\Conefrac}{\Cone_{u}}
\newcommand{\Hzero}{\mathscr{H}}
\newcommand{\mumap}{\partial}

\newcommand{\tgP}{\widetilde{\gP}}
\newcommand{\czero}{c}

\newcommand{\cegsBbasechangelemma}{\cite[Lem.~2.1.4]{cegsB}}
\newcommand{\cegsBtypesareunmixed}{\cite[Prop.~4.2.12]{cegsB}}
\newcommand{\cegsBdeterminantconditionexplicit}{\cite[Lem.~4.2.11]{cegsB}}
\newcommand{\cegsBtresramlemma}{\cite[Lem.~A.4]{cegsB}}
\newcommand{\cegsBdieudonnesection}{\cite[Sec.~4.6]{cegsB}}
\newcommand{\cegsBtauprop}{\cite[Prop.~3.3.5]{cegsB}}
\newcommand{\cegsBctaubt}{\cite[Prop.~4.2.7]{cegsB}}

\newcommand{\cegsBpotBTreps}{\cite[Lem.~4.2.16]{cegsB}}
\newcommand{\cegsBpotBTrepsremark}{\cite[Rem.~4.2.8]{cegsB}}
\newcommand{\cegsBCmainresults}{\cite[Cor.~4.5.3, Prop.~5.2.21]{cegsB}}
\newcommand{\cegsBZmainresults}{\cite[Thm.~5.1.2, Prop.~5.2.20]{cegsB}}

\begin{document}
\selectlanguage{english}
\title[Components of moduli stacks]{Components of moduli stacks of two-dimensional Galois
representations}

\author[A. Caraiani]{Ana Caraiani}\email{a.caraiani@imperial.ac.uk}
\address{Department of Mathematics, Imperial College London, London SW7 2AZ, UK}

\author[M. Emerton]{Matthew Emerton}\email{emerton@math.uchicago.edu}
\address{Department of Mathematics, University of Chicago,
5734 S.\ University Ave., Chicago, IL 60637, USA}

\author[T. Gee]{Toby Gee} \email{toby.gee@imperial.ac.uk} \address{Department of
  Mathematics, Imperial College London,
  London SW7 2AZ, UK}

\author[D. Savitt]{David Savitt} \email{savitt@jhu.edu}
\address{Department of Mathematics, Johns Hopkins University, 3400 N.\ Charles St., Baltimore, MD 21210, USA}

\begin{abstract}

 In the article \cite{cegsB} we introduced various  
 moduli stacks of two-dimensional tamely potentially Barsotti--Tate representations of the
 absolute Galois group of a $p$-adic local field, as well as related moduli stacks of Breuil--Kisin modules with descent data.
 We study the irreducible components of these stacks, establishing in particular that the components of the former 
 are naturally indexed by certain Serre weights.
\end{abstract}
\maketitle

\selectlanguage{english}

\setcounter{tocdepth}{1}
\tableofcontents

\section{Introduction} 

Fix a prime number $p$, and let $K/\Qp$ be a finite extension with residue field
$k$ and absolute Galois group $G_K := \Gal(\overline{K}/K)$. In the paper \cite{cegsB}, inspired by a construction of Kisin~\cite{kis04} in the
setting of formal deformations, we constructed and began to study the geometry of certain moduli stacks $\cZ^{\dd}$. The stacks $\cZ^{\dd}$ can be thought of as moduli of two-dimensional tamely potentially Barsotti--Tate representations of $G_K$;\ they are
in fact moduli stacks of \'etale $\varphi$-modules with descent data, and by construction are 
 equipped with a partial resolution 
\[ \cC^{\dd,\BT} \to \cZ^{\dd} \]
where $\cC^{\dd,\BT}$ is a moduli stack of rank two Breuil--Kisin modules with  tame descent data and height one.

The purpose of this paper is to make an explicit study of the morphism $ \cC^{\dd,\BT} \to \cZ^{\dd}$  at the level of irreducible components. To be precise, for each two-dimensional tame inertial type $\tau$ there are closed substacks $\cC^{\tau,\BT} \subset \cC^{\dd,\BT}$ and $\cZ^{\tau} \subset \cZ^{\dd}$ corresponding to representations having inertial type $\tau$, and a morphism $\cC^{\tau,\BT} \to \cZ^{\tau}$. These are $p$-adic formal algebraic stacks;\ let $\cC^{\tau,\BT,1}$ be the special fibre of $\cC^{\tau,\BT}$, and $\cZ^{\tau,1}$ its scheme-theoretic image in $\cZ^{\tau}$ (in the sense of \cite{EGstacktheoreticimages}). These were proved in \cite{cegsB} to be equidimensional of dimension $[K:\Qp]$. Moreover the finite type points $\Spec(\F) \to \cZ^{\tau,1}$ are in bijection with Galois representations $G_K \to \GL_2(\F)$ admitting a potentially Barsotti--Tate lift of type $\tau$.

(In fact $\cC^{\tau,\BT,1}$ is shown in \cite{cegsB} to be reduced, from which it follows that $\cZ^{\tau,1}$ is also reduced. The special fibre of $\cZ^{\tau}$ need not be reduced, so it need not equal $\cZ^{\tau,1}$, but it will be proved in the sequel \cite{cegsA} that it is \emph{generically reduced}, using the results of this paper as input.)

Much of the work in our study of the irreducible components of $\cZ^{\tau,1}$
involves an explicit construction of families of extensions of
characters.  Intuitively, a natural source of ``families'' of representations
$\rbar:G_K \to\GL_2(\Fpbar)$ is given by the
extensions of two fixed characters. Indeed, given two
characters~$\chi_1,\chi_2: G_K \to\Fpbartimes$, the
$\Fpbar$-vector space $\Ext^1_{G_K}(\chi_2,\chi_1)$
is usually $[K:\Qp]$-dimensional, and a back of the envelope
calculation 
suggests that as a stack the collection of these representations should have dimension $[K:\Qp]-2$:\ the difference between an extension and a
representation counts for a $-1$, as does the~$\Gm$ of endomorphisms. Twisting~$\chi_1,\chi_2$ independently
by unramified characters gives a candidate for a $[K:\Qp]$-dimensional
family;\ if contained in $\cZ^{\tau}$, then since~$\cZ^{\tau}$ is equidimensional of dimension~$[K:\Qp]$, the
closure of such a family should be an irreducible component of~$\cZ^{\tau}$.

Since there are only finitely many possibilities for the restrictions
of the~$\chi_i$ to the inertia subgroup~$I_K$, this gives a finite list of
maximal-dimensional families. On the other hand, there are up to
unramified twist only finitely many irreducible two-dimensional
representations of~$G_K$, which suggests that the
irreducible representations should correspond to $0$-dimensional
substacks. Together these considerations suggest that the irreducible
components of our moduli stack should be given by the closures of the
families of extensions considered in the previous paragraph, and in
particular that the irreducible representations should arise as limits
of reducible representations. This could not literally be the case
for families of Galois representations, rather than families of
\'etale $\varphi$-modules, and may seem surprising at first glance,
but it is indeed what happens.

In the body of the paper we make this analysis 
rigorous, and we show
that the different families that we have constructed exhaust the
irreducible components. 
We can therefore label the irreducible components of~$\cZ^{\tau,1}$
as follows. A component is specified by an ordered pair of characters~$I_K \to\Fpbar^\times$, which via local class field theory corresponds to a pair of characters~$k^\times\to\Fpbartimes$.  Such a pair can be thought of as the  highest weight of a \emph{Serre
  weight}:\ an irreducible $\Fpbar$-representation of $\GL_2(k)$. To each irreducible component we have thus  associated a Serre weight. (In fact, we need to make a
shift in this dictionary, corresponding to half the sum of the
positive roots of~$\GL_2(k)$, but we ignore this for the purposes of
this introduction.)

This might seem artificial, but in fact it is completely natural, for
the following reason. Following the pioneering work of
Serre~\cite{MR885783} and Buzzard--Diamond--Jarvis~\cite{bdj} (as
extended in~\cite{MR2430440} and~\cite{gee061}), we now know how to
associate a set $W(\rbar)$ of Serre weights to each continuous
representation $\rbar:G_K\to\GL_2(\Fpbar)$, with the property
that if $F$ is a totally real field and $\rhobar:G_F\to\GL_2(\Fpbar)$
is an irreducible representation coming from a Hilbert modular form,
then the possible weights of Hilbert modular forms giving rise
to~$\rhobar$ are precisely determined by the sets
$W(\rhobar|_{G_{F_v}})$ for places $v|p$ of~$F$ (see for example
\cite{blggu2,geekisin,gls13}).

Going back to our labelling of irreducible components
above, we have associated a Serre weight~$\sigmabar$ to each
irreducible component of~$\cZ^{\tau,1}$.  The inertial local Langlands
correspondence assigns a finite set of Serre weights
$\JH(\sigmabar(\tau))$ to~$\tau$, the Jordan--H\"older factors of the
reduction mod~$p$ of the representation~$\sigma(\tau)$ of~$\GL_2(\cO_K)$
corresponding to~$\tau$.  One of our main theorems is that the components of $\cZ^{\tau,1}$ are labeled precisely by the Serre weights $\sigmabar \in \JH(\sigmabar(\tau))$. Furthermore the component labeled by $\sigmabar$ has a dense set of finite type points $\rbar$ with $\sigmabar \in W(\rbar)$. In the sequel \cite{cegsA} this will be strengthened to the statement that the representations~$\rbar$ on the
irreducible component labelled by~$\sigmabar$ are precisely the
representations with $\sigmabar\in W(\rbar)$.

We also study the irreducible components of the stack~$\cC^{\tau,\BT,1}$. If $\tau$ is a non-scalar principal series type then the set~$\JH(\sigmabar(\tau))$ can be identified with a subset of the
power set~$\cS$ of the set of embeddings~$k\into\Fpbar$ (hence, after fixing one such embedding, with a subset $\cP_{\tau}$ of $\Z/f\Z$). For generic
choices of~$\tau$, this subset is the whole of~$\cS$. 
We are able to show, using the theory of Dieudonn\'e modules, that for
any non-scalar principal series type~$\tau$ the irreducible components of~$\cC^{\tau,\BT,1}$
can be identified with~$\cS$, and those irreducible components not
corresponding to elements of~$\JH(\sigmabar(\tau))$ have image
in~$\cZ^\tau$ of positive codimension. There is an analogous statement for cuspidal types, while for scalar types,
both~$\cC^{\tau,\BT,1}$ and~$\cZ^{\tau,1}$ are irreducible.

To state our main results precisely we must first introduce a bit more notation. Fix a tame inertial type $\tau$ and a uniformiser $\pi$ of $K$. Let $L$ be the unramified quadratic extension of $K$, and write $f$ for the inertial degree of $K/\Qp$.  We set $K' = K(\pi^{1/p^f-1})$ if  $\tau$ is principal series, and set $K' = L(\pi^{1/(p^{2f}-1)})$ if $\tau$ is cuspidal. Our moduli stacks of $p$-adic Hodge theoretic objects with descent data will have descent data from $K'$ to $K$. Let $f'$ be the inertial degree of $K'/\Qp$, so that $f' = f$ if the type $\tau$ is principal series, while $f' = 2f$ if the type $\tau$ is cuspidal. 

 We say that a subset $J \subset \Z/f'\Z$ is a \emph{profile} if:
\begin{itemize}
  \item $\tau$ is scalar and $J = \varnothing$, 
  \item $\tau$ is a non-scalar principal series type and $J$ is arbitrary, or
  \item $\tau$ is cuspidal and $J$ has the property that $i \in J$ if and only if $i+f \not\in J$.
\end{itemize}
 If $\tau$ is non-scalar then there are exactly $2^f$ profiles. 

 As above, write $\sigma(\tau)$ for the representation of $\GL_2(\cO_K)$ corresponding to $\tau$ under the inertial local Langlands correspondence of Henniart. The Jordan--H\"older factors of the reduction mod $p$ of $\sigma(\tau)$ are parameterized by an explicit set of profiles $\cP_{\tau}$, and we write $\sigmabar(\tau)_J$ for the factor corresponding to $J$. 

 To each profile $J$, we will associate a closed substack $\overline{\mathcal{C}}(J)$ of $\cC^{\tau,\BT,1}$. The stack $\overline{\cZ}(J)$ is then defined to be the scheme-theoretic image of $\overline{\cC}(J)$ under the  map $\cC^{\tau,\BT,1} \to \cZ^{\tau,1}$, in the sense of \cite{EGstacktheoreticimages}. Then the following is our main result, combining Proposition~\ref{prop:C to Z mono}, Theorem~\ref{thm: unique serre weight}, Corollary~\ref{cor: components of Z are exactly the Z(J)}, and 
Theorem~\ref{thm: components of C}.

\begin{ithm}\label{thm:main thm cegsC} The irreducible components of $\cC^{\tau,\BT,1}$ and $\cZ^{\tau,1}$ are as follows. 
\begin{enumerate} 
  \item  The irreducible
    components of~$\cC^{\tau,1}$ are precisely the~$\overline{\cC}(J)$
    for profiles~$J$, and if $J\ne J'$ then~$\overline{\cC}(J)\ne\overline{\cC}(J')$.

 \item The irreducible
    components of~$\cZ^{\tau,1}$ are precisely the~$\overline{\cZ}(J)$
    for profiles~$J\in\cP_\tau$, and if $J\ne J'$ then~$\overline{\cZ}(J)\ne\overline{\cZ}(J')$.

  \item For each $J \in \cP_{\tau}$, there is a dense open substack $\cU$ of 
$\overline{\cC}(J)$ such that the  map $\overline{\cC}(J) 
\to \overline{\cZ}(J)$ restricts to an open immersion on $\cU$.

 \item 
  For each $J\in\cP_\tau$, there is a dense set of finite type
  points of $\overline{\cZ}(J)$ with the property that the corresponding Galois
  representations have $\sigmabar(\tau)_J$ as a Serre weight, and which
  furthermore admit a unique Breuil--Kisin model of type~$\tau$.

\end{enumerate}
\end{ithm}

\begin{iremark}\label{rem:phantom-weights}
We emphasize in Theorem~\ref{thm:main thm cegsC} that the components of $\cZ^{\tau,1}$ are indexed by profiles $J \in \cP_{\tau}$, \emph{not} by all profiles. If $J \not\in \cP_{\tau}$, then  the stack $\overline{\cZ}(J)$ has dimension strictly smaller than $[K:\Qp]$, and so is properly contained in some component of $\cZ^{\tau,1}$. We anticipate that the loci $\ocZ(J)$ will nevertheless be of interest when $J \not\in \cP_\tau$:\ we expect that they will correspond to ``phantom'' (partial weight one) Serre weights of relevance to the geometric variant of the weight part of Serre's conjecture proposed by Diamond--Sasaki \cite{DiamondSasaki}. This will be the subject of future work.
\end{iremark}

We assume that~$p>2$ in much of the paper; while we expect that our
results should also hold if~$p=2$, there are several reasons to
exclude this case. We are frequently able to considerably simplify our
arguments by assuming that the extension~$K'/K$ is not just tamely
ramified, but in fact of degree prime to~$p$; this is problematic
when~$p=2$, as the consideration of cuspidal types involves a
quadratic unramified extension. Furthermore, in the sequel \cite{cegsA} we will use results on the
Breuil--M\'ezard conjecture which ultimately depend on automorphy
lifting theorems that are not available in the case $p=2$ at present
(although it is plausible that the methods of~\cite{Thornep=2} could
be used to prove them). 

We conclude this introduction by discussing the relationship between our results and those of \cite{EGmoduli}. Two of us (M.E. and T.G.) have constructed moduli stacks $\cX_{d}$ of rank~$d$ \'etale $(\varphi,\Gamma)$-modules for $K$, as well as substacks $\cX_d^{\lambda,\tau}$ which may be regarded as stacks of potentially crystalline representations of $G_K$ with inertial type $\tau$ and Hodge type~$\lambda$. When $d=2$ and $\lambda$ is the trivial Hodge type, these are stacks $\cX_2^{\tau,\BT}$ of potentially Barsotti--Tate representations of $G_K$ of inertial type $\tau$, and we anticipate that $\cX_2^{\tau,\BT}$ is isomorphic to $\cZ^{\tau,\BT}$ (but since we do not need this fact, we have not proved it).

One of the main results of the book \cite{EGmoduli} is that the irreducible components of the underlying reduced stacks $\cX_{d,\red}$ are in bijection with the irreducible representations of $\GL_d(k)$. This bijection is characterised in essentially exactly the same way as our description of the components of $\cZ^{\tau,1}$ in this paper:\ a Serre weight has a highest weight, which corresponds to a tuple of inertial characters, which gives rise to a family of successive extensions of $1$-dimensionals representations. Then the closure of this family is a component of $\cX_{d,\red}$.

The crucial difference between our setting and that of \cite{EGmoduli}
is that we could prove in \cite{cegsB} that the stacks $\cZ^{\tau,1}$
are reduced. 
  The proof makes use of the resolution $\cC^{\tau,\BT,1} \to \cZ^{\tau,1}$ and the fact that we are able to relate the stack $\cC^{\tau,\BT}$ to a local model at Iwahori level, whose special fibre is known to be reduced. In the sequel \cite{cegsA} we combine the characterisation of the components of $\cZ^{\tau,1}$ from this paper with the reducedness of $\cZ^{\tau,1}$ from \cite{cegsB} to prove that the special fibre of $\cZ^{\tau}$ is \emph{generically} reduced. This will then allow us to  completely characterise \emph{all} of the finite type points on each component of $\cZ^{\tau,1}$ (not just a dense set of points), and to prove geometrisations of the Breuil--M\'ezard conjecture and of the weight part of Serre's conjecture for the stacks $\cZ^{\dd,1}$. Furthermore, by means of a comparison of versal rings, these results can be transported to the stacks $\cX^{\tau,\BT}_2$ of \cite{EGmoduli} as well.

\subsection{Acknowledgements}We would like to thank 
Kalyani Kansal for helpful comments.

\subsection{Notation and conventions}\label{subsec: notation}
\subsubsection*{Topological groups} If~$M$ is an abelian
topological group with a linear topology, then as
in~\cite[\href{https://stacks.math.columbia.edu/tag/07E7}{Tag
  07E7}]{stacks-project} we say that~$M$ is {\em complete} if the
natural morphism $M\to \varinjlim_i M/U_i$ is an isomorphism,
where~$\{U_i\}_{i \in I}$ is some (equivalently any) fundamental
system of neighbourhoods of~$0$ consisting of subgroups. Note that in
some other references this would be referred to as being~{\em complete
  and separated}. In particular, any $p$-adically complete ring~$A$ is
by definition $p$-adically separated.

\subsubsection*{Galois theory and local class field theory} If $M$ is a field, we let $G_M$ denote its
absolute Galois group.
If~$M$ is a global field and $v$ is a place of $M$, let $M_v$ denote
the completion of $M$ at $v$. If~$M$ is a local field, we write~$I_M$
for the inertia subgroup of~$G_M$. 

 Let $p$ be a prime number. 
 Fix a finite extension $K/\Qp$, with
 ring of integers $\cO_K$ and residue field $k$.  Let $e$ and $f$
 be the  ramification and inertial degrees of $K$, respectively, and
 write $\# k=p^f$ for the cardinality of~$k$.   
Let $K'/K$ be a finite
tamely ramified Galois extension. Let $k'$ be the residue field of $K'$, and let $e',f'$ be the
ramification and inertial degrees of $K'$ respectively.

Our representations of $G_K$ will have coefficients in $\Qpbar$,
a fixed algebraic closure of $\Qp$ whose residue field we denote by~$\Fpbar$. Let $E$ be a finite
extension of $\Qp$ contained in $\Qpbar$ and containing the image of every
embedding of $K'$ into $\Qpbar$. Let $\cO$ be the ring of integers in
$E$, with uniformiser $\varpi$ and residue field $\F \subset
\Fpbar$.  

Fix an embedding $\sigma_0:k'\into\F$, and recursively define
$\sigma_i:k'\into\F$ for all $i\in\Z$ so that
$\sigma_{i+1}^p=\sigma_i$; of course, we have $\sigma_{i+f'}=\sigma_i$
for all~$i$. We let $e_i\in k'\otimes_{\Fp} \F$ denote the idempotent
satisfying $(x\otimes 1)e_i=(1\otimes\sigma_i(x))e_i$ for all $x\in
k'$; note that $\varphi(e_i)=e_{i+1}$. We also denote by $e_i$ the
natural lift of $e_i$ to an idempotent in
$W(k')\otimes_{\Zp}\cO$. If $M$ is an
$W(k')\otimes_{\Zp}\cO$-module, then we write $M_i$ for
$e_iM$.

We write $\Art_K \col K^\times\to W_K^{\ab}$ for
the isomorphism of local class field theory, normalised so that
uniformisers correspond to geometric Frobenius elements. 

\begin{lemma}\label{lem:cft} Let  $\pi$ be any uniformiser
of $\cO_K$. 
The composite $I_K \to \cO_K^{\times} \to k^{\times}$, where  the map
$I_K \to \cO_K^\times$ is induced by the restriction of $\Art_K^{-1}$,
sends an element $g \in I_K$ to the image in $k^{\times}$ of
$g(\pi^{1/(p^f-1)})/\pi^{1/(p^f-1)}$. 
\end{lemma}

\begin{proof}
This follows (for example) from the construction in \cite[Prop.~4.4(iii), Prop.~4.7(ii), Cor.~4.9, Def.~4.10]{MR2487860}.
\end{proof}

 For each $\sigma\in \Hom(k,\Fpbar)$ we
define the fundamental character $\omega_{\sigma}$   to~$\sigma$ to be the composite \[\xymatrix{I_K \ar[r] & \cO_{K}^{\times}\ar[r] & k^{\times}\ar[r]^{\sigma} & \Fpbarx,}\]
where the map $I_K \to \cO_K^\times$ is induced by the restriction of $\Art_K^{-1}$.
Let $\varepsilon$ denote the $p$-adic cyclotomic
character and $\varepsilonbar$ the mod~$p$ cyclotomic
character, so that $\prod_{\sigma \in \Hom(k,\Fpbar)}
\omega_{\sigma}^{e} = \varepsilonbar$.   We will often identify
characters of  $I_K$ with characters of $k^{\times}$ via the Artin
map.

\subsubsection*{Inertial local Langlands} A two-dimensional \emph{tame inertial type} is (the isomorphism
class of) a tamely ramified representation
$\tau : I_K \to \GL_2(\Zpbar)$ that extends to a representation of $G_K$ and
whose kernel is open. Such a representation is of the form $\tau
\simeq \eta \oplus \eta'$, and we say that $\tau$ is a \emph{tame principal series
  type} if 
$\eta,\eta'$ both extend to characters of $G_K$. Otherwise,
$\eta'=\eta^q$, and $\eta$ extends to a character
of~$G_L$, where~$L/K$ is a quadratic unramified extension. 
In this case we say
that~$\tau$ is a \emph{tame cuspidal type}.

Henniart's appendix to \cite{breuil-mezard}
associates a finite dimensional irreducible $E$-representation $\sigma(\tau)$ of
$\GL_2(\cO_K)$ to each inertial type $\tau$; we refer to this association as the {\em
  inertial local Langlands correspondence}. Since we are only working
with tame inertial types, this correspondence can be made very
explicit as follows. 

If $\tau
\simeq \eta \oplus \eta'$ is a tame principal series type, then we
also write $\eta,\eta':k^\times\to\cO^\times$ for the 
multiplicative characters determined by
$\eta\circ\Art_K|_{\cO_{K}^\times},\eta'\circ\Art_K|_{\cO_{K}^\times}$
respectively. If $\eta=\eta'$, then we set
$\sigma(\tau)=\eta\circ\det$. Otherwise, we write $I$ for the Iwahori subgroup of $\GL_2(\cO_K)$ consisting of
matrices which are upper triangular modulo a uniformiser~$\varpi_K$
of~$K$, and write $\chi = \eta'\otimes \eta:
I\to\cO^\times$ for the character \[
\begin{pmatrix}
  a&b\\\varpi_K c&d
\end{pmatrix}\mapsto \eta'(\overline{a})\eta(\overline{d}).\] Then $\sigma(\tau) := \Ind_I^{\GL_2(\cO_K)}
\chi$. 

If $\tau=\eta\oplus\eta^q$ is a tame cuspidal type, then as above we
write~$L/K$ for a quadratic unramified extension, and~$l$ for the
residue field of~$\cO_L$. We write
$\eta :l^\times\to\cO^\times$ for the 
multiplicative character determined by
$\eta\circ\Art_L|_{\cO_{L}^\times}$; then $\sigma(\tau)$ is the
inflation to $\GL_2(\cO_K)$ of the cuspidal representation of $\GL_2(k)$
denoted by~$\Theta(\eta)$ in~\cite{MR2392355}.

\subsubsection*{$p$-adic Hodge theory} We normalise Hodge--Tate weights so that all Hodge--Tate weights of
the cyclotomic character are equal to $-1$. We say that a potentially
crystalline representation $\rho:G_K\to\GL_2(\Qpbar)$ has \emph{Hodge
  type} $0$, or is \emph{potentially Barsotti--Tate}, if for each
$\varsigma :K\into \Qpbar$, the Hodge--Tate weights of $\rho$ with
respect to $\varsigma$ are $0$ and $1$.  (Note that this is a more
restrictive definition of potentially Barsotti--Tate than is sometimes
used; however, we will have no reason to deal with representations
with non-regular Hodge-Tate weights, and so we exclude them from
consideration. Note also that it is more usual in the literature to
say that $\rho$ is potentially Barsotti--Tate if it is potentially
crystalline, and $\rho^\vee$ has Hodge type $0$.) 

We say
that a potentially crystalline representation
$\rho:G_K\to\GL_2(\Qpbar)$ has  \emph{inertial type} $\tau$ if the traces of
elements of $I_K$ acting on~$\tau$ and on
\[\Dpcris(\rho)=\varinjlim_{K'/K}(\Bcris\otimes_{\Qp}V_\rho)^{G_{K'}}\] are
equal (here~$V_\rho$ is the underlying vector space
of~$V_\rho$). 
A representation $\rbar:G_K\to\GL_2(\Fpbar)$ \emph{has a potentially
    Barsotti--Tate lift of
    type~$\tau$} if and
  only if $\rbar$ admits a lift to a representation
  $r:G_K\to\GL_2(\Zpbar)$ of Hodge type~$0$ and inertial type~$\tau$.

\subsubsection*{Serre weights}
By definition, a \emph{Serre weight} is an irreducible
$\F$-representation of $\GL_2(k)$. Concretely, such a
representation is of the form
\numequation\label{eq:serreweight} \sigmabar_{\vec{t},\vec{s}}:=\otimes^{f-1}_{j=0}
(\det{\!}^{t_j}\Sym^{s_j}k^2) \otimes_{k,\sigma_{j}} \F,\end{equation}
where $0\le s_j,t_j\le p-1$ and not all $t_j$ are equal to
$p-1$. We say that a Serre weight is \emph{Steinberg} if $s_j=p-1$ for all $j$,
and \emph{non-Steinberg} otherwise.

\subsubsection*{A remark on normalisations}  Given a continuous representation $\rbar:G_K\to\GL_2(\Fpbar)$, there
is an associated (nonempty) set of Serre weights~$W(\rbar)$ whose
precise definition we will recall in Appendix~\ref{sec: appendix on tame types}. There are in fact
several different definitions of~$W(\rbar)$ in the literature; as a
result of the papers~\cite{blggu2,geekisin,gls13}, these definitions
are known to be equivalent up to normalisation. 

However, the normalisations of
Hodge--Tate weights and of inertial local Langlands used in
\cite{geekisin,gls13,emertongeesavitt} are not all the same, and so
for clarity we lay out how they differ, and how they compare to the normalisations of
this paper. 

Our conventions for Hodge--Tate weights and
inertial types agree with those of~\cite{geekisin, emertongeesavitt}, but our
representation~$\sigma(\tau)$ is the
representation~$\sigma(\tau^\vee)$ of~\cite{geekisin, emertongeesavitt}
(where~$\tau^\vee=\eta^{-1}\oplus(\eta')^{-1}$);\ to see this, note the
dual in the definition of~$\sigma(\tau)$ in~\cite[Thm.\
2.1.3]{geekisin} and the discussion in \S 1.9 of
\cite{emertongeesavitt}.\footnote{However, this dual is erroneously
  omitted when the inertial local Langlands correspondence is made
  explicit  at the end of  \cite[\S3.1]{emertongeesavitt}. See
  Remark~\ref{arem: wtf were we thinking in EGS}.}

In all cases one chooses to normalise the set of Serre weights so
that the condition of Lemma~\ref{lem: list of things we need to know about Serre
  weights}(1) holds.  Consequently, our set of weights~$W(\rbar)$ is the
set of duals of the weights~$W(\rbar)$ considered
in~\cite{geekisin}. In turn, the paper~\cite{gls13} has the opposite
convention for the signs of Hodge--Tate weights to our convention (and
to the convention of~\cite{geekisin}), so we find that our set of
weights~$W(\rbar)$ is the set of duals of the weights~$W(\rbar^\vee)$
considered in~\cite{gls13}.

\subsubsection*{Stacks}We follow the terminology of~\cite{stacks-project}; in
particular, we write ``algebraic stack'' rather than ``Artin stack''. More
precisely, an algebraic stack is a stack in groupoids in the \emph{fppf} topology,
whose diagonal is representable by algebraic spaces, which admits a smooth
surjection from a
scheme. See~\cite[\href{http://stacks.math.columbia.edu/tag/026N}{Tag
  026N}]{stacks-project} for a discussion of how this definition relates to
others in the literature, and~\cite[\href{http://stacks.math.columbia.edu/tag/04XB}{Tag
  04XB}]{stacks-project} for key properties of morphisms
representable by algebraic spaces.

For a commutative ring $A$, an \emph{fppf stack over $A$} (or
\emph{fppf} $A$-stack) is a stack fibred in groupoids over the big \emph{fppf}
site of $\Spec A$.

\section{Preliminaries}

We begin by reviewing the various constructions and results that we will need from \cite{cegsB}. Section~\ref{subsec: kisin modules with dd} recalls the definition and a few basic algebraic properties of Breuil--Kisin modules with coefficients and descent data, while Section~\ref{subsec: etale phi modules
  and Galois representations} does the same for \'etale $\varphi$-modules. In Section~\ref{sec:recoll-from-citec} we define the stacks $\cC^{\tau,\BT,1}$ and $\cZ^{\tau,1}$ (as well as various other related stacks) and state the main results of \cite{cegsB}. Finally, in Section~\ref{sec:dieudonne-stacks} we introduce and study stacks of Dieudonn\'e modules that will be used at the end of the paper to determine the irreducible components of $\cC^{\tau,\BT,1}$.

\subsection{Breuil--Kisin modules 
  with descent data}\label{subsec: kisin modules with dd} 
Recall that we have a finite
tamely ramified Galois extension $K'/K$. Suppose further that there exists a uniformiser $\pi'$ of
$\cO_{K'}$ such that $\pi:=(\pi')^{e(K'/K)}$ is an element of~$K$, 
where $e(K'/K)$ is the ramification index of
$K'/K$. 
Recall that $k'$ is the residue field of $K'$, while $e',f'$ are the
ramification and inertial degrees of $K'$ respectively.
 Let $E(u)$ be the minimal polynomial of $\pi'$ over $W(k')[1/p]$. 

Let $\varphi$ denote the arithmetic Frobenius automorphism of $k'$, which lifts uniquely
to an automorphism of $W(k')$ that we also denote by $\varphi$. Define
$\gS:=W(k')[[u]]$, and extend $\varphi$ to $\gS$ by \[\varphi\left(\sum a_iu^i\right)=\sum
\varphi(a_i)u^{pi}.\] By our assumptions that $(\pi')^{e(K'/K)} \in K$ 
and that $K'/K$ is Galois, for each
$g\in\Gal(K'/K)$ we can write $g(\pi')/\pi'=h(g)$ with $h(g)\in
\mu_{e(K'/K)}(K') \subset W(k')$,
and we
let $\Gal(K'/K)$ act on $\gS$ via \[g\left(\sum a_iu^i\right)=\sum g(a_i)h(g)^iu^i.\]

Let $A$ be a $p$-adically complete $\Zp$-algebra, set $\gS_A:=(W(k')\otimes_{\Zp} A)[[u]]$, and extend
the actions of $\varphi$ and $\Gal(K'/K)$ on $\gS$ to actions on $\gS_A$ in the
obvious ($A$-linear) fashion. 

\begin{lemma}\label{lem:projectivity descends}
An $\gS_A$-module is 
projective if and only if it is projective as an 
$A[[u]]$-module. 
\end{lemma}

\begin{proof}
 Suppose that $\gM$ is an $\gS_A$-module that is projective as an
 $A[[u]]$-module. Certainly $W(k') \otimes_{\Zp} \gM$ is projective
 over $\gS_A$, and we claim that it has $\gM$ as an $\gS_A$-module direct summand. 
 Indeed, this  follows by rewriting $\gM$ as $W(k')\otimes_{W(k')}
 \gM$ and noting that $W(k')$ is a $W(k')$-module direct summand of $W(k')
 \otimes_{\Zp} W(k')$.
\end{proof}

The actions of $\varphi$ and $\Gal(K'/K)$ on $\gS_A$
extend to actions on $\gS_A[1/u]=(W(k')\otimes_{\Zp} A)((u))$ in the obvious
way.  It will sometimes be necessary to consider the subring $\gS_A^0
:=(W(k)\otimes_{\Zp} A)[[v]]$ of $\gS_A$
  consisting of power series in 
$v:=u^{e(K'/K)}$, on which  $\Gal(K'/K)$ acts
  trivially. 

\begin{defn}\label{defn: Kisin module with descent data}
Fix a $p$-adically complete $\Zp$-algebra~$A$. A \emph{Breuil--Kisin module with
  $A$-coefficients and descent data from $K'$ to $K$} (or often simply
a \emph{Breuil--Kisin module}) 
  is a triple $(\gM,\varphi_{\gM},\{\hat{g}\}_{g\in\Gal(K'/K)})$ consisting of
  a 
  $\gS_A$-module~$\gM$ and a $\varphi$-semilinear map
  $\varphi_{\gM}:\gM\to\gM$ 
   such that:
  \begin{itemize}
  \item the $\gS_A$-module $\gM$ is finitely generated and projective,
    and 
  \item the  induced
  map $\Phi_{\gM} = 1 \otimes \varphi_{\gM} :\varphi^*\gM\to\gM$ is an isomorphism after
  inverting $E(u)$  (here as usual we write $\varphi^*\gM:=\gS_A \otimes_{\varphi,\gS_A}\gM$), 
  \end{itemize}
together with
  additive bijections  $\hat{g}:\gM\to\gM$, satisfying the further
  properties that 
    the maps $\hat{g}$ commute with $\varphi_\gM$, satisfy 
    $\hat{g_1}\circ\hat{g_2}=\widehat{g_1\circ g_2}$, and have
    $\hat{g}(sm)=g(s)\hat{g}(m)$ for all $s\in\gS_A$, $m\in\gM$. 
   We say that $\gM$ is has \emph{height at most $h$} if the cokernel of 
   $\Phi_{\gM}$ is killed by $E(u)^h$. 

The Breuil--Kisin module $\gM$ is said to be  of rank~$d$ if the underlying
finitely generated projective $\gS_A$-module has constant rank~$d$. It
is said to be free if the underlying $\gS_A$-module is free.
\end{defn}

A morphism of Breuil--Kisin modules with descent data is
a  morphism 
of~$\gS_A$-modules 
that commutes with $\varphi$ and with the~$\hat{g}$. 
In the case that $K'=K$ the data of the $\hat{g}$ is trivial, so
it can be forgotten, giving the category of \emph{Breuil--Kisin modules with
  $A$-coefficients.} In this case it will sometimes be convenient to elide the difference between a
Breuil--Kisin module with trivial descent data, and a Breuil--Kisin module without
descent data, in order to avoid making separate definitions in the
case of Breuil--Kisin modules without descent data.

\begin{rem} 
\label{rem:projectivity for Kisin modules} We refer the reader
to~\cite[\S5.1]{EGstacktheoreticimages} for a
discussion of foundational results concerning finitely generated modules
over the power series ring $A[[u]]$. In particular (using
Lemma~\ref{lem:projectivity descends}) we note the
following.
\begin{enumerate}
\item An $\gS_A$-module $\gM$ is finitely generated and projective if
  and only if it is $u$-torsion free and $u$-adically complete, and $\gM/u\gM$ is a finitely generated projective
  $A$-module (\cite[Prop.~5.1.8]{EGstacktheoreticimages}).

\item  If the $\gS_A$-module $\gM$ is projective of
rank~$d$, then it is Zariski locally free of rank~$d$ in the sense that there is a cover of $\Spec A$
by affine opens $\Spec B_i$ such that each of the base-changed modules
 $\gM\otimes_{\gS_A}\gS_{B_i}$ is free of rank $d$  (\cite[Prop.~5.1.9]{EGstacktheoreticimages}).

\item If $A$ is coherent (so in
    particular, if $A$ is Noetherian), then $A[[u]]$ is faithfully
    flat over~$A$, and so $\gS_A$ is faithfully flat over~$A$, but
    this need not hold if $A$ is not coherent.
\end{enumerate}

\end{rem}

\begin{df}
\label{def:completed tensor}
If $Q$ is any (not necessarily finitely generated) $A$-module,
and $\gM$ is an $A[[u]]$-module,
then  we let $\gM\cotimes_A Q$ denote the $u$-adic completion
of $\gM\otimes_A Q$.
\end{df} 

\begin{lem}
  \label{rem: base change of locally free Kisin module is a
    locally free Kisin module}
If $\gM$ is a Breuil--Kisin module and $B$ is an $A$-algebra, then the base
change $\gM \cotimes_A B$ is a Breuil--Kisin module.
\end{lem}

\begin{proof} This is \cegsBbasechangelemma.
\end{proof}

We make the following two further remarks concerning base change.

\begin{remark}
\label{rem:completed tensor} 
(1) If $A$ is Noetherian, if $Q$ is finitely generated over $A$,
and if $\gN$ is 
finitely generated over $A[[u]]$, then $\gN\otimes_A Q$
is finitely generated over $A[[u]]$, and hence (by the Artin--Rees
lemma) is automatically $u$-adically complete.  Thus 
in this case the natural morphism $\gN\otimes_A Q \to \gN\cotimes_A Q$
is an isomorphism.

\smallskip

(2)
Note that $A[[u]]\cotimes_A Q = Q[[u]]$ (the $A[[u]]$-module
consisting of power series with coefficients in the $A$-module $Q$),
and so if $\gN$ is Zariski locally free on $\Spec A$,
then $\gN\cotimes_A Q$ is Zariski locally isomorphic to a direct sum
of copies of $Q[[u]]$, and hence is $u$-torsion free (as well as
being $u$-adically complete). In particular, by
Remark~\ref{rem:projectivity for Kisin modules}(2), this holds if~$\gN$
is projective.
\end{remark}

Let $A$ be a $\Zp$-algebra. We define a \emph{Dieudonn\'e module of rank $d$ with $A$-coefficients and
  descent data from $K'$ to $K$} to be a finitely generated projective
$W(k')\otimes_{\Zp}A$-module $D$ of constant rank 
$d$ on $\Spec A$, together with:

\begin{itemize}
\item $A$-linear endomorphisms $F,V$ satisfying $FV = VF = p$ such that $F$ is $\varphi$-semilinear and $V$ is
  $\varphi^{-1}$-semilinear for the action of $W(k')$, and
\item a $W(k')\otimes_{\Zp}A$-semilinear action of $\Gal(K'/K)$
which commutes with $F$ and $V$. 
\end{itemize}

\begin{defn}\label{def: Dieudonne module formulas}
If $\gM$ is a Breuil--Kisin module of height at most~$1$ and rank~$d$ with descent data, 
then there is a
corresponding Dieudonn\'e module $D=D(\gM)$ of rank~$d$ defined as follows. We set
$D:=\gM/u\gM$
with the induced action of $\Gal(K'/K)$, and $F$ given by the induced
action of $\varphi$. 
The endomorphism $V$ is determined as follows.  Write $E(0) = \czero p$,
so that we have $p \equiv \czero^{-1}E(u) \pmod{u}$. The
condition that the cokernel of $\varphi^*\gM\to\gM$ is killed by $E(u)$
allows us to factor the multiplication-by-$E(u)$ map on $\gM$ uniquely
as $\mathfrak{V} \circ \varphi$,  and $V$ is
defined to be
$\czero^{-1} \mathfrak{V}$ modulo~$u$. 
\end{defn}

\subsection{\'Etale  \texorpdfstring{$\varphi$}{phi}-modules and
  Galois representations}
\label{subsec: etale phi modules
  and Galois representations}

\begin{defn}\label{defn: etale phi module} 
Let $A$ be a $\Z/p^a\Z$-algebra for some $a\ge 1$.  A \emph{weak \'etale
  $\varphi$-module} with $A$-coefficients
  and descent data from $K'$ to $K$ is a triple
  $(M,\varphi_M,\{\hat{g}\})$ consisting of: 
  \begin{itemize}
\item 
a finitely generated
  $\gS_A[1/u]$-module $M$; 
\item a  $\varphi$-semilinear map $\varphi_M:M\to M$ with the
  property that the induced
  map   \[\Phi_M = 1 \otimes \varphi_M:\varphi^*M:=\gS_A[1/u]\otimes_{\varphi,\gS_A[1/u]}M\to M\]is an
  isomorphism,
  \end{itemize}
together with   additive bijections  $\hat{g}:M\to M$ for $g\in\Gal(K'/K)$, satisfying the further
  properties that the maps $\hat{g}$ commute with $\varphi_M$, satisfy
    $\hat{g_1}\circ\hat{g_2}=\widehat{g_1\circ g_2}$, and have
    $\hat{g}(sm)=g(s)\hat{g}(m)$ for all $s\in\gS_A[1/u]$, $m\in M$.

If $M$ as above is projective as an $\gS_A[1/u]$-module then we say
simply that $M$ is an \'etale $\varphi$-module. The \'etale $\varphi$-module $M$ is said to be  of rank~$d$ if the underlying
finitely generated projective $\gS_A[1/u]$-module has constant rank~$d$.

\end{defn}

\begin{rem} 
  \label{rem: completed version if $p$ not nilpotent}We could also
  consider \'etale $\varphi$-modules for general $p$-adically complete $\Zp$-algebras~$A$, but
  we would need to replace $\gS_A[1/u]$ by its $p$-adic completion. As
  we will not need to consider these modules in this paper, we do not
  do so here, but we refer the interested reader to~\cite{EGmoduli}. 
\end{rem}
A morphism
of weak \'etale
$\varphi$-modules with $A$-coefficients and descent data from $K'$ to
$K$ 
is a morphism of~$\gS_A[1/u]$-modules 
that commutes with $\varphi$ and with the
$\hat{g}$. Again, in the case $K'=K$ the descent data is trivial, and we
obtain the  usual category of \'etale $\varphi$-modules with
$A$-coefficients. 

Note that
if $A$ is a $\Z/p^a\Z$-algebra, and $\gM$ is a  Breuil--Kisin module 
with descent data, then $\gM[1/u]$ naturally has the
structure of an \'etale $\varphi$-module 
 with descent data.

Suppose that $A$ is an $\cO$-algebra (where $\cO$ is as in
Section~\ref{subsec: notation}). In making calculations, it is often
convenient to use the idempotents~$e_i$ (again as in
Section~\ref{subsec: notation}). In particular if $\gM$ is a Breuil--Kisin
module, then writing as usual
$\gM_i:=e_i\gM$, we write $\Phi_{\gM,i}:\varphi^*(\gM_{i-1})\to\gM_{i}$ for
the morphism induced by~$\Phi_{\gM}$.  Similarly if $M$ is an
\'etale $\varphi$-module  then we write
$M_i:=e_iM$, and we write $\Phi_{M,i}:\varphi^*(M_{i-1}) \to M_{i}$ for
the morphism induced by~$\Phi_{M}$.

To connect \'etale $\varphi$-modules to
$G_{K_{\infty}}$-representations we begin by recalling 
 from \cite{kis04} some constructions arising in $p$-adic Hodge theory
and the theory of fields of norms, which go back to~\cite{MR1106901}. 
Following Fontaine,
we write $R:=\varprojlim_{x\mapsto
  x^p}\cO_{\Kbar}/p$. 
Fix a compatible system $(\! \sqrt[p^n]{\pi}\,)_{n\ge 0}$ of
$p^n$th roots of $\pi$ in $\Kbar$ (compatible in the obvious sense that 
$\bigl(\! \sqrt[p^{n+1}]{\pi}\,\bigr)^p = \sqrt[p^n]{\pi}\,$),
and let
$K_{\infty}:=\cup_{n}K(\sqrt[p^n]{\pi})$, and
also $K'_\infty:=\cup_{n}K'(\sqrt[p^n]{\pi})$. Since $(e(K'/K),p)=1$, the compatible system
$(\! \sqrt[p^n]{\pi}\,)_{n\ge 0}$ determines a unique compatible system $(\!
\sqrt[p^n]{\pi'}\,)_{n\ge 0}$ of $p^n$th roots of~$\pi'$ such that $(\!
\sqrt[p^n]{\pi'}\,)^{e(K'/K)} =\sqrt[p^n]{\pi}$.
Write
$\underline{\pi}'=(\sqrt[p^n]{\pi'})_{n\ge 0}\in R$, and $[\underline{\pi}']\in
W(R)$ for its image under the natural multiplicative map $R \to W(R)$. We have a Frobenius-equivariant
inclusion $\gS\into W(R)$ by sending $u\mapsto[\underline{\pi}']$.  We can naturally identify
$\Gal(K'_\infty/K_\infty)$ with $\Gal(K'/K)$, and doing this we see that the
action of $g\in G_{K_\infty}$ on $u$ is via $g(u)=h(g)u$.

We let $\cO_{\cE}$ denote the $p$-adic completion of $\gS[1/u]$, and let $\cE$ be the
field of fractions of~$\cO_\cE$. The inclusion $\gS\into W(R)$ extends to an
inclusion $\cE\into W(\operatorname{Frac}(R))[1/p]$. Let $\cE^{\text{nr}}$ be
the maximal unramified extension of $\cE$ in $ W(\operatorname{Frac}(R))[1/p]$,
and let $\cO_{\cE^{\text{nr}}}\subset W(\operatorname{Frac}(R))$ denote
its ring of
integers. Let $\cO_{\widehat{\cE^{\text{nr}}}}$ be the $p$-adic completion of
$\cO_{\cE^{\text{nr}}}$. Note that $\cO_{\widehat{\cE^{\text{nr}}}}$ is stable
under the action of $G_{K_\infty}$. 

 \begin{defn} 
Suppose that $A$ is a $\Z/p^a\Z$-algebra for some $a \ge 1$.  If $M$
is a weak \'etale $\varphi$-module with $A$-coefficients and descent data, set
  $T_A(M):=\left(\cO_{\widehat{\cE^{\text{nr}}}}\otimes_{\gS[1/u]}M
  \right)^{\varphi=1}$, an $A$-module with a
  $G_{K_\infty}$-action (via the diagonal action on
  $\cO_{\widehat{\cE^{\text{nr}}}}$ and $M$, the latter given by
  the~$\hat{g}$). If $\gM$ is a Breuil--Kisin module with
  $A$-coefficients and descent data,
  set 
  $T_A(\gM):=T_A(\gM[1/u])$. 
\end{defn}

\begin{lem}
  \label{lem: Galois rep is a functor if A is actually finite local} Suppose
  that $A$ is a local $\Zp$-algebra and that $|A|<\infty$. Then~$T_A$ induces an 
  equivalence of categories from the category of weak \'etale
  $\varphi$-modules with $A$-coefficients and descent data to the category of continuous
  representations of $G_{K_\infty}$ on finite $A$-modules. If $A\to A'$
  is finite, then there is a natural isomorphism $T_A(M)\otimes_A A'\iso
  T_{A'}(M\otimes_A A')$. A weak \'etale $\varphi$-module with
  $A$-coefficients and descent
  data~$M$ is free of rank~$d$ if and only if $T_A(M)$ is a free
  $A$-module of rank~$d$. 
\end{lem}
\begin{proof}
  This is due to Fontaine~\cite[\S1.2]{MR1106901}, and can be proved in exactly the same way as~\cite[Lem.\ 1.2.7]{kis04}.
\end{proof}

We will frequently simply write $T$ for $T_A$. Note that if we let
$M'$ be the \'etale $\varphi$-module obtained from $M$ by forgetting the
descent data, then by definition we have
$T(M')=T(M)|_{G_{K'_\infty}}$.

\begin{rem}\label{rem:ht-1-extend}
Although \'etale $\varphi$-modules naturally give rise to representations 
of~$G_{K_\infty}$, those coming from Breuil--Kisin modules of height at most~$1$
admit canonical extensions to~$G_K$ by~\cite[Prop.\ 1.1.13]{kis04}.   
\end{rem}

\begin{lem}
  \label{lem: restricting to K_infty doesn't lose information about
    rbar}If $\rbar,\rbar':G_K\to\GL_2(\Fpbar)$ are continuous
  representations, both of which arise as the reduction mod~$p$ of
  potentially Barsotti--Tate representations of tame inertial type,
 and there is an isomorphism $\rbar|_{G_{K_\infty}}\cong \rbar'|_{G_{K_\infty}}
  $, then $\rbar\cong\rbar'$.
\end{lem}
\begin{proof}
  The extension $K_\infty/K$ is totally wildly ramified. Since the
  irreducible $\Fpbar$-representations of~$G_K$ are induced from
  tamely ramified characters, we see that~$\rbar|_{G_{K_\infty}}$ is
  irreducible if and only if~$\rbar$ is irreducible, and if
  $\rbar$ or $\rbar'$ is irreducible then we are done. In the
  reducible case, we see that $\rbar$ and $\rbar'$ are extensions of
  the same characters, and the result then follows from~\cite[Lem.\
  5.4.2]{gls13} and Lemma~\ref{lem: list of things we need to know about Serre weights}~(2). 
\end{proof}

\subsection{Recollections from \texorpdfstring{\cite{cegsB}}{[CEGS20b]}}
\label{sec:recoll-from-citec}

The main objects of study in this paper are certain algebraic stacks
$\cC^{\tau,\BT}$  and $\cZ^{\tau,1}$,  of rank two  Breuil--Kisin modules and 
\'etale $\varphi$-modules respectively, that were introduced and studied in
\cite{cegsB}. We review their definitions  now, and recall the main
properties of these stacks that were established in \cite{cegsB}. 

To define $\cC^{\tau,\BT,1}$ we first introduce  stacks of Breuil--Kisin modules with descent
data;\ then we impose two conditions on them, corresponding (in the
analogy with Galois representations) to fixing an inertial type~$\tau$
and requiring all pairs Hodge--Tate weights to be
$\{0,1\}$.

Take $K'/K$ to be any Galois extension such that $[K':K]$ is prime to
  $p$. 

\begin{defn}
  \label{defn: C^dd,a }For each integer $a\ge 1$, we let $\cC_{d,h}^{\dd,a}$ be
  the {\em fppf} stack over~$\cO/\varpi^a$ which associates to any $\cO/\varpi^a$-algebra $A$
the groupoid $\cC_{d,h}^{\dd,a}(A)$ 
of rank~$d$ Breuil--Kisin modules of height at most~$h$ with $A$-coefficients and descent data from
$K'$ to~$K$. 

By~\cite[\href{http://stacks.math.columbia.edu/tag/04WV}{Tag 04WV}]{stacks-project}, 
we may also regard each of the stacks $\cC_{d,h}^{\dd,a}$ as an {\em fppf}
stack over $\cO$,
and we then write $\cC_{d,h}^{\dd}:=\varinjlim_{a}\cC_{d,h}^{\dd,a}$; this
is again an {\em fppf} stack over~$\cO$.  We will omit the subscripts $d,h$ from this notation 
when doing so will not cause confusion. 
\end{defn} 

\begin{defn}
 Let $\tau$  be a $d$-dimensional $E$-representation of $I(K'/K)$. We say that an
 object  $\gM$ of $\cC^{\dd,a}$ is has \emph{type} $\tau$ if Zariski locally on $\Spec A$ there is an
$I(K'/K)$-equivariant isomorphism $\gM_i/u \gM_i \cong A \otimes_{\cO}
\tau^\circ$ for each $i$. (Here we recall that $\gM_i := e_i \gM$, and
$\tau^{\circ}$ denotes an $\cO$-lattice in $\tau$.)
\end{defn}

\begin{defn}
  Let $\cC^{\tau}$ be the \'etale substack of $\cC^{\dd}$
  consisting of the objects of type  $\tau$. This is an open and
  closed substack of $\cC^{\dd}$ (see \cegsBtauprop).
\end{defn}

For the remainder of this section we fix $d=2$ and $h=1$,
and write $N = K \cdot W(k')[1/p]$ (the maximal subextension of $K'$
which is unramified over~$K$). 

  Suppose that~$A$ is an $\cO/\varpi^a$-algebra and consider a pair~$(\gL,\gL^+)$, where:
  \begin{itemize}
  \item $\gL$ is a rank $2$ projective $\cO_{K'}\otimes_{\Zp} A$-module,
    with a $\Gal(K'/K)$-semilinear, $A$-linear action of~$\Gal(K'/K)$;
  \item $\gL^+$ is an $\cO_{K'}\otimes_{\Zp} A$-submodule of~$\gL$, which is
    locally on~$\Spec A$ a direct summand of~$\gL$ as an $A$-module
    (or equivalently, for which $\gL/\gL^+$ is  projective as an $A$-module),
    and is preserved by~$\Gal(K'/K)$.
  \end{itemize}
For each character $\xi : I(K'/K) \to \cO^{\times}$, let $\gL_{\xi}$
(resp.\ $\gL^+_{\xi}$) be the $\cO_N\otimes_{\Zp} A$-submodule of $\gL$ (resp.\ $\gL^+$) on
which $I(K'/K)$ acts through $\xi$. 
We say that the pair $(\gL,\gL^+)$ \emph{satisfies the strong determinant condition} if Zariski
  locally on $\Spec A$ the
  following condition holds: for all $\alpha\in\cO_N$ and all $\xi$, we
  have \numequation\label{eqn: strong det condn}
  \det{\!}_A(\alpha{}|\gL^+_\xi)
  =\prod_{\psi:N\into E}\psi(\alpha{}) 
  \end{equation}
  as polynomial functions on $\cO_N$ 
  in the sense of~\cite[\S5]{MR1124982}.

  \begin{defn}\label{def:strongdet}
  An object $\gM$ of  $\cC^{\dd,a}$  \emph{satisfies the strong
    determinant condition} if the pair $(\gM/E(u)\gM, \im
  \Phi_{\gM}/E(u)\gM)$ satisfies the strong determinant condition as
  in the previous paragraph.

We define $\cC^{\tau,\BT}$ to be the substack of $\cC^{\tau}$ of
objects satisfying the strong determinant condition. This is a
$\varpi$-adic formal algebraic stack of finite presentation
over~$\cO$ by 
\cegsBctaubt, and so its special fibre $\cC^{\tau,\BT,1}$ is an
algebraic stack, locally of finite type over $\F$.

The motivation for imposing the strong determinant condition is that
the corresponding substack $\cC^{\dd,\BT}$ of $\cC^{\dd}$   is precisely the $\cO$-flat
substack of $\cC^{\dd}$ whose $\Spf(\cO_{E'})$-points, for any finite extension $E'/E$, correspond to
potentially Barsotti--Tate Galois representations $G_K
\to \GL_2(\cO_{E'})$; so $\cC^{\tau,\BT}$ corresponds to those
representations which are furthermore of inertial type~$\tau$. (See \cegsBpotBTreps, as well as the discussion at \cegsBpotBTrepsremark.)

\end{defn}

The following result combines \cegsBCmainresults.

\begin{thm} \label{cor: Kisin moduli consequences of local models}  We have:
  \begin{enumerate}
  \item $\cC^{\tau,\BT}$ is analytically normal, and Cohen--Macaulay.
  \item The special fibre $\cC^{\tau,\BT,1}$ is reduced and
    equidimensional of dimension equal to $[K:\Qp]$.
  \item  $\cC^{\tau,\BT}$ is flat over~$\cO$.
  \end{enumerate}
  \end{thm}

We now introduce our  stacks of \'etale $\varphi$-modules.

\begin{defn}\label{defn: R^dd}
Let 
 $\cR^{\dd,1}$ be the \emph{fppf} $\F$-stack which
  associates
  to any $\F$-algebra $A$ the groupoid $\cR^{\dd,1}(A)$ of rank $2$ \'etale
  $\varphi$-modules  with $A$-coefficients and  descent data from $K'$ to
  $K$.

\end{defn}

Inverting $u$  gives a proper morphism $\cC^{\dd,1} \to \cR^{\dd,1}$, which
then restricts to a proper morphism $\cC^{\tau,\BT,1} \to \cR^{\dd,1}$ for each
$\tau$.

We now briefly remind the reader of some 
definitions from~\cite[\S3.2]{EGstacktheoreticimages}. 
Let
 $\cX \to \cF$ be a proper morphism of stacks over a locally
Noetherian base-scheme~$S$, 
where $\cX$ is an algebraic stack which is locally of finite presentation over~$S$,
and the diagonal of $\cF$ is representable by algebraic spaces and locally of
finite presentation.

We refer to~\cite[Defn.\ 3.2.8]{EGstacktheoreticimages} for the
definition of the \emph{scheme-theoretic image}~$\cZ$ of the proper morphism $\cX \to
\cF$. By definition, it is a full subcategory in groupoids of~$\cF$, and in fact
by~\cite[Lem.\ 3.2.9]{EGstacktheoreticimages} it is a Zariski substack
of~$\cF$. By~\cite[Lem.\ 3.2.14]{EGstacktheoreticimages}, the finite type points
of~$\cZ$ are precisely the finite type points of~$\cF$ for which the
corresponding fibre of~$\cX$ is nonzero. 

 The results of~\cite[\S3.2]{EGstacktheoreticimages} give criteria
for~$\cZ$ to be an algebraic stack, and prove a number of associated results
(such as universal properties of the morphism $\cZ\to\cF$, and a description of
versal deformation rings for~$\cZ$). This formalism applies in
particular to the proper morphism $\cC^{\tau,\BT,1} \to \cR^{\dd,1}$,
and so we make the following definition.

\begin{defn}
 We define $\cZ^{\tau,1}$ to be the scheme-theoretic image  (in the
 sense of~\cite[Defn.\ 3.2.8]{EGstacktheoreticimages}) of the morphism
 $\cC^{\tau,\BT,1}\to\cR^{\dd,1}$.
\end{defn}

In \cegsBZmainresults\ we established the following properties of this
construction.

\begin{prop}\label{prop:Zproperties} \leavevmode
  \begin{enumerate}
  \item $\cZ^{\tau,1}$ is an algebraic stack of finite presentation
    over $\F$, and is a closed substack of $\cR^{\dd,1}$.
  \item The morphism  $\cC^{\tau,\BT,1}\to\cR^{\dd,1}$ factors through
    a morphism  $\cC^{\tau,\BT,1}\to\cZ^{\dd,1}$ which is
    representable by algebraic spaces, scheme-theoretically dominant,
    and proper.
  \item The $\Fpbar$-points of $\cZ^{\tau,1}$ are naturally in
    bijection with the continuous representations $\rbar : G_K \to
    \GL_2(\Fpbar)$ which have a potentially Barsotti--Tate lift of
    type $\tau$.
  \end{enumerate}
\end{prop}

\begin{thm}
  \label{prop: dimensions of the Z stacks}
The algebraic stacks 
$\cZ^{\tau,1}$ are equidimensional of
dimension equal to~$[K:\Qp]$.
\end{thm}

\subsection{Dieudonn\'e and gauge stacks}
\label{sec:dieudonne-stacks}

In this subsection we express the passage from Breuil--Kisin modules (with descent data)
to Dieudonn\'e modules in stacky terms.  The consequent geometric results are most
conveniently expressed in
the language of effective Cartier on formal algebraic stacks, 
and so we briefly recall the relevant notions.  We also note that the material of
this subsection is only used in Subsection~\ref{subsec: map to Dieudonne stack} below.

To begin with,
we choose a tame inertial type $\tau=\eta\oplus\eta'$,
and we then specialise the choice of $K'$ in the following way. 
We fix a uniformiser $\pi$ of~$K$, and if $\tau$ is a tame
principal series type, we take $K'=K(\pi^{1/(p^f-1)})$, while
if~$\tau$ is a tame cuspidal type, we let $L$ be an unramified
quadratic extension of~$K$, and set $K'=L(\pi^{1/(p^{2f}-1)})$.
In either case $K'/K$ is a Galois extension; in the principal series
case, we have $e'=(p^f-1)e$, $f'=f$, and in the cuspidal case we have
$e'=(p^{2f}-1)e$, $f'=2f$.  Here and below, we refer to this choice of extension as the
\emph{standard choice} (for the fixed type $\tau$ and uniformiser
$\pi$). 

For the rest of this section we assume that
$\eta\ne\eta'$.  (For our intended
application we will not need to consider Dieudonn\'e modules for
scalar types).
Let $\gM$ be an object of $\cC^{\tau,\BT}(A)$, and let $D := \gM/u\gM$ be
its corresponding Dieudonn\'e module as in Definition~\ref{def:
  Dieudonne module formulas}.  The group 
  $I(K'/K)$ is abelian of order prime to $p$, and so we can write
  $D=D_\eta\oplus D_{\eta'}$, where $D_\eta$ is the submodule on which
  $I(K'/K)$ acts via~$\eta$. Setting $D_{\eta,j} := e_j D_{\eta}$, it
  follows from the
  projectivity of $\gM$ that each $D_{\eta,j}$ is an invertible
$A$-module. The maps $F,V$ induce linear
maps $F:D_{\eta,j}\to D_{\eta,j+1}$ and $V: D_{\eta,j+1} \to D_{\eta,j}$ 
such that $FV = VF = p$.

\begin{defn}\label{def:dieudonne-stack}
If $\tau$ is a principal series type we define a stack 
$$\cD_{\eta} :=
\Big[
\bigl( \Spec W(k)[X_0,Y_0,\ldots,X_{f-1},Y_{f-1}]/(X_j Y_j - p)_{j = 0,\ldots, f-1}) \bigr) /
\mathbb G_m^f  \big],$$
where the $f$ copies of $\mathbb G_m$ act as 
$(u_0,\ldots,u_{f-1}) \cdot (X_j,Y_j) \mapsto (u_j u_{j+1}^{-1} X_j, u_{j+1} u_j^{-1} Y_j),$
in which we take $u_f := u_0$.

If instead  $\tau$ is a cuspidal type we define  $$\cD_{\eta} :=
\Big[
\bigl( \Spec W(k)[X_0,Y_0,\ldots,X_{f-1},Y_{f-1}]/(X_j Y_j - p)_{j = 0,\ldots, f-1}) \times\Gm\bigr) /
\mathbb G_m^{f+1}  \big],$$
where  the $f+1$ copies of $\mathbb G_m$ act as $$(u_0,\ldots,u_{f-1},u_f) \cdot ((X_j,Y_j),\alpha) \mapsto ((u_j u_{j+1}^{-1}
X_j, u_{j+1} u_j^{-1} Y_j),\alpha ),$$
in which $\alpha$ and $u_0,\ldots,u_f$ are the coordinates of the $\Gm$'s appearing in the definition of $\cD_{\eta}$.
\end{defn}

In \cegsBdieudonnesection\ we explained how the stack $\cD_{\eta}$
classifies the line bundles $D_{\eta,j}$ together with the maps $F,V$,
so  that in either case
(principal series or cuspidal) there is a  natural map $\cC^{\tau,\BT} \to
\cD_{\eta}$.

It will be helpful to introduce another stack,
the stack $\cG_{\eta}$ of $\eta$-gauges.  This classifies
$f$-tuples of line bundles $\cD_j$ ($j = 0,\ldots,f-1$) equipped
with sections $X_j \in \cD_j$ and $Y_j \in \cD_j^{-1}$.
Explicitly, it can be written as the quotient stack
$$\cG_{\eta} :=
\Big[
\bigl( \Spec W(k)[X_0,Y_0,\ldots,X_{f-1},Y_{f-1}]/(X_j Y_j - p)_{j = 0,\ldots, f-1}) \bigr) /
\mathbb G_m^f  \big],$$
where the $f$ copies of $\mathbb G_m$ act as follows:
$$(v_0,\ldots,v_{f-1}) \cdot (X_j,Y_j) \mapsto (v_j X_j, v_j^{-1}
Y_j).$$

There is a morphism of stacks $\cD_{\eta} \to \cG_{\eta}$ which we can
define explicitly using their  descriptions as quotient stacks. 
Indeed, in  the principal series case 
we have a morphism $\Gm^f \to \Gm^f$ given by
$(u_j)_{j = 0,\ldots,f-1} \mapsto (u_j u_{j+1}^{-1})_{j = 0,\ldots,f-1}$,
which is compatible with the actions of these two groups on
$\Spec W(k)[(X_j,Y_j)_{j=0,\ldots,f-1}]/(X_j Y_j - p)_{j = 0,
\ldots , f-1},$ and we are just considering the map from the quotient
by the first $\Gm^f$ to the quotient by the second~$\Gm^f$. In the
cuspidal case we have a morphism $\Gm^{f+1} \to \Gm^f$ given by 
$(u_j)_{j = 0,\ldots,f} \mapsto (u_j u_{j+1}^{-1})_{j = 0,\ldots,f-1}$, 
and the morphism $\cD_{\eta} \to 
\cG_{\eta}$ is the obvious one which forgets the factor of~$\Gm$
coming from~$\alpha$.

Composing our morphism $\cC^{\tau,\BT} \to \cD_{\eta}$ with the forgetful morphism
$\cD_{\eta} \to \cG_{\eta}$, we obtain a morphism $\cC^{\tau,\BT} \to \cG_{\eta}$.

For our analysis of the irreducible components of the stacks
$\cC^{\tau,\BT,1}$  at the end of Section~\ref{sec: extensions of rank one Kisin
   modules},
it will be useful to have a 
more directly geometric interpretation
of a morphism $S \to \cG_{\eta}$, in the case that
the source is a {\em flat} $W(k)$-scheme, or, more generally,
a flat $p$-adic formal algebraic stack over~$\Spf
W(k)$. In order to do this we will need some basic material on
effective Cartier divisors for (formal) algebraic stacks; while it is
presumably possible to develop this theory in considerable generality,
 we only need a very special case, and we limit ourselves to this
 setting.

The property of a closed subscheme being an effective Cartier divisor is not 
preserved under arbitrary pull-back, but it is preserved under flat
pull-back. More precisely, we have the following result.

\begin{lemma}\label{lem:Cartier divisors are flat local}
       	If $X$ is a scheme,
	and $Z$ is a closed subscheme of $X$,
	then the following are equivalent:
	\begin{enumerate}
		\item $Z$ is an effective Cartier divisor on $X$.
		\item For any flat morphism of schemes $U \to X$,
			the pull-back $Z\times_{X} U$
			is an effective Cartier divisor on $U$.
	 	\item For some fpqc covering $\{X_i \to X\}$ of $X$,
	 		each of the pull-backs $Z\times_{X} X_i$
	 		is an effective Cartier divisor on $X_i$.
	\end{enumerate}
\end{lemma}
\begin{proof}
	Since $Z$ is an effective Cartier divisor if and only if its ideal sheaf
	$\cI_Z$ is an invertible sheaf on $X$, this follows from
	the fact that the invertibility of a quasi-coherent sheaf
	is a local property in the {\em fpqc} topology.
\end{proof}

\begin{lemma}
	\label{lem:comparing closed subsets}
	If $A$ is a Noetherian adic topological ring,
	then pull-back under the natural morphism $\Spf A \to \Spec A$
	induces a bijection between the closed subschemes of 
	$\Spec A$ and the closed subspaces of
	$\Spf A$.
\end{lemma}
\begin{proof}
	It follows
	from~\cite[\href{http://stacks.math.columbia.edu/tag/0ANQ}{Tag
 0ANQ}]{stacks-project}
that closed immersions $Z \to \Spf A$
are necessarily of the form $\Spf B \to \Spf A$,
and correspond to continuous morphisms $A \to B$, for some complete 
linearly topologized
ring $B$, which are taut (in the sense
	of~\cite[\href{http://stacks.math.columbia.edu/tag/0AMX}{Tag
 0AMX}]{stacks-project}),
have closed kernel, and dense image.
Since $A$ is adic, it admits a countable basis of neighbourhoods of the origin,
and so it follows 
	from~\cite[\href{http://stacks.math.columbia.edu/tag/0APT}{Tag
 0APT}]{stacks-project} (recalling also~\cite[\href{http://stacks.math.columbia.edu/tag/0AMV}{Tag
 0AMV}]{stacks-project})  that $A\to B$ is surjective.  
Because any ideal of definition $I$ of $A$ is finitely generated, it follows 
	from~\cite[\href{http://stacks.math.columbia.edu/tag/0APU}{Tag
 0APU}]{stacks-project} that $B$ is endowed with the $I$-adic topology.
Finally, since $A$ is Noetherian, any ideal in $A$ is $I$-adically closed.
Thus closed immersions $\Spf B \to \Spf A$ are determined by giving
the kernel of the corresponding morphism $A \to B$, which can be arbitrary.
The same is true of closed immersions $\Spec B \to \Spec A$,
and so the lemma follows.
\end{proof}

\begin{df} If $A$ is a Noetherian adic topological ring,
then we say that a closed subspace of $\Spf A$
is an {\em effective Cartier divisor} on $\Spf A$ if the corresponding closed
subscheme of $\Spec A$ is an effective Cartier divisor on $\Spec A$.
\end{df}

\begin{lemma}
	Let $\Spf B \to \Spf A$ be a flat adic morphism of Noetherian
	affine formal algebraic spaces.
	If $Z \hookrightarrow \Spf A$ is a Cartier divisor,
	then $Z \times_{\Spf A} \Spf B \hookrightarrow \Spf B$ 
	is a Cartier divisor.  Conversely, if $\Spf B \to \Spf A$
	is furthermore surjective, and if $Z \hookrightarrow \Spf A$
	is a closed subspace for which the base-change
	$Z \times_{\Spf A} \Spf B \hookrightarrow \Spf B$ 
	is a Cartier divisor,
	then $Z$ is a Cartier divisor on $\Spf A$.
\end{lemma}
\begin{proof}
	The morphism $\Spf B \to \Spf A$ corresponds to an adic flat morphism
	$A \to B$ 
(\cite[\href{http://stacks.math.columbia.edu/tag/0AN0}{Tag
0AN0}]{stacks-project}
and \cite[Lem.\ 8.18]{Emertonformalstacks})
	and hence is induced by a flat morphism $\Spec B \to \Spec A$,
	which is furthermore faithfully flat if and only if 
	$\Spf B \to \Spf A$ is surjective
(again by \cite[Lem.\ 8.18]{Emertonformalstacks}).
	The present lemma thus follows from Lemma~\ref{lem:Cartier divisors
	are flat local}. 
\end{proof}

The preceding lemma justifies the following
definition.

\begin{df} We say that a closed substack $\cZ$ of a locally Noetherian
	formal algebraic stack
	$\cX$
	is an {\em effective Cartier divisor} on $\cX$ if
		 for any morphism $U \to \cX$ whose source
			is a Noetherian affine formal algebraic space,
			and which is representable by algebraic spaces and
			flat,
			the pull-back $\cZ\times_{\cX} U$
			is an effective Cartier divisor on $U$.
\end{df}

We consider the $W(k)$-scheme $\Spec W(k)[X,Y]/(XY - p)$, which we endow with
a $\Gm$-action via $u\cdot (X,Y) := (uX, u^{-1} Y).$
There is an obvious morphism
$$\Spec W(k)[X,Y]/(XY -p) \to \Spec W(k)[X]=\mathbb A^1$$
given by $(X,Y) \to X$, which is $\Gm$-equivariant (for the action
of~$\Gm$ on~$\mathbb A^1$ given by $u\cdot X:=uX$),
and so induces a morphism
\numequation
\label{eqn:Cartier divisor construction}
[\bigl(\Spec W(k)[X,Y]/(XY -p)\bigr) / \Gm]
\to  [\mathbb A^1/\Gm].
\end{equation}

\begin{lemma}\label{lem: f equals 1 effective Cartier} If $\cX$ is a
  locally Noetherian $p$-adic formal algebraic stack
	which is furthermore flat over $\Spf W(k)$,
        then the groupoid of morphisms
	\[\cX \to [\Spec W(k)[X,Y]/(XY - p) / \Gm]\] is in fact 
	a setoid, and is equivalent to the set of effective Cartier divisors 
	on $\cX$ that are contained in the effective Cartier divisor
	$(\Spec k) \times_{\Spf W(k)} \cX$ on~$\cX$. 
\end{lemma}
\begin{proof}
  Essentially by definition (and taking into account \cite[Lem.\
  8.18]{Emertonformalstacks}), it suffices to prove this in the case
  when $\cX = \Spf B$, where $B$ is a flat Noetherian adic
  $W(k)$-algebra admitting $(p)$ as an ideal of definition.  In this
  case, the restriction map
  \[[\Spec W(k)[X,Y]/(XY - p) / \Gm](\Spec B)\to [\Spec W(k)[X,Y]/(XY -
  p) / \Gm](\Spf B)\] is an equivalence of groupoids. 
Indeed, the
  essential surjectivity follows from the (standard and easily
  verified) fact that if $\{M_i\}$ is a compatible family of locally
  free $B/p^iB$-modules of rank one, then $M := \varprojlim M_i$ is a
  locally free $B$-module of rank one, for which each of the natural
  morphisms $M/p^iM \to M_i$ is an isomorphism.  The full
  faithfulness follows from the fact that a locally free $B$-module
  of rank one is $p$-adically complete, and so is recovered as the
  inverse limit of its compatible family of quotients $\{M/p^iM\}.$

We are therefore reduced to the same statement with $\cX = \Spec B$.  The composite morphism $\Spec B \to  [\mathbb A^1/\Gm]$ induced
  by~\eqref{eqn:Cartier divisor construction} corresponds to giving a
  pair~$(\cD,X)$ where~$\cD$ is a line bundle  on~$\Spec B$, and~$X$
  is a global section of~$\cD^{-1}$. Indeed, giving a morphism $\Spec
  B \to  [\mathbb A^1/\Gm]$ is equivalent
	to giving a $\Gm$-torsor $P \to \Spec B$, together with a $\Gm$-equivariant
	morphism $P \to \mathbb A^1$.   Giving a $\Gm$-torsor 
	$P$ over $\Spec B$ is equivalent to giving an invertible sheaf
        $\cD$ on $\Spec B$
	(the associated $\Gm$-torsor is then obtained by deleting the
	zero section from the line bundle $D\to X$ corresponding to $\cD$),
	and giving a $\Gm$-equivariant morphism $P \to \mathbb A^1$
	is equivalent to giving a global section of $\cD^{-1}$.

It follows that giving a morphism
  $\Spec B \to [\Spec W(k)[X,Y]/(XY - p) / \Gm]$ corresponds to giving
  a line bundle $\cD$ and sections $X \in \cD^{-1}$, $Y \in \cD$ satisfying
  $X Y = p$.  To say that $B$ is flat over $W(k)$ is just to say that
  $p$ is a regular element on $B$, and so we see that $X$ 
  (resp.\ $Y$) is a regular section of $\cD^{-1}$ (resp.\ $\cD$).
  Again, since $p$ is a
  regular element on $B$, we see that $Y$ is uniquely determined by
  $X$ and the equation $X Y = p$, and so giving a morphism
  $\Spec B\to [\Spec W(k)[X,Y]/(XY - p) / \Gm]$ is equivalent to giving a line
  bundle $\cD$ and a regular section $X$ of $\cD^{-1}$, such that
  $pB \subset X\otimes_B \cD \subset \cD^{-1}\otimes_B \cD\iso B$;
  this last condition guarantees the existence of the (then uniquely
  determined) $Y$.

Now giving a line bundle $\cD$ on $\Spec B$ and a regular section
$X\in\cD^{-1}$ is the
same as giving the zero locus $D$ of $X$, which is a Cartier divisor
on $\Spec B$.
(There is a canonical isomorphism $(\cD,X) \cong
\bigl(\cI_D,1\bigr)$, where $\cI_D$ denotes the ideal sheaf of $D$.) 
The condition that $pB \subset X\otimes_B \cD$ is equivalent 
to the condition that $p \in \cI_D$,
i.e.\ that $D$ be contained in $\Spec B/pB$, and we are done.
\end{proof}

\begin{lemma}\label{lem: maps to gauge stack as Cartier
    divisors} 
If $\cS$ is a locally Noetherian $p$-adic formal algebraic stack which
is flat over $W(k)$, 
then giving a morphism $\cS \to \cG_{\eta}$
over $W(k)$ is equivalent to giving a collection
of effective Cartier divisors $\cD_j$ on $\cS$ {\em (}$j = 0,
\ldots,f-1${\em )}, with each $\cD_j$ contained in the Cartier divisor
$\overline{\cS}$ cut out by the equation $p = 0$ on $\cS$ {\em (}i.e.\
the {\em special fibre} of $\cS${\em )}.
\end{lemma}
\begin{proof}
	This follows immediately from Lemma~\ref{lem: f equals 1
          effective Cartier}, by the definition of~$\cG_\eta$.
\end{proof}

\section{Families of extensions of Breuil--Kisin modules}\label{sec: extensions of rank one Kisin modules}

The goal of the next two sections is to construct certain universal families of
extensions of rank one Breuil--Kisin modules over $\F$ with descent data;\ these families will be
used in Section~\ref{sec:Components} to describe the generic behaviour of the various irreducible 
components of the special fibres of~$\cC^{\tau,\BT}$ and~$\cZ^{\tau}$.

In Subsections~\ref{subsec:ext generalities} and~\ref{subsec:families of extensions} we present some generalities
on extensions of Breuil--Kisin modules and on families of these extensions, respectively. In Subsection~\ref{subsec:universal families}
we specialize the discussion of Subsection~\ref{subsec:families of extensions} to the case of extensions of two rank one Breuil--Kisin modules, and thus explain how to construct our desired families of extensions.
In Section~\ref{sec:extensions-of-rank-one} we  recall the fundamental
computations related to extensions of rank one Breuil--Kisin modules
from \cite{DiamondSavitt}, to which the results of
Subsection~\ref{subsec:universal families} will be applied at the end of Subsection~\ref{sec:extensions-profile-J} to construct the components $\overline{\cC}(J)$ and $\overline{\cZ}(J)$ of Theorem~\ref{thm:main thm cegsC}.

We assume throughout this section that $[K':K]$ is not divisible
by~$p$; since we are assuming throughout the paper that $K'/K$ is
tamely ramified, this is equivalent to assuming that $K'$ does not
contain an unramified extension of $K$ of degree~$p$. In our final
applications $K'/K$ will contain unramified extensions of degree at
most~$2$, and $p$ will be odd, so this assumption will be satisfied.
(In fact, we specialize to such a context begining in
Subsection~\ref{subsec:irreducible}.)

\subsection{Extensions of Breuil--Kisin modules with descent data}
\label{subsec:ext generalities}
When discussing the general theory of extensions of Breuil--Kisin
modules, it is convenient to embed the category of Breuil--Kisin modules
in a larger category which is abelian,
contains enough injectives and projectives,
and is closed under passing to arbitrary limits and colimits.
The simplest way to obtain such a category is as the category of modules 
over some ring, and so we briefly recall how a Breuil--Kisin module with
$A$-coefficients and  descent 
data can be interpreted as a module over a certain $A$-algebra.

Let $\gS_A[F]$ denote the twisted polynomial ring over $\gS_A$,
in which the variable $F$ obeys the following commutation relation
with respect to elements $s \in \gS_A$:
$$ F \cdot s = \varphi(s) \cdot F.$$
Let $\gS_A[F, \Gal(K'/K)]$ denote the twisted group ring over $\gS_A[F]$,
in which the elements $g \in \Gal(K'/K)$ commute with $F$,
and obey the following commutation
relations with elements $s  \in \gS_A$:
$$ g \cdot s = g(s) \cdot g.$$
One immediately confirms that giving a left $\gS_A[F, \Gal(K'/K)]$-module $\gM$
is equivalent to equipping the underlying $\gS_A$-module
$\gM$ with a $\varphi$-linear morphism
$\varphi:\gM \to \gM$ and a semi-linear action of $\Gal(K'/K)$
which commutes with $\varphi$.

In particular, if we let $\K{A}$ denote the category of left $\gS_A[F,
\Gal(K'/K)]$-modules, then
a Breuil--Kisin module with descent data from $K'$ to $K$ 
 may naturally be regarded as an object of $\K{A}$.
In the following lemma, we record the fact that extensions of Breuil--Kisin modules
with descent data may be computed as extensions in the category~$\K{A}$.  

\begin{lemma}
\label{lem:ext of a Kisin module is a Kisin module}
If $0 \to \gM' \to \gM \to \gM'' \to 0$ is a short exact sequence
in $\K{A}$, such that $\gM'$ {\em (}resp.\ $\gM''${\em )}
is a Breuil--Kisin module with descent data
of rank $d'$ and height at most $h'$ 
{\em (}resp.\ of rank $d''$ and height at most $h''$\emph{)}, 
then $\gM$ is a Breuil--Kisin module with descent data
of rank $d'+d''$ and height at most $h'+h''$.  

More generally, if
$E(u)^h\in\Ann_{\gS_A}(\coker\Phi_{\gM'})\Ann_{\gS_A}(\coker\Phi_{\gM''})$,
then $\gM$ is a Breuil--Kisin module with descent data of height at most~$h$.
\end{lemma}
\begin{proof}
Note that since $\Phi_{\gM'}[1/E(u)]$ and $\Phi_{\gM''}[1/E(u)]$ are both isomorphisms by
assumption, it follows from the snake lemma that~$\Phi_{\gM}[1/E(u)]$ is
isomorphism.  Similarly we have a short exact sequence of
$\gS_A$-modules \[0\to\coker\Phi_{\gM'}\to\coker\Phi_{\gM}\to\coker\Phi_{\gM''}\to
0.\]
The claims about the height and rank of~$\gM$ follow immediately.
\end{proof}

We now turn to giving an explicit description of the functors
$\Ext^i(\gM, \text{--} \, )$ for a Breuil--Kisin module with descent data
$\gM$. 

\begin{df}
\label{def:explicit Ext complex}
Let $\gM$ be a 
Breuil--Kisin module with $A$-coefficients and descent data (of some 
height).
If $\gN$ is any object of $\K{A}$, then 
we let $C^{\bullet}_{\gM}(\gN)$ denote the complex
$$ 
\Hom_{\gS_A[\Gal(K'/K)]}(\gM,\gN) \to
\Hom_{\gS_A[\Gal(K'/K)]}(\varphi^*\gM,\gN), 
$$
with differential being given by 
$$\alpha \mapsto \Phi_{\gN} \circ \varphi^* \alpha - \alpha \circ \Phi_{\gM}.$$
Also let $\Phi_{\gM}^*$ denote the map $C^0_{\gM}(\gN) 
\to C^1_{\gM}(\gN)$ given by $\alpha \mapsto \alpha \circ
\Phi_{\gM}$. When $\gM$ is clear from the context we will usually suppress it from the
notation and write simply $C^{\bullet}(\gN)$. 
\end{df}

Each $C^i(\gN)$ is naturally an $\gS_A^0$-module. 
The formation of $C^{\bullet}(\gN)$ is evidently functorial in $\gN$,
and is also exact in $\gN$, since $\gM$, and hence also $\varphi^*\gM$,
is projective over $\gS_A$, and since $\Gal(K'/K)$ has prime-to-$p$ order.
Thus the cohomology functors
$H^0\bigl(C^{\bullet}(\text{--})\bigr)$
and
$H^1\bigl(C^{\bullet}(\text{--})\bigr)$
form a $\delta$-functor on $\K{A}$.

\begin{lemma}\label{lem: C computes Hom}
There is a natural isomorphism
$$\Hom_{\K{A}}(\gM,\text{--}\, ) \cong
H^0\bigl(C^{\bullet}(\text{--}\,)\bigr).$$
\end{lemma}
\begin{proof}
This is immediate.
\end{proof}

It follows from this lemma and 
a standard dimension shifting argument (or, equivalently, the theory
of $\delta$-functors) that there is an embedding of functors
\numequation
\label{eqn:ext embedding}
\Ext^1_{\K{A}}(\gM, \text{--} \, ) \hookrightarrow
H^1\bigl(C^{\bullet}(\text{--})\bigr).
\end{equation}

\begin{lemma}\label{lem: C computes Ext^1}
The embedding of functors~\eqref{eqn:ext embedding}
is an isomorphism.
\end{lemma}
\begin{proof}
We first describe the embedding~\eqref{eqn:ext embedding} explicitly.
Suppose that 
$$0 \to \gN \to \gE \to \gM \to 0$$
is an extension in $\K{A}$.  Since $\gM$ is projective over $\gS_A$,
and since $\Gal(K'/K)$ is of prime-to-$p$ order, 
we split this
short exact sequence
over the twisted group ring $\gS_A[\Gal(K'/K)],$ 
say via some element $\sigma \in \Hom_{\gS_A[\Gal(K'/K)]}(\gM,\gE).$
This splitting is well-defined up to the addition of an
element $\alpha \in \Hom_{\gS_A[\Gal(K'/K)]}(\gM,\gN).$

This splitting is a homomorphism in $\K{A}$ if and only
if the element
$$\Phi_{\gE}\circ \varphi^*\sigma - \sigma \circ \Phi_{\gM}
\in \Hom_{\gS_A[\Gal(K'/K)]}(\varphi^*\gM,\gN)$$ vanishes.
If we replace $\sigma$ by $\sigma +\alpha,$ then this element
is replaced by
$$(\Phi_{\gE}\circ \varphi^*\sigma - \sigma \circ \Phi_{\gM})
+  (\Phi_{\gN} \circ \varphi^* \alpha - \alpha \circ \Phi_{\gM}).$$
Thus the coset of 
$\Phi_{\gE}\circ \varphi^*\sigma - \sigma \circ \Phi_{\gM}$
in $H^1\bigl(C^{\bullet}(\gN)\bigr)$ is well-defined, independent
of the choice of $\sigma,$
and this coset is the image of the class of the extension
$\gE$ under the embedding
\numequation
\label{eqn:explicit embedding}
\Ext^1_{\K{A}}(\gM,\gN) \hookrightarrow 
H^1\bigl(C^{\bullet}(\gN)\bigr)
\end{equation}
(up to a possible overall sign, which we ignore, since it doesn't
affect the claim of the lemma).

Now, given any element  
$\nu \in \Hom_{\gS_A[\Gal(K'/K)]}(\varphi^*\gM,\gN)$,
we may give the $\gS_A[\Gal(K'/K)]$-module $\gE := \gN \oplus \gM$ the
structure of a $\gS_A[F,\Gal(K'/K)]$-module as follows: we need to
define a $\varphi$-linear morphism $\gE\to\gE$, or equivalently a linear
morphism $\Phi_{\gE}:\varphi^*\gE\to\gE$. We do this
by setting $$\Phi_{\gE} := \begin{pmatrix} \Phi_{\gN} & \nu \\ 0 & \Phi_{\gM} 
\end{pmatrix}.$$
Then $\gE$ is an extension of $\gM$ by $\gN$,
and if we let $\sigma$ denote the obvious embedding of $\gM$ into $\gE$,
then one computes that 
$$\nu = \Phi_{\gE}\circ \varphi^*\sigma - \sigma \circ \Phi_{\gM} .$$
This shows that~\eqref{eqn:explicit embedding} is an isomorphism, as claimed.
\end{proof}

Another dimension shifting argument, taking into account the preceding
lemma, shows that $\Ext^2_{\K{A}}(\gM,\text{--} \,)$ embeds into
$H^2\bigl( C^{\bullet}(\text{--}) \bigr).$ Since the target of this
embedding vanishes,
we find that the same is true of the source.  This yields the
following corollary.

\begin{cor}
\label{cor:ext2 vanishes}
If $\gM$ is a Breuil--Kisin module with $A$-coefficients and descent data,
then $\Ext^2_{\K{A}}(\gM, \text{--} \, ) = 0.$
\end{cor}
We summarise the above discussion in the following corollary.
\begin{cor}
  \label{cor:complex computes Hom and Ext}If $\gM$ is a 
  Breuil--Kisin module with $A$-coefficients and descent data, and $\gN$ is an
  object of~$\K{A}$, then we have a natural short exact
  sequence \[0\to\Hom_{\K{A}}(\gM,\gN)\to C^0(\gN)\to
    C^1(\gN)\to\Ext^1_{\K{A}}(\gM,\gN)\to 0.\]
\end{cor}

The following lemma records the behaviour of these complexes with
respect to base change.

\begin{lemma}\label{lem:base-change-complexes}
  Suppose that $\gM$, $\gN$ are Breuil--Kisin modules with descent data and
  $A$-coefficients, that $B$ is an $A$-algebra, and that $Q$ is a
  $B$-module. Then the
  complexes $C^{\bullet}_{\gM}(\gN \cotimes_A Q)$ and
  $C^{\bullet}_{\gM \cotimes_A B}(\gN \cotimes_A Q)$ coincide, the
  former complex
  formed with respect to $\K{A}$ and the latter with respect to $\K{B}$.
\end{lemma}

\begin{proof}
Indeed, there is a natural isomorphism  $$\Hom_{\gS_A[\Gal(K'/K)]}(\gM, \gN\cotimes_A Q) \cong
\Hom_{\gS_B[\Gal(K'/K)]}(\gM\cotimes_A B, \gN\cotimes_A Q),$$
and similarly with $\varphi^*\gM$ in place of $\gM$.
\end{proof}

The following slightly technical lemma is
crucial for establishing 
finiteness properties, and also base-change properties,
of Exts of Breuil--Kisin modules.

\begin{lem}
  \label{lem:truncation argument used to prove f.g. of Ext Q version}Let $A$ be
  a $\cO/\varpi^a$-algebra for some $a\ge 1$, suppose that
  $\gM$ 
is a Breuil--Kisin module with descent data and $A$-coefficients,
of height at most~$h$,
and suppose that $\gN$ is a $u$-adically complete, $u$-torsion
free object of $\K{A}$. 

Let $C^\bullet$ be the complex defined
  in \emph{Definition~\ref{def:explicit Ext complex}}, and write~$\delta$ for
  its differential. Suppose that $Q$ is an $A$-module with
the property that $C^i\otimes_A Q$ is $v$-torsion free for $i=0,1$
and $v$-adically  separated for $i=0$.

 Then:
  \begin{enumerate}
  \item For any integer $M\ge (eah+1)/(p-1)$, $\ker (\delta\otimes \id_Q)\cap
    v^MC^0\otimes_AQ=0$.
  \item For any integer $N \ge (peah+1)/(p-1)$, $\delta\otimes\id_Q$ induces an isomorphism
\[(\Phi_{\gM}^*)^{-1}(v^N C^1\otimes_AQ) \isoto v^N (C^1\otimes_AQ).\]
  \end{enumerate}
Consequently, for $N$ as in \emph{(2)} the natural morphism of complexes of $A$-modules
$$[ C^0\otimes_AQ \buildrel \delta\otimes\id_Q \over \longrightarrow C^1\otimes_AQ] \to
[C^0\otimes_AQ/\bigl((\Phi_{\gM}^*)^{-1}(v^N C^1\otimes_AQ) \bigr) \buildrel \delta\otimes\id_Q
\over \longrightarrow C^1\otimes_AQ/v^N C^1\otimes_AQ ]$$
is a quasi-isomorphism.
\end{lem}

Since we are assuming that the $C^i\otimes_AQ$ are $v$-torsion free, 
the expression $v^r C^i(\gN) \otimes_A Q$ may be interpreted as
denoting either $v^r \bigl( C^i(\gN)\otimes_A Q\bigr)$ or
$\bigl(v^r C^i(\gN) \bigr)\otimes_A Q$, the two being naturally
isomorphic.

\begin{rem}\label{rem:truncation-remark} Before giving the proof of Lemma~\ref{lem:truncation
    argument used to prove f.g. of Ext Q version}, we observe that the
  hypotheses on the $C^i \otimes_A Q$ are satisfied if either $Q=A$, or
  else $\gN$ is a projective $\gS_A$-module and $Q$ is a finitely
  generated $B$-module for some finitely
generated $A$-algebra $B$.  (Indeed $C^1 \otimes_A Q$
  is $v$-adically separated as well in these cases.)

(1)  Since $\gM$ is projective of finite rank over $A[[u]]$,
and since $\gN$ is $u$-adically complete and $u$-torsion free,
each  $C^i$ is $v$-adically 
separated
and $v$-torsion free. In particular the hypothesis on~$Q$ is always satisfied
by~$Q=A$. (In fact since $\gN$ is $u$-adically complete it also follows that 
the $C^i$ are $v$-adically complete. Here we use that $\Gal(K'/K)$ has
order prime to $p$ to see that $C^0$ is an $\gS_A^0$-module direct
summand of $\Hom_{\gS_A}(\gM,\gN)$, and similarly for $C^1$.)

(2) Suppose $\gN$ is a projective $\gS_A$-module. Then the $C^i$ are
projective $\gS_A^0$-modules, again using that $\Gal(K'/K)$ has
order prime to $p$. Since each $C^i(\gN)/v C^i(\gN)$ is $A$-flat, it
follows that $C^i(\gN) \otimes_A Q$ is $v$-torsion free. If furthermore $B$ is a finitely
generated $A$-algebra, and $Q$ is a finitely generated $B$-module,
then the $C^i(\gN)\otimes_A Q$ are
$v$-adically separated (being finitely generated modules over
the ring $A[[v]]\otimes_A B$, which is a finitely generated 
algebra over the Noetherian ring $A[[v]]$, and hence is itself
Noetherian).
\end{rem} 
\begin{proof}[Proof of Lemma~{\ref{lem:truncation argument used to
    prove f.g. of Ext Q version}}] Since $p^a=0$ in $A$, there exists $H(u)\in\gS_A$ with
  $u^{e'ah}=E(u)^hH(u)$ in $\gS_A$. Thus the image of $\Phi_{\gM}$
contains $u^{e'ah} \gM=v^{eah}\gM$, and there exists a map $\Upsilon : \gM \to
\varphi^* \gM$ such that $\Phi_{\gM} \circ \Upsilon$ is multiplication by $v^{eah}$.

We begin with~(1). Suppose that 
  $f\in\ker (\delta\otimes\id_Q)\cap v^MC^0\otimes_AQ$. 
Since $C^0\otimes_AQ$ is $v$-adically separated, 
it is enough, applying induction on $M$, 
to show that $f\in v^{M+1}C^0\otimes_AQ$. Since
  $f\in\ker(\delta\otimes\id_Q)$, we have
  $f\circ\Phi_{\gM}=\Phi_{\gN}\circ\varphi^*f$. Since $f\in v^MC^0\otimes_AQ$,
  we have $f \circ \Phi_{\gM} = \Phi_{\gN} \circ \varphi^*f \in
  v^{pM} C^1 \otimes_A Q$. Precomposing with $\Upsilon$ gives 
 $v^{eah} f \in v^{pM} C^0 \otimes_A Q$. 
 Since~$C^0 \otimes_A Q$ is $v$-torsion free, it follows that $f\in
  v^{pM-eah}C^0\otimes_AQ\subseteq u^{M+1}C^0\otimes_AQ$, as required.
  
We now move on to~(2).  Set $M = N - eah$. 
By precomposing with $\Upsilon$ we see that $\alpha \circ \Phi_{\gM} \in v^N C^1\otimes_A Q$ implies $\alpha \in v^M C^0\otimes_A Q$; from this, together with the inequality
$pM \ge N$, it is
straightforward to check that
\[(\Phi_{\gM}^*)^{-1}(v^N C^1\otimes_AQ) = (\delta \otimes \id_Q)^{-1}(v^N
C^1\otimes_AQ)\cap  v^M C^0\otimes_AQ.\]
Note that $M$ satisfies the condition in~(1). To complete the proof we
will show that for any $M$ as in (1) and any $N \ge M + eah$ the map
$\delta$ induces an isomorphism 
\[ (\delta \otimes \id_Q)^{-1}(v^N C^1\otimes_AQ)\cap  v^M C^0\otimes_AQ \isoto v^N C^1\otimes_AQ.\]
 By~(1), $\delta\otimes\id_Q$ induces an injection $(\delta\otimes\id_Q)^{-1}(v^N C^1\otimes_AQ)\cap
    v^M C^0\otimes_AQ\hookrightarrow v^N C^1\otimes_AQ$, so it is enough to show that
    $(\delta\otimes\id_Q)(v^MC^0\otimes_AQ)\supseteq v^N
    C^1\otimes_AQ$. 
Equivalently, we need to show that 
$$
v^N C^1 \otimes_A Q
\to
(C^1\otimes_A Q) / 
(\delta\otimes\id_Q)\bigl(v^M C^0\otimes_A Q) 
$$
is identically zero. Since the formation of cokernels is compatible with tensor products,
we see that this morphism is obtained by tensoring the corresponding morphism
$$
v^N C^1 
\to
C^1/
\delta\bigl(v^M C^0\bigr)
$$
with $Q$ over $A$, so we are reduced to the case $Q=A$. (Recall from
Remark~\ref{rem:truncation-remark}(1) that the
hypotheses of the Lemma are satisfied in this case, and that $C^1$ is 
$v$-adically separated.)

We claim that for
    any $g\in v^NC^1$, we can find an $f\in v^{N-eah}C^0$ such that
    $\delta(f)-g\in v^{p(N-eah)}C^1$. Admitting the claim, given any
    $g\in v^N C^1$, we may find $h\in v^MC^0$ with $\delta(h)=g$ by
    successive approximation in the following way: 
Set $h_0=f$ for~$f$
    as in the claim; then $h_0\in v^{N-eah}C^0\subseteq v^MC^0$, and
    $\delta(h_0)-g\in v^{p(N-eah)}C^1\subseteq v^{N+1}C^1$. Applying
    the claim again with $N$ replaced by $N+1$, and $g$ replaced by
    $g-\delta(h_0)$, we find $f\in v^{N+1-eah}C^0\subseteq v^{M+1}C^0$
    with $\delta(f)-g+\delta(h_0)\in v^{p(N+1-eah)}C^1\subseteq
    v^{N+1}C^1$. Setting $h_1=h_0+f$, and proceeding inductively, we
    obtain a Cauchy sequence converging (in the $v$-adically
complete $A[[v]]$-module $C^0$) to the required element~$h$.

It remains to prove the claim. Since
$\delta(f)=\Phi_{\gN}\circ\varphi^*f-f\circ\Phi_{\gM} $, and since if
$f\in v^{N-eah}C^0$ then $\Phi_{\gN}\circ\varphi^*f\in v^{p(N-eah)}C^1$,
it is enough to show that we can find an $f\in v^{N-eah}C^0$ with
$f\circ\Phi_{\gM}=-g$. Since $\Phi_{\gM}$ is injective, the map
$\Upsilon \circ \Phi_\gM$ is also multiplication by $v^{eah}$, and so
it suffices to take $f$ with $v^{eah} f = -g \circ \Upsilon \in v^N C^0$. 
\end{proof}

\begin{cor}\label{cor: base change completion for complex in free case}
  Let $A$ be a Noetherian 
  $\cO/\varpi^a$-algebra,
  and let $\gM$, $\gN$ be Breuil--Kisin modules with descent data
  and $A$-coefficients. If $B$ is a finitely generated 
$A$-algebra, and $Q$ is a finitely generated $B$-module,
then the natural morphism of complexes of $B$-modules
$$[ C^0(\gN)\otimes_A Q \buildrel {\delta\otimes \id_Q} \over \longrightarrow
C^1(\gN)\otimes_A Q] \to
[C^0(\gN\cotimes_A Q) \buildrel \delta \over \longrightarrow
C^1(\gN\cotimes_A Q)]$$
is a quasi-isomorphism.
\end{cor}
\begin{proof} 
By Remarks~\ref{rem:truncation-remark} and~\ref{rem:completed tensor}(2)  we can apply Lemma~\ref{lem:truncation argument used to
  prove f.g. of Ext Q version} to both
$C^i(\gN\cotimes_AQ)$ and  $C^i(\gN)\otimes_AQ$, and we see that it is enough to show that the
natural morphism of complexes 
\[\begin{adjustbox}{max width=\textwidth}
\begin{tikzcd}
{[\bigl(C^0(\gN)\otimes_A Q \bigr)/
(\Phi_{\gM}^*\otimes \id_Q)^{-1}\bigl(v^N C^1(\gN)\otimes_A Q \bigr) 
\buildrel \delta
\over \longrightarrow \bigl(C^1(\gN)\otimes_A Q\bigr)/\bigl(v^N C^1(\gN)
\otimes_A Q\bigr) ]}
\arrow{d}{} \\
{[C^0(\gN\cotimes_A Q)
/\bigl(\Phi_{\gM}^*)^{-1}(v^N C^1(\gN\cotimes_A Q)\bigr) 
\buildrel \delta
\over \rightarrow C^1(\gN\cotimes_A Q)/v^N C^1(\gN\cotimes_A Q) ]}
\end{tikzcd}
\end{adjustbox}\]
is a quasi-isomorphism. In fact, it is even an isomorphism. 
\end{proof}

\begin{prop}
\label{prop:exts are f.g. over A} Let $A$ be a 
$\cO/\varpi^a$-algebra for some $a\ge 1$, and let $\gM$, $\gN$ be
Breuil--Kisin modules with descent data and $A$-coefficients. 
Then  $\Ext^1_{\K{A}}(\gM,\gN)$ and $\Ext^1_{\K{A}}(\gM,\gN/u^i\gN)$
for $i \ge 1$ are finitely presented $A$-modules.

If
furthermore $A$ is Noetherian, then $\Hom_{\K{A}}(\gM,\gN)$ and
$\Hom_{\K{A}}(\gM,\gN/u^i\gN)$ for $i\ge 1$ are also
finitely presented {\em (}equivalently, finitely generated{\em )} $A$-modules.
\end{prop}
\begin{proof} 
The statements for $\gN/u^i\gN$
  follow easily from those for $\gN$, by considering the short exact sequence $0 \to u^i\gN \to \gN \to
\gN/u^i\gN \to 0$ in $\K{A}$ and applying Corollary~\ref{cor:ext2
  vanishes}.
	By Corollary~\ref{cor:complex computes Hom and Ext}, it is enough to consider the cohomology of the
  complex~$C^\bullet$. By Lemma~\ref{lem:truncation argument used to
    prove f.g. of Ext Q version} with $Q=A$, 
the cohomology of~$C^\bullet$ agrees with the
  cohomology of the induced complex \[C^0/\bigl((\Phi_{\gM}^*)^{-1}(v^N C^1) )\to C^1/ v^N C^1,\] 
for an appropriately chosen value of $N$. It follows that for an
appropriately chosen value of~$N$, $\Ext^1_{\K{A}}(\gM,\gN)$ can be
computed as the cokernel of the induced morphism $C^0/v^N C^0 \to C^1/ v^N C^1$.

Under our hypothesis on~$\gN$, $C^0/v^N C^0$ and $C^1/v^NC^1$ are finitely
generated projective $A$-modules, and thus finitely presented. It follows  that
$\Ext^1_{\K{A}}(\gM,\gN)$ is finitely presented. 

In the case that $A$ is furthermore assumed to be Noetherian, it is
enough to note that since $v^NC^0\subseteq (\Phi_{\gM}^*)^{-1}(v^N C^1)$,
the quotient $C^0/\bigl((\Phi_{\gM}^*)^{-1}(v^N C^1) \bigr)$ is a finitely generated $A$-module.
\end{proof}

\begin{prop}
  \label{prop:descent for Homs of free Kisin modules}Let $A$ be a 
$\cO/\varpi^a$-algebra for some $a\ge 1$,
and let $\gM$ and~$\gN$ be Breuil--Kisin modules with descent
data and $A$-coefficients. Let $B$ be an $A$-algebra, and let
      $f_B:\gM\cotimes_AB\to\gN\cotimes_AB$ be  a morphism of Breuil--Kisin
      modules with $B$-coefficients.

      Then there is a finite type $A$-subalgebra $B'$ of~$B$ and a morphism
      of Breuil--Kisin modules $f_{B'}:\gM\cotimes_A B'\to\gN\cotimes_A B'$
      such that $f_B$ is the base change of~$f_{B'}$.
\end{prop}
\begin{proof}

By Lemmas~\ref{lem: C computes Hom}
and~\ref{lem:base-change-complexes} (the latter applied with $Q=B$) we can and do think
  of $f_B$ as being an element of the kernel of
  $\delta:C^0(\gN\cotimes_A B)\to C^1(\gN\cotimes_A
  B)$, the complex $C^\bullet$ here and throughout this proof denoting
  $C^{\bullet}_{\gM}$ as usual.

Fix $N$ as in
  Lemma~\ref{lem:truncation argument used to prove f.g. of Ext Q version}, and
  write~$\fbar_B$ for the corresponding element of
  $C^0(\gN\cotimes_A B)/v^N=(C^0(\gN)/v^N)\otimes_A B$ (this equality
  following easily from the assumption that $\gM$ and $\gN$ are
  projective $\gS_A$-modules of finite rank). Since $C^0(\gN)/v^N$
  is a projective $A$-module of finite
  rank, 
  it follows 
  that for some finite type
  $A$-subalgebra $B'$ of~$B$, there is an element $\fbar_{B'}\in
  (C^0(\gN)/v^N)\otimes_A B'=C^0(\gN\cotimes_A B')/v^N$ such that
  $\fbar_{B'}\otimes_{B'}B=\fbar_B$. Denote also by $\fbar_{B'}  $ the induced element of \[C^0(\gN\cotimes
_A B')/\bigl(\Phi_{\gM}^*)^{-1}(v^N C^1(\gN\cotimes
_A B')).\] 

By Lemma~\ref{lem:truncation argument used to prove
f.g. of Ext Q version} (and Lemma~\ref{lem: C computes Hom}) we have a
commutative diagram with exact rows \[\xymatrix{0 \ar[r] & H^0(C^{\bullet}(\gN\cotimes
_A B')) \ar[d] \ar[r] & C^0(\gN\cotimes
_A B')/\bigl((\Phi_{\gM}^*)^{-1}(v^N C^1(\gN\cotimes
_A B')) 
\bigr) \ar[r]^-{\delta}\ar[d] & C^1(\gN\cotimes
_A B')/v^N\ar[d]\\ 0 \ar[r] & H^0(C^{\bullet}(\gN\cotimes
_A B)) \ar[r] & C^0(\gN\cotimes
_A B)/\bigl((\Phi_{\gM}^*)^{-1}(v^N C^1 (\gN\cotimes
_A B)) 
\bigr) \ar[r]^-{\delta} & C^1(\gN\cotimes
_A B)/v^N }  \]
in which the vertical arrows 
are induced by $\cotimes_{B'}B$.  By a diagram chase we only need to
show that $\delta(\fbar_{B'})=0$. Since $\delta(f_B)=0$, it is
enough to show that the right hand vertical arrow is an
injection. This arrow can be rewritten as the tensor product of the
injection of $A$-algebras $B'\into B$ with the flat (even projective
of finite rank) $A$-module $C^1(\gN)/v^N$, so the result follows.
\end{proof}
We have the following key base-change result for $\Ext^1$'s of 
Breuil--Kisin modules with descent data.

\begin{prop}
\label{prop:base-change for exts}
Suppose that $\gM$ and $\gN$ are Breuil--Kisin modules with
descent data and coefficients in a $\cO/\varpi^a$-algebra $A$.
Then for any $A$-algebra $B$, and for any $B$-module $Q$,
there are natural isomorphisms
$\Ext^1_{\K{A}}(\gM,\gN)\otimes_A Q \iso
\Ext^1_{\K{B}}(\gM\cotimes_A B, \gN\cotimes_A B) \otimes_B Q
\iso 
\Ext^1_{\K{B}}(\gM\cotimes_A B,\gN\cotimes_A Q).$
\end{prop}
\begin{proof}
We first prove the lemma in the case of an $A$-module $Q$.
It follows from Lemmas \ref{lem: C computes Ext^1}
and~\ref{lem:truncation argument used to prove f.g. of Ext Q version}
that we may compute
$\Ext^1_{\K{A}}(\gM,\gN)$
as the cokernel of the morphism
$$C^0(\gN)/v^N C^0(\gN)
\buildrel \delta \over \longrightarrow C^1(\gN)/v^N C^1(\gN),$$for
some sufficiently large value of $N$ (not depending on $\gN$),
and hence that we may compute
$\Ext^1_{\K{A}}(\gM,\gN)\otimes_A Q$
as the cokernel of the morphism
$$\bigl(C^0(\gN)/v^N C^0(\gN)\bigr) \otimes_A Q
\buildrel \delta \over \longrightarrow 
\bigl(C^1(\gN)/v^N C^1(\gN)\bigr) \otimes_A Q.$$
We may similarly compute
$\Ext^1_{\K{A}}(\gM,\gN\cotimes_A Q)$
as the cokernel of the morphism
$$C^0(\gN\cotimes_A Q)/v^N C^0(\gN\cotimes_A Q)
\buildrel \delta \over \longrightarrow
C^1(\gN\cotimes_A Q)/v^N C^1(\gN\cotimes_A Q).$$
(Remark~\ref{rem:completed tensor}~(2) shows
that $\gN\cotimes_A Q$ satisfies the necessary hypotheses
for Lemma~\ref{lem:truncation argument used to prove f.g. of Ext Q version}
to apply.)
Once we note that the natural morphism
$$\bigl( C^i(\gN)/v^N C^i(\gN)\bigr)\otimes_A Q \to C^i(\gN\cotimes_A Q)/v^N
C^i(\gN\cotimes_A Q)$$
is an isomorphism for $i = 0$ and $1$ (because $\gM$ is a finitely
generated projective $\gS_A$-module), 
we obtain the desired isomorphism
$$\Ext^1_{\K{A}}(\gM,\gN)\otimes_A Q \iso \Ext^1_{\K{A}}(\gM,\gN\cotimes_A Q).$$

If $B$ is an $A$-algebra, and $Q$ is a $B$-module,
then by Lemma~\ref{lem:base-change-complexes} 
there is a natural isomorphism
$$\Ext^1_{\K{A}}(\gM, \gN\cotimes_A Q) \iso
\Ext^1_{\K{B}}(\gM\cotimes_A B, \gN\cotimes_A Q);$$
combined with the preceding base-change result, this yields
one of our claimed isomorphisms, namely
\[\Ext^1_{\K{A}}(\gM,\gN)\otimes_A Q \iso 
\Ext^1_{\K{B}}(\gM\cotimes_A B, \gN\cotimes_A Q).\]
Taking $Q$ to be $B$ itself, we then obtain
an isomorphism
\[\Ext^1_{\K{A}}(\gM,\gN)\otimes_A B \iso 
\Ext^1_{\K{B}}(\gM\cotimes_A B, \gN\cotimes_A B).\]
This allows us to identify
$\Ext^1_{\K{A}}(\gM,\gN)\otimes_A Q$,
which is naturally isomorphic to 
$\bigr(\Ext^1_{\K{A}}(\gM,\gN)\otimes_A B\bigr) \otimes_B Q$,
with 
$\Ext^1_{\K{B}}(\gM\cotimes_A B, \gN\cotimes_A B)\otimes_B Q$,
yielding the second claimed isomorphism.
\end{proof}

In contrast to the situation for extensions
(\emph{cf}.\ Proposition~\ref{prop:base-change for exts}), the formation of
homomorphisms between Breuil--Kisin modules is in general
not compatible with arbitrary base-change, as the following example shows.

\begin{example}\label{example:rank one unramified}
Take $A = (\Z/p\Z)[x^{\pm 1}, y^{\pm 1}]$, and let $\gM_x$ be the
free Breuil--Kisin module of rank one and $A$-coefficients with $\varphi(e)
= xe$ for some generator $e$ of $\gM_x$. Similarly define $\gM_y$ with
$\varphi(e') = ye'$ for some generator $e'$ of $\gM_y$. Then
$\Hom_{\K{A}}(\gM_x,\gM_y)=0$.  On the other hand, if $B=A/(x-y)$ then
$\gM_x \cotimes_A B$ and $\gM_y \cotimes_A B$ are isomorphic, so that 
$\Hom_{\K{B}}(\gM_x \cotimes B, \gM_y \cotimes B) \not\cong
\Hom_{\K{A}}(\gM_x,\gM_y) \otimes_A B$. 
\end{example}

However, it is possible
to establish such a compatibility in some settings.
Corollary~\ref{cor:vanishing of homs non Noetherian}, which
gives a criterion for the vanishing of $\Hom_{\K{B}} (\gM\cotimes_A
B, \gN\cotimes_A B)$ for any $A$-algebra $B$,  is a first example of
a result in this direction. Lemma~\ref{lem: flat base change for
  Homs} deals with flat base change, and Lemma~\ref{lem: vanishing of Kisin module homs implies vanishing on dense open}, which will
be important in Section~\ref{subsec:universal
families}, proves that formation of
homomorphisms is  compatible with base-change over a dense open
subscheme of $\Spec A$. 

\begin{prop}
\label{prop:vanishing of homs}
Suppose that $A$ is a Noetherian $\cO/\varpi^a$-algebra,
and that $\gM$ and $\gN$ are objects of $\K{A}$ that are
finitely generated over $\gS_A$ {\em (}or, equivalently,
over $A[[u]]${\em )}. Suppose also that~$\gN$ is a flat $\gS_A$-module.
Consider the following conditions:
\begin{enumerate}
\item
$\Hom_{\K{B}} (\gM\cotimes_A B, \gN\cotimes_A B) = 0$
for any finite type $A$-algebra $B$.
\item
$\Hom_{\K{\kappa(\m)}}\bigl(\gM\otimes_A \kappa(\mathfrak m),
\gN\otimes_A \kappa(\mathfrak m) \bigr) = 0$ 
for each maximal ideal $\mathfrak m$ of $A$.
\item
$\Hom_{\K{A}}(\gM, \gN\otimes_A Q) = 0$ 
for any 
finitely generated $A$-module $Q$.
\end{enumerate}
Then we have (1)$\implies$(2)$\iff$(3).  If $A$ is furthermore
Jacobson, then all three conditions are equivalent.
\end{prop}
\begin{proof}
If $\mathfrak m$ is a maximal ideal of $A$, then $\kappa(\mathfrak m)$
is certainly a finite type $A$-algebra, and so evidently~(1) implies~(2). 
It is even a finitely generated $A$-module, and so also~(2) follows
from~(3). 

We next
prove that~(2) implies~(3).
To this end, recall that if $A$ is any ring, and $M$ is any $A$-module,
then $M$ injects into the product of its localizations at all maximal ideals. 
If $A$ is Noetherian, and $M$ is finitely generated, then, by combining
this fact with the Artin--Rees
Lemma, we see that $M$ embeds into the product of its completions at all 
maximal ideals.   Another way to express this is that, if $I$ runs
over all cofinite length ideals in $A$ (i.e.\ all ideals for which $A/I$
is finite length), then $M$ embeds into the projective limit 
of the quotients $M/IM$  (the point being that
this projective limit is the same as the product
over all $\mathfrak m$-adic completions).
We are going to apply this observation with $A$ replaced by $\gS_A$,
and with $M$ taken to be $\gN\otimes_A Q$ for some finitely generated
$A$-module $Q$.

In $A[[u]]$, one sees that $u$ lies in the Jacobson radical (because 
$1  + fu$ is invertible in $A[[u]]$ for every $f \in A[[u]]$), and thus
in every maximal ideal, and so the maximal ideals of $A[[u]]$ are of
the form $(\mathfrak m, u)$, where $\mathfrak m$ runs over the maximal 
ideals of~$A$.
Thus the ideals of the form $(I,u^n)$, where $I$ is a cofinite length
ideal in $A$, are
cofinal in all cofinite length ideals in $A[[u]]$.
Since $\gS_A$ is finite over $A[[u]]$, we see that the ideals
$(I,u^n)$ in $\gS_A$ are also
cofinal in all cofinite length ideals in $A[[u]]$.
Since $A[[u]]$, and hence $\gS_A$, is furthermore Noetherian when $A$ is,
we see that if $Q$ is a
finitely generated $A$-module, and $\gN$ is a finitely generated
$\gS_A$-module,
then $\gN\otimes_A (Q/IQ)$ is $u$-adically complete,
for any cofinite length ideal $I$ in $A$, and
hence equal to the limit over $n$ of $\gN \otimes_A Q/(I,u^n)$.  
Putting this together with the observation of the preceding paragraph,
we see that the natural morphism
$$\gN\otimes_A Q \to \varprojlim_I \gN\otimes_A (Q/IQ)$$ 
(where $I$ runs over all cofinite length ideals of $A$)
is an embedding.
The induced morphism
$$ \Hom_{\K{A}}(\gM,\gN\otimes_A Q)
\to
\varprojlim_I \Hom_{\K{A}}(\gM,\gN\otimes_A (Q/IQ))$$
is then evidently also an embedding.

Thus, to conclude that 
$ \Hom_{\K{A}}(\gM,\gN\otimes_A Q)$
vanishes,
it suffices to show that
$\Hom_{\K{A}}(\gM,\gN\otimes_A (Q/IQ))$  vanishes for each
cofinite length ideal $I$ in $A$.  An easy induction (using the
flatness of~$\gN$) on the
length of $A/I$ reduces this to showing that
$\Hom_{\K{A}}\bigl(\gM,\gN\otimes_A \kappa(\mathfrak m)\bigr),$
or, equivalently, $\Hom_{\K{\kappa(\mathfrak{m})}}\bigl(\gM\otimes_A \kappa(\mathfrak m),
\gN\otimes_A \kappa(\mathfrak m)\bigr),$
vanishes for each maximal ideal~$\mathfrak m$. 
Since this is 
the hypothesis of~(2), we see that indeed~(2) implies~(3).

It remains to show that~(3) implies~(1) when $A$ is Jacobson. 
Applying the result
``(2) implies~(3)'' (with $A$ replaced by~$B$, and taking $Q$ in~(3) to be $B$ itself as a $B$-module) to $\gM\cotimes_A B$ and $\gN\cotimes_A B$,
we see that it suffices to prove the vanishing of
$$\Hom_{\K{B}}\bigl( (\gM\cotimes_A B)\otimes_B \kappa(\mathfrak n),
(\gN\cotimes_A B)\otimes_B \kappa(\mathfrak n) \bigr)
= \Hom_{\K{A}}\bigl( \gM, \gN\cotimes_A \kappa(\mathfrak n) \bigr)
$$
for each maximal ideal $\mathfrak n$ of $B$.
Since $A$ is Jacobson, the field $\kappa(\mathfrak n)$ is in fact a
finitely generated
$A$-module, hence $\gN\cotimes\kappa(\mathfrak n) = \gN\otimes_A
\kappa(\mathfrak n)$, and so the desired vanishing is a special case of~(3).
\end{proof}

\begin{cor}
\label{cor:vanishing of homs non Noetherian}
If $A$ is a Noetherian and Jacobson $\cO/\varpi^a$-algebra,
and if $\gM$ and $\gN$ are Breuil--Kisin modules with descent
data and $A$-coefficients, 
then the
following three conditions are equivalent:
\begin{enumerate}
\item
$\Hom_{\K{B}} (\gM\cotimes_A B, \gN\cotimes_A B) = 0$
for any $A$-algebra $B$.
\item
$\Hom_{\K{\kappa(\mathfrak{m})}}\bigl(\gM\otimes_A \kappa(\mathfrak m),
\gN\otimes_A \kappa(\mathfrak m) \bigr) = 0$ 
for each maximal ideal $\mathfrak m$ of $A$.
\item
$\Hom_{\K{A}}(\gM, \gN\otimes_A Q) = 0$ 
for any 
finitely generated $A$-module $Q$.
\end{enumerate}
\end{cor}
\begin{proof}By Proposition~\ref{prop:vanishing of homs}, we need only
  prove that if $\Hom_{\K{B}} (\gM\cotimes_A B, \gN\cotimes_A B)$
  vanishes 
for all finitely generated $A$-algebras~$B$, then it vanishes for all
$A$-algebras~$B$. This is immediate from Proposition~\ref{prop:descent for Homs of free Kisin modules}.
\end{proof}

\begin{cor}
\label{cor:freeness for exts}
Suppose that $\gM$ and $\gN$ are Breuil--Kisin modules with
descent data and coefficients in a Noetherian $\cO/\varpi^a$-algebra $A$,
and that furthermore
$\Hom_{\K{A}}\bigl(\gM\otimes_A \kappa(\mathfrak m),
\gN\otimes_A \kappa(\mathfrak m) \bigr)$ vanishes
for each maximal ideal $\mathfrak m$ of $A$.
Then the $A$-module
$\Ext^1_{\K{A}}(\gM,\gN)$
is projective
of finite rank. 
\end{cor}
\begin{proof}
By Proposition~\ref{prop:exts are f.g. over A},
in order to prove that $\Ext^1_{\K{A}}(\gM,\gN)$
is projective of finite rank over $A$,
it suffices to prove that it is flat over $A$.
For this, it suffices to show that
$Q \mapsto \Ext^1_{\K{A}}(\gM,\gN)\otimes_A Q$
is exact when applied to finitely generated $A$-modules $Q$.
Proposition~\ref{prop:base-change for exts} (together with 
Remark~\ref{rem:completed tensor}~(1)) allows us to identify
this functor with the functor
$Q \mapsto \Ext^1_{\K{A}}(\gM,\gN\otimes_A Q).$
Note that the functor $Q\mapsto\gN\otimes_A Q$ is an exact functor of $Q$,
since $\gS_A$ is a flat $A$-module (as $A$ is Noetherian; see Remark~\ref{rem:projectivity for Kisin modules}(3)).
Thus, taking into account
Corollary~\ref{cor:ext2 vanishes},
we see that it suffices to show that
$\Hom_{\K{A}}(\gM,\gN\otimes_A Q) = 0$
for each finitely generated $A$-module~$Q$,
under the hypothesis that
$\Hom_{\K{A}}\bigl(\gM\otimes_A \kappa(\mathfrak m),
\gN\otimes_A \kappa(\mathfrak m) \bigr) = 0$ 
for each maximal ideal $\mathfrak m$ of~$A$.
This is the implication (2) $\implies$ (3) of  Proposition~\ref{prop:vanishing of homs}.
\end{proof}

\begin{lemma}\label{lem: flat base change for Homs}
Suppose that $\gM$ is a Breuil--Kisin modules with
descent data and coefficients in a Noetherian $\cO/\varpi^a$-algebra
$A$. Suppose that $\gN$ is either a Breuil--Kisin module with
$A$-coefficients, or that $\gN=\gN'/u^N\gN'$, where $\gN'$  a Breuil--Kisin module with
$A$-coefficients and $N\ge 1$. 
Then,
if $B$ is a finitely generated flat 
$A$-algebra, we have a natural isomorphism
\[\Hom_{\K{B}}(\gM\cotimes_{A} B, 
\gN\cotimes_{A} B) \iso \Hom_{\K{A}}(\gM,\gN)\otimes_{A}B.  \] 
\end{lemma}
\begin{proof}
   By Corollary~\ref{cor:complex computes Hom and Ext} and the
   flatness of~$B$,
  we have a left exact sequence
  \[0\to \Hom_{\K{A}}(\gM,\gN)\otimes_AB\to C^0(\gN)\otimes_AB\to
    C^1(\gN)\otimes_AB\]  and therefore (applying
Corollary~\ref{cor: base change completion for complex in free case}
to treat the case that $\gN$ is projective)  a left exact sequence
\[0\to \Hom_{\K{A}}(\gM,\gN)\otimes_AB\to C^0(\gN\cotimes_AB)\to
  C^1(\gN\cotimes_AB).\]
The result follows from Corollary~\ref{cor:complex computes Hom and
  Ext} and Lemma~\ref{lem:base-change-complexes}.
\end{proof}

\begin{lemma}\label{lem: vanishing of Kisin module homs implies vanishing on dense open}
Suppose that $\gM$ is a Breuil--Kisin module with
descent data and coefficients in a Noetherian $\cO/\varpi^a$-algebra
$A$ which is furthermore a domain.
Suppose also that $\gN$ is either a Breuil--Kisin module with
$A$-coefficients, or that $\gN=\gN'/u^N\gN'$, where $\gN'$  is a Breuil--Kisin module with
$A$-coefficients and $N\ge 1$. 
Then there is some nonzero $f\in
A$ with the following property: 
writing 
$\gM_{A_f}=\gM\cotimes_A A_f$ and $\gN_{A_f}=\gN\cotimes_A A_f$, then for any
finitely generated $A_f$-algebra $B$, and any finitely 
generated $B$-module $Q$, there are natural isomorphisms
\begin{multline*}
\Hom_{\K{A_f}}(\gM_{A_f},\gN_{A_f})\otimes_{A_f}Q \iso
\Hom_{\K{B}}(\gM_{A_f}\cotimes_{A_f} B, \gN_{A_f}\cotimes_{A_f} B)\otimes_B Q
\\
\iso
\Hom_{\K{B}}(\gM_{A_f}\cotimes_{A_f} B, \gN_{A_f}\cotimes_{A_f} Q).
\end{multline*}
\end{lemma}
\begin{proof}[Proof of Lemma~{\ref{lem: vanishing of Kisin module homs implies vanishing on dense open}}.]
Note that since $A$ is Noetherian, by Remark~\ref{rem:projectivity for
  Kisin modules}(3) we see that~$\gN$ is $A$-flat. 
 By Corollary~\ref{cor:complex computes Hom and Ext}
  we have an exact sequence
  \[0\to \Hom_{\K{A}}(\gM,\gN)\to C^0(\gN)\to C^1(\gN) \to
    \Ext^1_{\K{A}}(\gM,\gN)\to 0.\]
  Since by assumption $\gM$ is a projective $\gS_A$-module, and 
 $\gN$ is a flat
$A$-module, the $C^i(\gN)$ are also flat $A$-modules.

By Proposition~\ref{prop:exts are f.g. over A},
$\Ext^1_{\K{A}}(\gM,\gN)$ is a finitely generated $A$-module, so
by the generic freeness
theorem~\cite[\href{http://stacks.math.columbia.edu/tag/051R}{Tag
    051R}]{stacks-project} there is some nonzero $f\in A$ such that $\Ext^1_{\K{A}}(\gM,\gN)_f$ is
free over~$A_f$.

Since localisation is exact, we obtain an exact
sequence
\[0\to \Hom_{\K{A_f}}(\gM,\gN)_f \to C^0(\gN)_f \to C^1(\gN)_f \to
  \Ext^1_{\K{A}}(\gM,\gN)_f\to 0\]and therefore (applying
Corollary~\ref{cor: base change completion for complex in free case}
to treat the case that $\gN$ is a Breuil--Kisin module) an exact sequence
\[0\to \Hom_{\K{A_f}}(\gM_{A_f},\gN_{A_f})\to C^0(\gN_{A_f})\to C^1(\gN_{A_f}) \to
  \Ext^1_{\K{A}}(\gM,\gN)_f\to 0.\] 

Since the last three terms are flat over~$A_f$, this sequence remains
exact upon tensoring over $A_f$
with $Q$.
Applying  Corollary~\ref{cor: base change completion for complex in free case}
again to treat the case that $\gN$ is a Breuil--Kisin module, we see that in particular we
have a left exact sequence
\[0\to
\Hom_{\K{A_f}}(\gM_{A_f},\gN_{A_f})\otimes_{A_f}Q
\to
  C^0(\gN_{A_f}\cotimes_{A_f}Q)\to C^1(\gN_{A_f}\cotimes_{A_f}Q),\]
and Corollary~\ref{cor:complex computes Hom and Ext} together with Lemma~\ref{lem:base-change-complexes}
yield one of the desired isomorphisms, namely
$$\Hom_{\K{A_f}}(\gM_{A_f},\gN_{A_f})\otimes_{A_f}Q \iso 
\Hom_{\K{B}}(\gM_{A_f}\cotimes_{A_f}B ,\gN_{A_f}\cotimes_{A_f} Q).$$
If we consider the case when $Q = B$, we  obtain an isomorphism
$$\Hom_{\K{A_f}}(\gM_{A_f},\gN_{A_f})\otimes_{A_f}B \iso 
\Hom_{\K{B}}(\gM_{A_f}\cotimes_{A_f}B ,\gN_{A_f}\cotimes_{A_f} B).$$
Rewriting the tensor product $\text{--}\otimes_{A_f} Q $ as
$\text{--}\otimes_{A_f} B \otimes_B Q,$
we then find that 
$$
\Hom_{\K{B}}(\gM_{A_f}\cotimes_{A_f}B ,\gN_{A_f}\cotimes_{A_f} B)\otimes_B Q
\iso
\Hom_{\K{B}}(\gM_{A_f}\cotimes_{A_f}B ,\gN_{A_f}\cotimes_{A_f} Q),$$
which gives the second desired isomorphism.
\end{proof}

Variants on the preceding result may be proved using other
versions of the generic freeness theorem.

\begin{example}\label{example:rank one unramified redux} Returning to
  the setting of
  Example~\ref{example:rank one unramified},
 one can check using Corollary~\ref{cor:vanishing of homs non
  Noetherian} that the conclusion of Lemma~\ref{lem: vanishing of
  Kisin module homs implies vanishing on dense open} (for $\gM =
\gM_x$ and $\gN = \gM_y)$  holds with $f =
x-y$. In this case all of the resulting $\Hom$ groups
vanish (\emph{cf}.\ also the proof of Lemma~\ref{lem: generically no Homs}).
It then follows from
Corollary~\ref{cor:freeness for exts} that
$\Ext^1_{\K{A}}(\gM,\gN)_{f}$ is projective over $A_f$, so that the proof
of Lemma~\ref{lem: vanishing of Kisin module homs implies vanishing on
  dense open} even goes through with this choice of $f$.
\end{example}

As well as considering homomorphisms and extensions of Breuil--Kisin modules, we need to
consider the homomorphisms and extensions of their associated \'etale $\varphi$-modules;
recall that the passage to associated \'etale $\varphi$-modules amounts
to inverting $u$, and so we briefly discuss this process in the general
context of the category $\K{A}$.

We let $\K{A}[1/u]$ denote the full subcategory of $\K{A}$
consisting of objects on which multiplication by $u$ is invertible.
We may equally well regard it as the category of left
$\gS_A[1/u][F,\Gal(K'/K)]$-modules (this notation being interpreted in
the evident manner).   
There are natural isomorphisms (of bi-modules)
\numequation
\label{eqn:left tensor iso}
\gS_A[1/u]\otimes_{\gS_A} \gS_A[F,\Gal(K'/K)] 
\iso \gS_A[1/u][F,\Gal(K'/K)]
\end{equation}
and
\numequation
\label{eqn:right tensor iso}
\gS_A[F,\Gal(K'/K)] \otimes_{\gS_A} \gS_A[1/u]
\iso \gS_A[1/u][F,\Gal(K'/K)].
\end{equation}
Thus (since $\gS_A \to \gS_A[1/u]$ is a flat morphism of commutative rings)
the morphism of rings $\gS_A[F,\Gal(K'/K)] \to \gS_A[1/u][F,\Gal(K'/K)]$
is both left and right flat.

If $\gM$ is an object of $\K{A}$, then we see from~(\ref{eqn:left tensor
iso}) that 
$\gM[1/u] := \gS_A[1/u]\otimes_{\gS_A} \gM \iso \gS_A[1/u][F,\Gal(K'/K)]
\otimes_{\gS_A[F,\Gal(K'/K)]} \gM$ is naturally an object
of $\K{A}[1/u]$.   Our preceding remarks about flatness show
that $\gM \mapsto \gM[1/u]$ is an exact functor $\K{A}\to \K{A}[1/u]$.

\begin{lemma}\label{lem:ext-i-invert-u}\leavevmode
\begin{enumerate}
\item If $M$ and $N$ are objects
of $\K{A}[1/u]$, then 
there is a natural isomorphism
$$\Ext^i_{\K{A}[1/u]}(M,N) \iso \Ext^i_{\K{A}}(M,N).$$
\item
If $\gM$ is an object of $\K{A}$ and $N$ is an object of $\K{A}[1/u]$,
then there is a natural isomorphism
$$\Ext^i_{\K{A}}(\gM,N) \iso \Ext^i_{\K{A}}(\gM[1/u],N),$$
for all $i\geq 0$.
\end{enumerate}
\end{lemma}
\begin{proof}
The morphism of~(1) can be understood in various ways; for example,
by thinking in terms of Yoneda Exts, and recalling that $\K{A}[1/u]$
is a full subcategory of $\K{A}.$   If instead we think in terms
of projective resolutions, we can begin with a projective resolution
$\gP^{\bullet} \to M$ in $\K{A}$, and then consider the induced
projective resolution $\gP^{\bullet}[1/u]$ of $M[1/u]$.  Noting 
that $M[1/u] \iso M$ for any object $M$ of $\K{A}[1/u]$,
we then find (via tensor adjunction) that $\Hom_{\K{A}}(\gP^{\bullet},
N) \iso \Hom_{\K{A}[1/u]}(\gP^{\bullet}[1/u], N)$,
which induces the desired isomorphism of $\Ext$'s by passing to 
cohomology.

Taking into account the isomorphism of~(1), the claim of~(2) is a general
fact about tensoring over a flat ring map (as can again be seen by
considering projective resolutions). 
\end{proof}

\begin{remark}
The preceding lemma is fact an automatic consequence of the abstract 
categorical properties of our situation:\ the functor $\gM \mapsto \gM[1/u]$
is left adjoint to the inclusion $\K{A}[1/u] \subset\K{A},$
and restricts to (a functor naturally equivalent to) the identity functor
on $\K{A}[1/u]$.
\end{remark}
The following lemma expresses the Hom between \'etale $\varphi$-modules
arising from Breuil--Kisin modules in terms
of a certain direct limit.

\begin{lem} 
  \label{lem:computing Hom as direct limit}Suppose that  $\gM$ is a
   Breuil--Kisin module with descent data in a Noetherian $\cO/\varpi^a$-algebra~$A$, and that~$\gN$ is an object of $\K{A}$ which is
   finitely generated and $u$-torsion free as an
   $\gS_A$-module. 
   Then there is a natural isomorphism
\[ \varinjlim_i\Hom_{\K{A}}(u^i\gM,\gN) \iso
	\Hom_{\K{A}[1/u]}(\gM[1/u],\gN[1/u]),\]
where the transition maps are induced by the inclusions $u^{i+1} \gM 
\subset u^i \gM$.
\end{lem}
\begin{rem}
  \label{rem: maps in direct limit are injections}Note that since
  $\gN$ is $u$-torsion free, the transition maps in the colimit are
  injections, so the colimit is just an increasing union.
\end{rem}
\begin{proof}There are compatible injections $\Hom_{\K{A}}(u^i\gM,\gN) \to
	\Hom_{\K{A}[1/u]}(\gM[1/u],\gN[1/u])$, taking $f'\in
        \Hom_{\K{A}}(u^i\gM,\gN)$ to $f\in\Hom_{\K{A}}(\gM,\gN[1/u])$
        where $f(m)=u^{-i}f'(u^im)$. Conversely, given
        $f\in\Hom_{\K{A}}(\gM,\gN[1/u])$, there is some~$i$ such that
        $f(\gM)\subset u^{-i}\gN$, as required.
\end{proof}
We have the following analogue of Proposition~\ref{prop:vanishing of homs}.
\begin{cor}
  \label{cor:vanishing homs with u inverted}Suppose that  $\gM$ and
  $\gN$ are Breuil--Kisin modules with descent data in a Noetherian $\cO/\varpi^a$-algebra~$A$. 
Consider the following conditions:
\begin{enumerate}
\item
$\Hom_{\K{B}[1/u]} \bigl((\gM\cotimes_A B)[1/u], (\gN\cotimes_A B)[1/u]\bigr) = 0$
for any finite type $A$-algebra $B$.
\item
$\Hom_{\K{\kappa(\m)}[1/u]}\bigl((\gM\otimes_A \kappa(\mathfrak m))[1/u],
(\gN\otimes_A \kappa(\mathfrak m))[1/u] \bigr) = 0$ 
for each maximal ideal $\mathfrak m$ of $A$.
\item
$\Hom_{\K{A}[1/u]}\bigl(\gM[1/u], (\gN\otimes_A Q)[1/u]\bigr) = 0$ 
for any 
finitely generated $A$-module $Q$.
\end{enumerate}
Then we have (1)$\implies$(2)$\iff$(3).  If $A$ is furthermore
Jacobson, then all three conditions are equivalent.
\end{cor}
\begin{proof}
By Lemma~\ref{lem:computing Hom as direct limit}, the three conditions
are respectively equivalent to the following conditions. 
\begin{enumerate}[label=(\arabic*$'$)]
\item
$\Hom_{\K{B}} \bigl(u^i(\gM\cotimes_A B), \gN\cotimes_A B\bigr) = 0$
for any finite type $A$-algebra $B$ and all $i\ge 0$.
\item
$\Hom_{\K{\kappa(\m)}}\bigl(u^i(\gM\otimes_A \kappa(\mathfrak m)),
\gN\otimes_A \kappa(\mathfrak m) \bigr) = 0$ 
for each maximal ideal $\mathfrak m$ of $A$ and all $i\ge 0$.
\item
$\Hom_{\K{A}}\bigl(u^i\gM, \gN\otimes_A Q\bigr) = 0$ 
for any 
finitely generated $A$-module $Q$ and all $i\ge 0$.
\end{enumerate}
Since $\gM$ is projective, the first two conditions are in turn
equivalent to
\begin{enumerate}[label=(\arabic*$''$)]
\item
$\Hom_{\K{B}} \bigl((u^i\gM)\cotimes_A B, \gN\cotimes_A B\bigr) = 0$
for any finite type $A$-algebra $B$ and all $i\ge 0$.
\item
$\Hom_{\K{\kappa(\m)}}\bigl((u^i\gM)\otimes_A \kappa(\mathfrak m),
\gN\otimes_A \kappa(\mathfrak m) \bigr) = 0$ 
for each maximal ideal $\mathfrak m$ of $A$ and all $i\ge 0$.
\end{enumerate}
The result then follows from Proposition~\ref{prop:vanishing of homs}.
\end{proof}

\begin{df}\label{def:kext}
If $\gM$ and $\gN$ are objects of $\K{A}$, then we define
$$\kExt^1_{\K{A}}(\gM,\gN)
:=
\ker\bigl(\Ext^1_{\K{A}}(\gM,\gN)\to\Ext^1_{\K{A}}(\gM[1/u],\gN[1/u])\bigr).$$  
The point of this definition is to capture, in the setting of
Lemma~\ref{lem: Galois rep is a functor if A is actually finite local}, the non-split extensions
of Breuil--Kisin modules whose underlying extension of Galois
representations is split. 
\end{df}

Suppose now that $\gM$ is a Breuil--Kisin module. 
The exact sequence in~$\K{A}$ \[0\to\gN\to \gN[1/u]\to\gN[1/u]/\gN\to 0\]
gives an exact sequence of complexes \[\xymatrix{0\ar[r]&
  C^0(\gN)\ar[d]\ar[r]&
  C^0(\gN[1/u])\ar[d]\ar[r]&C^0(\gN[1/u]/\gN)\ar[d]\ar[r]&0\\ 0\ar[r]&
  C^1(\gN)\ar[r]&
  C^1(\gN[1/u])\ar[r]&C^1(\gN[1/u]/\gN)\ar[r]&0. } \] It follows from
Corollary~\ref{cor:complex computes Hom and Ext},
Lemma~\ref{lem:ext-i-invert-u}(2), and 
the snake lemma that we have an exact
sequence  \numequation\label{eqn: computing kernel of Ext groups}\begin{split}0\to\Hom_{\K{A}}(\gM,\gN)\to\Hom_{\K{A}}(\gM,\gN[1/u])
\qquad \qquad \\
\to\Hom_{\K{A}}(\gM,\gN[1/u]/\gN) \to \kExt^1_{\K{A}}(\gM,\gN)\to
0.\end{split}\end{equation}

\begin{lem}\label{lem: bound on torsion in kernel of Exts}If $\gM$, $\gN$ are Breuil--Kisin modules with descent data and
  coefficients in a Noetherian $\cO/\varpi^a$-algebra~$A$, and $\gN$ has
  height at most~$h$, then $f(\gM)$ 
is killed by $u^i$ for any $f \in \Hom_{\K{A}}(\gM,\gN[1/u]/\gN)$ and
any $i \ge \lfloor e'ah/(p-1) \rfloor$. 
  \end{lem}
  \begin{proof}
    Suppose that $f$ is an  element of
    $\Hom_{\K{A}}(\gM,\gN[1/u]/\gN)$. Then $f(\gM)$ is a finitely
    generated submodule of $\gN[1/u]/\gN$, and it therefore killed
    by~$u^i$ for some $i\ge 0$. Choosing~$i$ to be the exponent of
    $f(\gM)$ (that is, choosing $i$ to be minimal), it follows
    that
    $(\varphi^*f)(\varphi^*\gM)$ has exponent 
    precisely~$ip$. (From the choice of $i$, we see that $u^{i-1}
    f(\gM)$ is nonzero but killed by $u$, i.e., it is just a $W(k')
    \otimes A$-module, and so its pullback by~$\varphi:\gS_A\to\gS_A$
    has exponent precisely $p$.  Then by the flatness  
   of~$\varphi:\gS_A\to\gS_A$  we have
    $u^{ip-1}(\varphi^*f)(\varphi^*\gM)=u^{p-1}\varphi^*(u^{i-1}
    f(\gM)) \neq 0$.)

We claim that $u^{i+e'ah}(\varphi^*f)(\varphi^*\gM)=0$; admitting this, we
deduce that $i+e'ah\ge ip$, as required. To see the claim, take
$x\in\varphi^*\gM$, so that $\Phi_{\gN}((u^i\varphi^*f)(x))=
u^if(\Phi_\gM(x))=0$. It is therefore enough to show that the kernel
of \[\Phi_{\gN}:\varphi^*\gN[1/u]/\varphi^*\gN\to \gN[1/u]/\gN\] is killed
by $u^{e'ah}$; but this follows immediately from an application of the
snake lemma to the commutative diagram \[\xymatrix{0\ar[r]&
  \varphi^*\gN\ar[r]\ar[d]_{\Phi_\gN}&\varphi^*\gN[1/u]\ar[r]\ar[d]_{\Phi_\gN}
&\varphi^*\gN[1/u]/\varphi^*\gN\ar[r]\ar[d]_{\Phi_\gN}&0
\\ 0\ar[r]&
  \gN\ar[r]&\gN[1/u]\ar[r]
&\gN[1/u]/\gN\ar[r]&0}
\]together with the assumption that $\gN$ has height at most~$h$ and
an argument as in the first line of the proof of
Lemma~\ref{lem:truncation argument used to prove f.g. of Ext Q version}.
  \end{proof}

\begin{lem}\label{lem:computing kernel of Ext groups finite level}If $\gM$, $\gN$ are Breuil--Kisin modules with descent data and
  coefficients in a Noetherian $\cO/\varpi^a$-algebra~$A$, and $\gN$ has
  height at most~$h$, then for any $i \ge \lfloor e'ah/(p-1) \rfloor$  we have an exact
  sequence 
 \[\begin{split}0\to\Hom_{\K{A}}(u^i\gM,u^i\gN)\to\Hom_{\K{A}}(u^i\gM,\gN) \qquad \qquad \\
 \to\Hom_{\K{A}}(u^i\gM,\gN/u^i\gN) \to \kExt^1_{\K{A}}(\gM,\gN)\to
0.\end{split}\]
  \end{lem}
\begin{proof} Comparing Lemma~\ref{lem: bound on torsion in kernel of
    Exts} with the proof of Lemma~\ref{lem:computing Hom as direct
    limit}, we see that the direct limit in that proof has stabilised
  at $i$, and we obtain an isomorphism $\Hom_{\K{A}}(\gM,\gN[1/u])
  \toisom \Hom_{\K{A}}(u^i \gM,\gN)$ sending a map $f$ to $f' : u^i m
  \mapsto u^i f(m)$.  The same formula evidently identifies
  $\Hom_{\K{A}}( \gM,\gN)$ with $\Hom_{\K{A}}(u^i\gM,u^i\gN)$ and
  $\Hom_{\K{A}}(\gM,\gN[1/u]/\gN)$ with $\Hom_{\K{A}}(u^i
  \gM,\gN[1/u]/u^i \gN)$. But any map in the latter group has image
 contained  in $\gN/u^i \gN$ (by Lemma~\ref{lem: bound on torsion in kernel of
    Exts} applied to  $\Hom_{\K{A}}(\gM,\gN[1/u]/\gN)$, together with
  the identification in the previous sentence), so that $\Hom_{\K{A}}(u^i
  \gM,\gN[1/u]/u^i \gN) = \Hom_{\K{A}}(u^i
  \gM,\gN/u^i \gN)$.
  \end{proof}

\begin{prop}
  \label{prop: base change for kernel of map to etale Ext}
Let $\gM$ and $\gN$ be Breuil--Kisin modules with descent data and
coefficients in a Noetherian $\cO/\varpi^a$-domain~$A$. 
Then there is some nonzero $f\in
A$ with the following property: if we write 
$\gM_{A_f}=\gM\cotimes_A A_f$ and $\gN_{A_f}=\gN\cotimes_A A_f$, then if~$B$
is any finitely generated $A_f$-algebra, and if $Q$ is any finitely
generated $B$-module, we have natural isomorphisms
\begin{multline*}
\kExt^1_{\K{A_f}}(\gM,\gN)\otimes_{A_f}Q \iso
\kExt^1_{\K{A_f}}(\gM_{A_f}\cotimes_{A_f} B, 
\gN\cotimes_{A_f} B)\otimes_B Q
\\
\iso
\kExt^1_{\K{A_f}}(\gM_{A_f}\cotimes_{A_f} B, 
\gN\cotimes_{A_f} Q).
\end{multline*}
\end{prop}
\begin{proof}In view of Lemma~\ref{lem:computing kernel of Ext groups
    finite level}, this follows from  Lemma~\ref{lem: vanishing of
    Kisin module homs implies vanishing on dense open}, 
  with $\gM$ there being our $u^i\gM$, and $\gN$ being
  each of $\gN$,  $\gN/u^i\gN$ in turn.
\end{proof}

The following result will be crucial in our investigation of the
decomposition of $\cC^{\dd,1}$ and $\cR^{\dd,1}$ into
irreducible components.

\begin{prop}
  \label{prop: we have vector bundles}
Suppose that $\gM$ and $\gN$ are Breuil--Kisin modules with
descent data and coefficients in a Noetherian $\cO/\varpi^a$-algebra $A$
which is furthermore a domain,
and suppose that
$\Hom_{\K{A}}\bigl(\gM\otimes_A \kappa(\mathfrak m),
\gN\otimes_A \kappa(\mathfrak m) \bigr)$ vanishes
for each maximal ideal $\mathfrak m$ of $A$.
Then there is some nonzero $f\in
A$ with the following property: if we write 
$\gM_{A_f}=\gM\cotimes_A A_f$ and $\gN_{A_f}=\gN\cotimes_A A_f$, then 
for any finitely generated $A_f$-algebra $B$,
each of $\kExt^1_{\K{B}}(\gM_{A_f}\cotimes_{A_f} B,\gN_{A_f}\cotimes_{A_f}B)$,
$\Ext^1_{\K{B}}(\gM_{A_f}\cotimes_{A_f} B,\gN_{A_f}\cotimes_{A_f}B)$,
and
$$\Ext^1_{\K{B}}(\gM_{A_f}\cotimes_{A_f}B ,\gN_{A_f}\cotimes_{A_f}B)/
\kExt^1_{\K{A_f}}(\gM_{A_f}\cotimes_{A_f} B,\gN_{A_f}\cotimes_{A_f} B)$$
is a finitely generated projective $B$-module.
\end{prop}
\begin{proof}
Choose $f$ as in 
  Proposition~\ref{prop: base change for kernel of map to etale Ext},
let $B$ be a finitely generated $A_f$-algebra,
and let $Q$ be a finitely generated $B$-module.
By Propositions~\ref{prop:base-change for exts} and~\ref{prop: base
    change for kernel of map to etale Ext}, the morphism
  \[\kExt^1_{\K{B}}(\gM_{A_f}\cotimes_{A_f} B,\gN_{A_f}\cotimes_{A_f} B)\otimes_B Q\to
\Ext^1_{\K{B}}(\gM_{A_f}\cotimes_{A_f} B,\gN_{A_f}\cotimes_{A_f}B)\otimes_B Q\]
  is naturally identified with the morphism
  \[\kExt^1_{\K{B}}(\gM_{A_f}\cotimes_{A_f} B,\gN_{A_f}\cotimes_{A_f} Q)\to
\Ext^1_{\K{B}}(\gM_{A_f}\cotimes_{A_f} B,\gN_{A_f}\cotimes_{A_f} Q);\]
  in particular, it is injective. 
By Proposition~\ref{prop:base-change for exts} and
Corollary~\ref{cor:freeness for exts} we see that
  $\Ext^1_{\K{B}}(\gM_{A_f}\cotimes_{A_f} B,\gN_{A_f}\cotimes_{A_f} B)$ is a
finitely generated projective $B$-module; hence it is also flat.
Combining this with the injectivity just proved, we find that
  \[\Tor^1_B\bigl(Q, \Ext^1_{\K{B}}(\gM\cotimes_{A_f}B,\gN_{A_f}\cotimes_{A_f} B)/
\kExt^1_{\K{B}}(\gM_{A_f}\cotimes_{A_f} B,\gN_{A_f}\cotimes_{A_f} B)\bigr)=0\]
  for every finitely generated $B$-module $Q$, and thus that
  $$\Ext^1_{\K{B}}(\gM_{A_f}\cotimes_{A_f}B,\gN_{A_f}\cotimes_{A_f}B)/
\kExt^1_{\K{B}}(\gM_{A_f}\cotimes_{A_f}B,\gN_{A_f}\cotimes_{A_f}B)$$
 is a finitely generated flat,
and therefore finitely generated projective, $B$-module.
Thus $\kExt^1_{\K{B}}(\gM_{A_f}\cotimes_{A_f}B,\gN_{A_f}\cotimes_{A_f}B)$
is a direct summand of
  the finitely generated projective $B$-module
$\Ext^1_{\K{B}}(\gM_{A_f}\cotimes_{A_f} B,\gN_{A_f}\cotimes_{A_f}B)$, and so is
  itself a finitely generated projective $B$-module.
\end{proof}

\subsection{Families of extensions}
\label{subsec:families of extensions}
Let $\gM$ and $\gN$ be Breuil--Kisin modules with descent data
and $A$-coefficients, so that
$\Ext^1_{\K{A}}(\gM,\gN)$ is an $A$-module.
Suppose that $\psi: V \to \Ext^1_{\K{A}}(\gM,\gN)$ is
a homomorphism of $A$-modules whose source is a projective $A$-module of 
finite rank.
Then we may regard $\psi$ as an element of
$$\Ext^1_{\K{A}}(\gM,\gN)\otimes_A V^{\vee} =
\Ext^1_{\K{A}}(\gM, \gN\otimes_A V^{\vee} ),$$ 
and in this way $\psi$ 
corresponds to an extension
\numequation
\label{eqn:universal extension}
0 \to \gN\otimes_A V^{\vee} \to \gE \to \gM \to 0,
\end{equation}
which we refer to as the {\em family of extensions} of $\gM$ by $\gN$
parametrised by $V$ (or by $\psi$, if we want to emphasise our
choice of homomorphism).  We let $\gE_v$ denote the pushforward of $\gE$ under
the morphism $\gN\otimes_A V^{\vee} \to \gN$
given by evaluation on $v\in V$.
In the special case that 
$\Ext^1_{\K{A}}(\gM,\gN)$ itself is a projective $A$-module of finite rank,
we can let $V$ be $\Ext^1_{\K{A}}(\gM,\gN)$ and take $\psi$ be the
identity map;
in this case we refer to~(\ref{eqn:universal extension}) as
the {\em universal extension} of $\gM$ by $\gN$.
The reason for this terminology is as follows:
if $v \in \Ext^1_{\K{A}}(\gM,\gN)$, 
then $\gE_v$ is the extension of $\gM$ by $\gN$ corresponding to the element
$v$. 

Let $B := A[V^{\vee}]$ denote the symmetric algebra over $A$ generated by
$V^{\vee}$.  
The short exact sequence~(\ref{eqn:universal extension}) is a
short exact sequence of Breuil--Kisin modules with descent data,
and so forming its $u$-adically completed tensor product with $B$ over $A$,
we obtain a short exact sequence 
$$
0 \to (\gN\otimes_A V^{\vee}) \cotimes_A B \to \gE\cotimes_A B \to
\gM\cotimes_A B
\to 0$$
of Breuil--Kisin modules with descent data over $B$ (see Lemma~\ref{rem: base change of locally free Kisin module is a
    locally free Kisin module}).
Pushing this short exact sequence forward under the natural map
$$V^{\vee} \cotimes_A B = V^{\vee} \otimes_A B \to B$$
induced by the inclusion of $V^{\vee}$ in $B$
and the multiplication map $B\otimes_A B \to B$,
we obtain a short exact sequence
\numequation
\label{eqn:geometric universal extension}
0 \to \gN\cotimes_A B \to \widetilde{\gE} \to \gM\cotimes_A B \to 0
\end{equation}
of Breuil--Kisin modules with descent data over $B$,
which we call the {\em family of extensions} of $\gM$  by $\gN$
parametrised by $\Spec B$ (which we note is (the total space of) the
vector bundle over $\Spec A$ corresponding to the projective
$A$-module~$V$).
 
If $\alpha_v: B \to A$ is the morphism induced
by the evaluation map
$V^{\vee} \to A$ given by some element $v \in V$,
then base-changing~(\ref{eqn:geometric universal extension}) by $\alpha_v$,
we recover the short exact sequence
$$0 \to \gN \to \gE_v \to \gM \to 0.$$
More generally, suppose that $A$ is a $\cO/\varpi^a$-algebra for some
$a\ge 1$, and let  $C$ be any $A$-algebra. Suppose that $\alpha_{\tilde{v}}:
B \to C$ is the morphism induced 
by the evaluation map
$V^{\vee} \to C$ corresponding to some element $\tilde{v} \in C\otimes_A V$.
Then base-changing~(\ref{eqn:geometric universal extension}) by
$\alpha_{\tilde{v}}$
yields a short exact sequence
$$0 \to \gN\cotimes_A C \to \widetilde{\gE}\cotimes_B C
\to \gM\cotimes_A C \to 0,$$
whose associated extension class corresponds
to the image of $\tilde{v}$ 
under the natural morphism
$C\otimes_A V \to C\otimes_A \Ext^1_{\K{A}}(\gM,
\gN) \cong \Ext^1_{\K{C}}(\gM\cotimes_A C, \gN\otimes_A C),$
the first arrow being induced by $\psi$
and the second arrow being the isomorphism of Proposition~\ref{prop:base-change
for exts}. 

\subsubsection{The functor represented by a universal family}
We now suppose that the ring~$A$ and the Breuil--Kisin modules $\gM$ and $\gN$ have the following 
properties:

\begin{assumption}
\label{assumption:vanishing}Let $A$ be a  Noetherian and Jacobson
$\cO/\varpi^a$-algebra for some $a\ge 1$, and assume that for each maximal ideal $\mathfrak m$ of $A$, we have that
$$\Hom_{\K{\kappa(\mathfrak{\m})}}\bigl(\gM\otimes_A \kappa(\mathfrak m) , \gN\otimes_A
\kappa(\mathfrak m)\bigr) = \Hom_{\K{\kappa(\mathfrak{\m})}}\bigl(\gN\otimes_A \kappa(\mathfrak m) , \gM\otimes_A
\kappa(\mathfrak m)\bigr) = 0.$$
\end{assumption}

By Corollary~\ref{cor:freeness for exts}, this assumption implies in particular
that $V:= \Ext^1_{\K{A}}(\gM,\gN)$ is projective of finite rank,
and so we may form $\Spec B := \Spec A[V^{\vee}]$,
which parametrised the universal family of 
extensions. 
We are then able to give the following precise description
of the functor represented by $\Spec B$.

\begin{prop}\label{prop: the functor that Spec B represents}
The scheme $\Spec B$ represents the functor which,
to any $\cO/\varpi^a$-algebra $C$, associates the set of isomorphism
classes of tuples $(\alpha, \gE, \iota, \pi)$, where $\alpha$ is a morphism
$\alpha: \Spec C \to \Spec A$, $\gE$ is a Breuil--Kisin module
with descent data and coefficients in $C$, and $\iota$ and $\pi$ are morphisms
$\alpha^* \gN \to \gE$ and $\gE \to \alpha^* \gM$ respectively,
with the property that $0 \to \alpha^*\gN \buildrel \iota
\over \to \gE \buildrel \pi  \over \to \alpha^* \gM \to 0$
is short exact.
\end{prop}
\begin{proof}
We have already seen that giving a morphism $\Spec C \to \Spec B$
is equivalent to giving the composite morphism $\alpha:\Spec C \to \Spec B 
\to \Spec A$, together with an extension class
$[\gE] \in \Ext^1_{\K{C}}(\alpha^*\gM,\alpha^*\gN).$
Thus to prove the proposition, we just have to show that
any automorphism of $\gE$ which restricts to the identity on $\alpha^*\gN$
and induces the identity on $\alpha^*\gM$ is itself the identity on
$\gE$.   This follows from Corollary~\ref{cor:vanishing of homs non
  Noetherian}, 
together with Assumption~\ref{assumption:vanishing}. 
\end{proof}
Fix an integer $h\ge 0$ so that $E(u)^h\in
\Ann_{\gS_A}(\coker\Phi_{\gM})\Ann_{\gS_A}(\coker\Phi_{\gN})$, so
that by Lemma~\ref{lem:ext of a Kisin module is a Kisin module}, every
Breuil--Kisin module parametrised by $\Spec B$ has height at most~$h$.
There is a natural action of $\Gm\times_{\cO} \Gm$ on $\Spec B$,
given by rescaling each of $\iota$ and $\pi$.
There is also an evident forgetful morphism
$\Spec B \to \Spec A \times_{\cO} \cC^{\dd,a}$, 
given by forgetting $\iota$ and $\pi$, which is evidently invariant
under the $\Gm\times_{\cO} \Gm$-action. (Here and below,
$\cC^{\dd,a}$ denotes the moduli stack defined in Subsection~\ref{sec:recoll-from-citec} 
for our fixed choice of~$h$ and for
$d$ equal to the sum of the ranks of $\gM$ and $\gN$.)  We thus obtain a morphism
\numequation
\label{eqn:ext relation}
\Spec B \times_{\cO} \Gm\times_{\cO} \Gm \to
\Spec B \times_{\Spec A \times_{\cO} \cC^{\dd,a}} \Spec B.
\end{equation}

\begin{cor}\label{cor: monomorphism to Spec A times C}
Suppose that
$\Aut_{\K{C}}(\alpha^*\gM) = \Aut_{\K{C}}(\alpha^*\gN) = C^{\times}$
for any morphism $\alpha: \Spec C \to \Spec A$.
Then the morphism~{\em (\ref{eqn:ext relation})} is an isomorphism,
and consequently the induced morphism
$$[\Spec B/ \Gm\times_{\cO} \Gm] \to \Spec A \times_{\cO} \cC^{\dd,a}$$
is a finite type monomorphism.
\end{cor}
\begin{proof}
By Proposition~\ref{prop: the functor that Spec B represents}, a morphism  \[\Spec C\to \Spec B \times_{\Spec A
  \times_{\cO} \cC^{\dd,a}} \Spec B\]  corresponds to an isomorphism class
of tuples $(\alpha,\beta:\gE\to\gE',\iota,\iota',\pi,\pi')$, where
\begin{itemize}
\item $\alpha$ is a morphism
$\alpha:\Spec C\to\Spec A$,
\item  $\beta:\gE\to\gE'$ is an isomorphism of Breuil--Kisin modules
with descent data and coefficients in $C$, 
\item $\iota:\alpha^* \gN \to \gE$ and $\pi :\gE \to
  \alpha^* \gM$ are morphisms
with the property that $$0 \to \alpha^*\gN \buildrel \iota
\over \to \gE \buildrel \pi  \over \to \alpha^* \gM \to 0$$ is short
exact, 
\item $\iota':\alpha^* \gN \to \gE'$ and $\pi' :\gE' \to
  \alpha^* \gM$ are morphisms
with the property that 
  $$0 \to \alpha^*\gN \buildrel \iota'
\over \to \gE' \buildrel \pi'  \over \to \alpha^* \gM \to 0$$ is
short exact.
\end{itemize}
Assumption~\ref{assumption:vanishing} and Corollary~\ref{cor:vanishing of homs non Noetherian}  together show that 
$\Hom_{\K{C}}(\alpha^*\gN, \alpha^*\gM)~=~0$. It follows that the
composite
$\alpha^*\gN\stackrel{\iota}{\to}\gE\stackrel{\beta}{\to}\gE' $
factors through $\iota'$, and the induced endomorphism of
$\alpha^*\gN$ is injective. Reversing the roles of $\gE$ and $\gE'$,
we see that it is in fact an automorphism of $\alpha^*\gN$, and it
follows easily that $\beta$ also induces an automorphism of
$\alpha^*\gM$. Again, Assumption~\ref{assumption:vanishing} and Proposition~\ref{cor:vanishing of homs non Noetherian} together show that 
$\Hom_{\K{C}}(\alpha^*\gM, \alpha^*\gN) = 0$, from which it follows
easily that $\beta$ is determined by the automorphisms of
$\alpha^*\gM$ and $\alpha^*\gN$ that it induces.

 Since $\Aut_{\K{C}}(\alpha^*\gM) =
\Aut_{\K{C}}(\alpha^*\gN) = C^{\times}$ by assumption, we see that
$\beta \circ \iota, \iota'$ and $\pi,\pi'\circ\beta$ differ only by the action of
$\Gm\times_{\cO}\Gm$, so  the first claim of the corollary
follows. 
The claim regarding the monomorphism is immediate from
Lemma~\ref{lem: morphism from quotient stack is a monomorphism}
below. Finally, note that $[\Spec B/\Gm\times_{\cO} \Gm]$ is 
of finite type over $\Spec A$, while $\cC^{\dd,a}$ has finite type diagonal.
It follows that the morphism 
$[\Spec B / \Gm \times_{\cO} \Gm ] \rightarrow \Spec A\times_{\cO}\cC^{\dd,a}$
is of finite type, as required.
\end{proof} 

\begin{lem}
  \label{lem: morphism from quotient stack is a monomorphism}Let $X$
  be a scheme over a base scheme~$S$, let $G$ be a smooth affine group
  scheme over~$S$, and let $\rho:X\times_S G\to X$ be a {\em
    (}right{\em )} action of $G$
  on~$X$. Let $X\to\cY$ be a $G$-equivariant morphism, whose target is
  an algebraic stack over~$S$ on which $G$ acts trivially. 
Then the induced
  morphism \[[X/G]\to\cY\] is a monomorphism if and only if the
  natural morphism \[X\times_S G\to X\times_{\cY} X\] {\em (}induced by the
  morphisms $\pr_1,\rho:X\times_S G\to X${\em )} is an isomorphism.
\end{lem} 
\begin{proof}
  We have a Cartesian diagram as follows. \[\xymatrix{X\times_S G\ar[r]\ar[d]&
    X\times_\cY X\ar[d]\\ [X/G]\ar[r]& [X/G]\times_\cY[X/G]}\]The
  morphism $[X/G]\to\cY$ is a monomorphism if and only if the bottom horizontal
  morphism of this square is an isomorphism; since the right hand
  vertical arrow is a smooth surjection, this is the case if and only
  if the top horizontal morphism is an isomorphism, as required.
\end{proof}
\subsection{Families of extensions of rank one Breuil--Kisin modules}
\label{subsec:universal families}
In this section we construct universal families of extensions of rank
one Breuil--Kisin modules. We will use these rank two families to study our
moduli spaces of Breuil--Kisin modules, and the corresponding spaces of
\'etale $\varphi$-modules. We show how to compute the dimensions of
these universal families; in the subsequent sections, we will combine
these results with explicit calculations to determine the irreducible
components of our moduli spaces. In particular, we will show that each
irreducible component has a dense open substack given by a family of
extensions.

\subsubsection{Universal unramified twists}
Fix a free Breuil--Kisin module with descent data 
$\gM$ over $\F$, and write $\Phi_i$   for
$\Phi_{\gM,i}:\varphi^*(\gM_{i-1}) \to\gM_{i}$. (Here we are using the
notation of Section~\ref{subsec: kisin modules with dd}, so
that~$\gM_i=e_i\gM$ is cut out by the idempotent~$e_i$ of Section~\ref{subsec: notation}.) 
We will construct the ``universal unramified twist'' of $\gM$.

\begin{df}
\label{def:unramified twist}
If $\Lambda$ is an $\F$-algebra, and if $\lambda \in \Lambda^\times$,
then we define $\gM_{\Lambda,\lambda}$ to be the free Breuil--Kisin
module with descent data and $\Lambda$-coefficients whose underlying
$\gS_\Lambda[\Gal(K'/K)]$-module 
is equal to $\gM\cotimes_\F\Lambda$
(so the usual base change of $\gM$ to $\Lambda$),
and for which $\Phi_{\gM_{\Lambda,\lambda}}: \varphi^*\gM_{\Lambda,\lambda} \to \gM_{\Lambda,\lambda}$ is defined via the $f'$-tuple
$(\lambda \Phi_0,\Phi_1,\ldots,\Phi_{f'-1}).$
We refer to $\gM_{\Lambda,\lambda}$ as the \emph{unramified twist} of $\gM$ by $\lambda$ over $\Lambda$.

If $M$ is a free \'etale $\varphi$-module with descent data, then we
define $M_{\Lambda,\lambda}$ in the analogous fashion. If we write
$X=\Spec \Lambda$, then we will sometimes write $\gM_{X,\lambda}$
(resp.\ $M_{X,\lambda}$) for $\gM_{\Lambda,\lambda}$ (resp.\
$M_{\Lambda,\lambda}$).
\end{df}

As usual, we write $\Gm := \Spec \F[x^{\pm 1}].$ 
We may then 
form the rank one Breuil--Kisin module with descent data
$\gM_{\Gm,x},$ which is the universal instance of an unramified twist:
given $\lambda\in\Lambda^\times$, 
there is a corresponding morphism $\Spec\Lambda
\to \Gm$ determined by the requirement that
$x \in \Gamma(\Gm, \cO_{\Gm}^{\times})$ pulls-back to $\lambda$,
and $\gM_{X,\lambda}$ is obtained by pulling back
$\gM_{\Gm,x}$ under this morphism (that is, by base changing under the
corresponding ring homomorphism $\F[x^{\pm 1}]\to\Lambda$).

\begin{lemma}
\label{lem:rank one endos}
If $\gM_\Lambda$ is a Breuil--Kisin module of rank one with $\Lambda$-coefficients, then $\End_{\K{\Lambda}}(\gM) = \Lambda.$
Similarly, if $M_\Lambda$ is a \'etale $\varphi$-module of rank
one with $\Lambda$-coefficients, then
$\End_{\K{\Lambda}}(M_\Lambda) = \Lambda.$ 
\end{lemma}
\begin{proof}
We give the proof for~$M_\Lambda$, the argument for~$\gM_\Lambda$ being essentially
identical. One reduces easily to the case where $M_\Lambda$ is free. Since an endomorphism~$\psi$ of~$M_\Lambda$ is in particular an
endomorphism of the underlying $\gS_\Lambda[1/u]$-module, we see that there
is some $\lambda\in \gS_\Lambda[1/u]$ such that $\psi$ is given by
multiplication by~$\lambda$. The commutation relation with $\Phi_{M_\Lambda}$
means that we must have $\varphi(\lambda)=\lambda$, so that
certainly 
(considering the powers of~$u$ in~$\lambda$ of lowest negative and positive degrees)
$\lambda\in W(k')\otimes_{\Zp}\Lambda$, and in fact $\lambda\in
\Lambda$. Conversely, multiplication by any element of~$\Lambda$ is evidently an
endomorphism of~$M_\Lambda$, as required.
\end{proof}

\begin{lem}
  \label{lem:rank one isomorphism over a field}Let $\kappa$ be a field
  of characteristic~$p$, and let $M_\kappa, N_\kappa$ be
  \'etale $\varphi$-modules of rank one with $\kappa$-coefficients and
  descent data. Then any nonzero element of $\Hom_{\K{\kappa}}(M_\kappa, N_\kappa)$
  is an isomorphism. 
\end{lem}
\begin{proof}
  Since $\kappa((u))$ is a field, it is enough to show that if one of
  the induced maps $M_{\kappa,i}\to N_{\kappa,i}$ is nonzero, then
  they all are; but this follows from the commutation relation with~$\varphi$.
\end{proof}

\begin{lemma}\label{lem: isomorphic twists are the same twist} If $\lambda,\lambda' \in \Lambda^{\times}$ and $\gM_{\Lambda,\lambda} \cong \gM_{\Lambda,\lambda'}$
{\em (as Breuil--Kisin modules with descent data over $\Lambda$)}, then $\lambda=\lambda'$. Similarly, if $M_{\Lambda,\lambda}\cong M_{\Lambda,\lambda'}$, then $\lambda=\lambda'$.
\end{lemma}
\begin{proof}
Again, we give the proof for~$M$, the argument for~$\gM$ being
essentially identical. Write $M_i=\F((u))m_i$, and write
$\Phi_i(1\otimes m_{i-1})=\theta_{i}m_{i}$, where $\theta_{i}\ne 0$. There
are $\mu_i\in \Lambda[[u]][1/u]$ such that the given isomorphism
$M_{\Lambda,\lambda}\cong M_{\Lambda,\lambda'}$  takes~$m_i$ to $\mu_im_i$. The
commutation relation between the given isomorphism and~$\Phi_{M}$ imposes the
condition \[\lambda_{i}\mu_{i}\theta_{i}m_{i}=\lambda'_{i}\varphi(\mu_{i-1})\theta_{i}m_{i}\]where
$\lambda_{i}$ (resp.\ $\lambda'_i$) equals~$1$ unless $i=0$, when it equals~$\lambda$
(resp.\ $\lambda'$). 

Thus we have $\mu_{i}=(\lambda'_i/\lambda_i)\varphi(\mu_{i-1})$, so that in
particular $\mu_0=(\lambda'/\lambda)\varphi^{f'}(\mu_0)$. Considering the powers
of~$u$ in~$\mu_0$ of lowest negative and positive degrees we
conclude that $\mu_0\in W(k') \otimes \Lambda$; but then
$\mu_0=\varphi^{f'}(\mu_0)$, so that~$\lambda'=\lambda$, as required.
\end{proof}

\begin{remark} 
If $\gM$ has height at most~$h$, and we let $\cC$ (temporarily) denote the
moduli stack of rank one Breuil--Kisin modules of height at most~$h$ with $\F$-coefficients and
descent data then Lemma~\ref{lem: isomorphic twists are the same twist} can be interpreted as saying that 
the morphism $\Gm \to \cC$ 
that classifies $\gM_{\Gm,x}$ 
is a monomorphism, i.e.\ the diagonal morphism $\Gm \to \Gm
\times_{\cC}
\Gm$ is an isomorphism. Similarly, the morphism $\Gm\to\cR$
(where we temporarily let $\cR$ denote the moduli stack
of rank one \'etale $\varphi$-modules with $\F$-coefficients and descent data)
that classifies $M_{\Gm,x}$ is a monomorphism.
\end{remark}

Now choose another rank one Breuil--Kisin module with descent data $\gN$ over $\F$.
Let $(x,y)$ denote the standard coordinates on $\Gm\times_{\F} \Gm$,
and consider the rank one Breuil--Kisin modules with descent data $\gM_{\Gm\times_{\F} \Gm,x}$ 
and $\gN_{\Gm\times_{\F} \Gm,y}$ 
over $\Gm\times_{\F} \Gm$.

\begin{lemma}\label{lem: generically no Homs}
There is a non-empty irreducible affine open subset $\Spec \Adist$ of $\Gm\times_{\F} \Gm$
whose finite type points 
are exactly the maximal ideals $\mathfrak m$ of $\Gm\times_{\F} \Gm$
such that
 \[\Hom_{\K{\kappa(\mathfrak m)}}\bigl( \gM_{\kappa(\mathfrak m),\xbar}[1/u],
\gN_{\kappa(\mathfrak m),\ybar}[1/u]\bigr)=0\]
{\em (}where we have written $\xbar$ and $\ybar$ to
denote the images of $x$ and $y$ in $\kappa(\mathfrak m)^{\times}${\em )}.

Furthermore, if $R$ is any finite-type $\Adist$-algebra,
and if $\mathfrak m$ is any maximal ideal of~$R$,
then 
\[\Hom_{\K{\kappa(\mathfrak m)}}\bigl( \gM_{\kappa(\mathfrak m),\xbar},
\gN_{\kappa(\mathfrak m),\ybar}\bigr)
= \Hom_{\K{\kappa(\mathfrak m)}}\bigl( \gM_{\kappa(\mathfrak m),\xbar}[1/u],
\gN_{\kappa(\mathfrak m),\ybar}[1/u]\bigr)
= 0,\]
and also
\[
\Hom_{\K{\kappa(\mathfrak m)}}\bigl( \gN_{\kappa(\mathfrak m),\ybar},
\gM_{\kappa(\mathfrak m),\xbar}\bigr)
=
 \Hom_{\K{\kappa(\mathfrak m)}}\bigl( \gN_{\kappa(\mathfrak m),\ybar}[1/u],
\gM_{\kappa(\mathfrak m),\xbar}[1/u]\bigr)=0.\]
In particular, Assumption~{\em \ref{assumption:vanishing}}
is satisfied by $\gM_{\Adist,x}$ and $\gN_{\Adist,y}$.
\end{lemma} 
\begin{proof} 
If  $\Hom\bigl(\gM_{\kappa(\mathfrak m),\xbar}[1/u],\gN_{\kappa(\mathfrak m),\ybar}[1/u])=0$ 
for all maximal ideals~$\m$ of $\F[x^{\pm 1},y^{\pm 1}]$, then we are
done: $\Spec A^{\dist} = \Gm\times\Gm$. Otherwise, we see that
for some finite extension $\F'/\F$ and some $a,a'\in\F'$, we have a
non-zero morphism $\gM_{\F',a}[1/u]\to\gN_{\F',a'}[1/u]$. By
Lemma~\ref{lem:rank one isomorphism over a field}, this morphism must
in fact be an isomorphism. 
Since $\gM$ and $\gN$ are both defined over $\F$, we furthermore see
that the ratio $a'/a$ lies in $\F$.
We then let $\Spec\Adist$ be the affine open subset of $\Gm\times_{\F}
\Gm$ where $a'x\ne ay$; the claimed property of $\Spec A^{\dist}$ 
then follows easily from 
Lemma~\ref{lem: isomorphic twists are the same twist}.

For the remaining statements of the lemma,
note that if $\mathfrak m$ is a maximal
ideal in a finite type $A^{\dist}$-algebra, then its pull-back to $A^{\dist}$
is again a maximal ideal $\mathfrak m'$ 
of $A^{\dist}$ (since $A^{\dist}$ is Jacobson),
and the vanishing of
$$
\Hom_{\K{\kappa(\mathfrak m)}}\bigl( \gM_{\kappa(\mathfrak m),\xbar}[1/u],
\gN_{\kappa(\mathfrak m),\ybar}[1/u]\bigr)
$$
follows from the corresponding statement for $\kappa(\mathfrak m')$,
together with  Lemma~\ref{lem: flat base change for Homs}.

Inverting $u$ induces an embedding
\[\Hom_{\K{\kappa(\mathfrak m)}}\bigl( \gM_{\kappa(\mathfrak m),\xbar},
\gN_{\kappa(\mathfrak m),\ybar}\bigr)
\hookrightarrow
\Hom_{\K{\kappa(\mathfrak m)}}\bigl( \gM_{\kappa(\mathfrak m),\xbar}[1/u],
\gN_{\kappa(\mathfrak m),\ybar}[1/u]\bigr),\]
and so certainly the vanishing of the target implies the vanishing of
the source.

The statements in which the roles of $\gM$ and $\gN$
are reversed follow from 
Lemma~\ref{lem:rank one isomorphism over a field}. 
\end{proof}

Define
$T := \Ext^1_{\K{\Gm\times_{\F}\Gm}}\bigl(\gM_{\Gm\times_{\F} \Gm,x},\gM_{\Gm\times_{\F}\Gm,y})$;
it follows from 
Proposition~\ref{prop:exts are f.g. over A} that
$T$ is finitely generated over $\F[x^{\pm 1},y^{\pm 1}],$  
while
Proposition~\ref{prop:base-change for exts}
shows that
$T_\Adist := T\otimes_{\F[x^{\pm 1}, y^{\pm 1}]} \Adist$ is naturally isomorphic to
$\Ext^1_{\K{\Adist}}\bigl(\gM_{\Adist,x}, \gN_{\Adist,y}\bigr)$. (Here and elsewhere
we abuse notation by writing $x$, $y$ for $x|_{\Adist}$, $y|_{\Adist}$.)
Corollary~\ref{cor:freeness for exts} and Lemma~\ref{lem: generically no Homs}
show that $T_\Adist$ is in fact a 
finitely generated projective $\Adist$-module.
If, for any $\Adist$-algebra $B$, we write $T_B := T_\Adist \otimes_\Adist B \iso
 T\otimes_{\F[x^{\pm 1}, y^{\pm 1}]} B$,
then 
Proposition~\ref{prop:base-change for exts} again shows that
$T_B \iso \Ext^1_{\K{B}}\bigl(\gM_{B,x}, \gN_{B,y}\bigr)$.

By Propositions~\ref{prop: base change for kernel of map to etale Ext}
and~\ref{prop: we have vector bundles}, together with
Lemma~\ref{lem: generically no Homs},
there is a nonempty (so dense) affine open subset ~$\Spec\Akfree$ of
  $\Spec \Adist$ with the properties that \[U_\Akfree :=
  \kExt^1_{\K{\Akfree}}(\gM_{\Akfree,x},\gN_{\Akfree,y})\] and
\begin{multline*}
	T_\Akfree/U_\Akfree \\
 \iso
  \Ext^1_{\K{\Akfree}}(\gM_{\Akfree,x},\gN_{\Akfree,y})/\kExt^1_{\K{\Akfree}}(\gM_{\Akfree,x},\gN_{\Akfree,y})
  \end{multline*}
  are finitely generated and projective over $\Akfree$, and furthermore so
  that for all finitely generated $\Akfree$-algebras $B$, the formation of
$\kExt^1_{\K{B}}(\gM_{B,x},\gN_{B,y})$ and 
$\Ext^1_{\K{B}}(\gM_{B,x},\gN_{B,y})/\kExt^1_{\K{B}}(\gM_{B,x},\gN_{B,y})$
is compatible with base change from $U_\Akfree$ and $T_\Akfree/U_\Akfree$ respectively.

We choose a finite rank projective module $V$ over $\F[x^{\pm 1},y^{\pm 1}]$
admitting a surjection $V \to T$.
Thus, if we write $V_\Adist := V\otimes_{\F[x^{\pm 1}, y^{\pm 1}]} \Adist$,
then the induced morphism $V_\Adist \to T_\Adist$ is a (split) surjection of
$\Adist$-modules.

Following the prescription 
of Subsection~\ref{subsec:families of extensions},
we form the symmetric algebra $\Btwist := \F[x^{\pm 1}, y^{\pm
  1}][V^{\vee}],$ 
and construct the  family of extensions $\widetilde{\gE}$
over $\Spec \Btwist$.    We may similarly form the symmetric algebras
$\Bdist := \Adist[T_{\Adist}^{\vee}]$ and $\Bkfree := \Akfree[T_{\Akfree}^{\vee}]$, and construct the families 
of extensions  $\gEdist$ and $\gEkfree$ over $\Spec
\Bdist$ and $\Spec\Bkfree$ respectively. 
Since $T_\Akfree/U_\Akfree$ is projective, the natural morphism
$T_{\Akfree}^{\vee} \to U_{\Akfree}^{\vee}$ is surjective, and hence 
$\Ckfree := A[U_{\Akfree}^{\vee}]$ is a quotient of $\Bkfree$; geometrically,
$\Spec \Ckfree $ is a subbundle of the vector bundle $\Spec \Bkfree$
over $\Spec \Akfree$. 

We write $X := \Spec \Bkfree \setminus \Spec \Ckfree$; it is an open subscheme
of the vector bundle $\Spec \Bkfree$.  
The restriction of 
$\widetilde{\gE}'$ to~$X$ is the universal family of extensions over~$\Akfree$ which
do not split after inverting~$u$.

\begin{remark}
	\label{rem:Zariski density}
	Since $\Spec A^{\dist}$ and $\Spec A^{\kfree}$ are irreducible,
	each of the vector bundles $\Spec B^{\dist}$ and $\Spec B^{\kfree}$
 	is also irreducible.  In particular, $\Spec B^{\kfree}$ is Zariski dense
	in $\Spec B^{\dist}$, and if $X$ is non-empty, then it is Zariski dense
	in each of $\Spec B^{\kfree}$ and $\Spec B^{\dist}$.   Similarly,
	$\Spec \Btwist \times_{\Gm\times_{\F}\Gm} \Spec A^{\dist}$ is Zariski dense in
	$\Spec \Btwist$.
\end{remark}

The surjection $V_\Adist \to T_\Adist$ induces a surjection of vector bundles
$\pi: \Spec \Btwist\times_{\Gm\times_{\F}\Gm} \Spec \Adist \to \Spec \Bdist$ over $\Spec \Adist$, 
and there is a natural isomorphism
\numequation
\label{eqn:pull-back iso}
 \pi^*\gEdist \iso
\widetilde{\gE}\cotimes_{\F[x^{\pm 1},y^{\pm 1}]} \Adist.
\end{equation}

The rank two Breuil--Kisin module with descent data
$\widetilde{\gE}$ is classified by a morphism
$\xi:\Spec \Btwist \to \cC^{\dd,1}$; 
similarly,
the rank two Breuil--Kisin module with descent data
$\gEdist$ is classified by a morphism
$\xi^\dist: \Spec \Bdist \to \cC^{\dd,1}.$
If we write $\xi_\Adist$ for the restriction of $\xi$ to the open
subset $\Spec \Btwist \times_{\Gm\times_{\F}\Gm} \Spec \Adist$ of $\Spec \Btwist$,
then the isomorphism~(\ref{eqn:pull-back iso}) shows that
$\xi^\dist\circ \pi = \xi_\Adist$.
We also write $\xi^{\kfree}$ for the restriction of $\xi^{\dist}$ to 
$\Spec B^{\kfree}$, and $\xi_X$ for the restriction of $\xi^{\kfree}$
to $X$.

\begin{lemma}\label{lem: scheme theoretic images coincide}
The scheme-theoretic images {\em (}in the sense
  of~{\em \cite[Def.\ 3.1.4]{EGstacktheoreticimages})} of
  $\xi:\Spec \Btwist\to \cC^{\dd,1}$,
$\xi^\dist:\Spec \Bdist\to \cC^{\dd,1}$,
and $\xi^{\kfree}: \Spec B^{\kfree}\to \cC^{\dd,1}$
all coincide; in particular, the
scheme-theoretic image of~$\xi$ is independent of the choice of
surjection $V\to T$, and the scheme-theoretic image of~$\xi^{\kfree}$ is
independent of the choice of~$\Akfree$. 
  If $X$ is non-empty, then the scheme-theoretic image
  of $\xi_X: X \to \cC^{\dd,1}$ also coincides with these other scheme-theoretic images, and is independent of the choice of $\Akfree$.
\end{lemma} 
\begin{proof} 
	This follows from the various observations about Zariski density 
	made in Remark~\ref{rem:Zariski density}.
\end{proof}

\begin{defn}
	\label{def:scheme-theoretic images}
  We let $\overline{\cC}(\gM,\gN)$ denote the scheme-theoretic image  of
  $\xi^\dist:\Spec \Bdist\to \cC^{\dd,1}$, and
  we let $\overline{\cZ}(\gM,\gN)$ denote the
  scheme-theoretic image of the composite
  $\xi^\dist:\Spec \Bdist\to \cC^{\dd,1}\to \cZ^{\dd,1}$.
Equivalently,
$\overline{\cZ}(\gM,\gN)$ is the scheme-theoretic image of the
composite $\Spec \Bdist\to \cC^{\dd,1}\to \cR^{\dd,1}$ (\emph{cf}.\
\cite[Prop.\ 3.2.31]{EGstacktheoreticimages}), and the scheme-theoretic
image of $\overline{\cC}(\gM,\gN)$
under the morphism $\cC^{\dd,1} \to \cZ^{\dd,1}$.
(Note that Lemma~\ref{lem: scheme theoretic images coincide}
provides various other alternative descriptions of $\overline{\cC}(\gM,\gN)$
(and therefore also~$\overline{\cZ}(\gM,\gN$)) as 
a scheme-theoretic image.)
\end{defn}
\begin{rem}
  \label{rem: C(M,N) is reduced}Note that $\overline{\cC}
		(\gM,\gN)$
and $\overline{\cZ}(\gM,\gN)$
are both reduced (because they are each defined as a scheme-theoretic
		image of~$\Spec\Bdist$, which is reduced by definition).
\end{rem}

As well as scheme-theoretic images, as in the preceding Lemma and Definition,
we will need to consider images of underlying topological spaces. If
$\mathcal{X}$ is an algebraic stack we let $|\mathcal{X}|$ be its
underlying topological space, as defined in 
 \cite[\href{https://stacks.math.columbia.edu/tag/04Y8}{Tag 04Y8}]{stacks-project}.

\begin{lem}
	\label{lem:ext images}
	The image of the morphism on underlying topological spaces
	$| \Spec \Btwist | \to | \cC^{\dd,1}|$ induced by $\xi$ is
	a constructible subset of $| \cC^{\dd,1}|$, and is
	independent of the choice of $V$.
\end{lem}
\begin{proof}
	The fact that the image of $|\Spec \Btwist|$ is a constructible 
	subset of $|\cC^{\dd,1}|$ follows from the fact that
	$\xi$ is a morphism of finite presentation between Noetherian
        stacks; see~\cite[App.\ D]{MR2818725}. 
	Suppose now that $V'$ is another choice of finite rank projective
	$\F[x^{\pm 1},y^{\pm 1}]$-module surjecting onto~$T$.  Since it
        is possible to choose
	a finite rank projective module surjecting onto the pullback of $V,V'$ with respect to
        their maps to $T$, 
        we see that it suffices to prove the independence claim of the lemma
	in the case when $V'$ admits a surjection onto $V$ (compatible
	with the maps of each of $V$ and $V'$ onto $T$).
	If we write $B' := \F[x^{\pm 1}, y^{\pm 1}][(V')^{\vee}],$
	then the natural morphism $\Spec B' \to \Spec \Btwist$ is a surjection,
	and the morphism $\xi': \Spec B' \to \cC^{\dd,1}$ is the composite
	of this surjection with the morphism $\xi$.  Thus indeed
	the images of $|\Spec B'|$ and of $|\Spec \Btwist|$ coincide as
	subsets of $|\cC^{\dd,1}|$.
\end{proof}

\begin{df}
	\label{df:constructible images}
	We write $|\cC(\gM,\gN)|$ to denote the constructible subset
	of $|\cC^{\dd,1}|$ described in Lemma~\ref{lem:ext images}.
\end{df}

\begin{remark}
	We caution the reader that we don't define a substack $\cC(\gM,\gN)$
	of $\cC^{\dd,1}$. Rather, we have defined a closed substack
	$\overline{\cC}(\gM,\gN)$ of $\cC^{\dd,1}$, and a constructible subset
	$|\cC(\gM,\gN)|$ of $|\cC^{\dd,1}|$.  It follows from
	the constructions that $|\overline{\cC}(\gM,\gN)|$ is the
	closure in $|\cC^{\dd,1}|$ of $|\cC(\gM,\gN)|$.
\end{remark}

As in Subsection~\ref{subsec:families of extensions},
there is a natural action of $\Gm\times_{\F}\Gm$ on $T$,
and hence on each of $\Spec \Bdist$, $\Spec \Bkfree$ and~$X$,
given by the action of $\Gm$ as automorphisms on each of $\gM_{\Gm\times_{\F} \Gm,x}$ 
and $\gN_{\Gm\times_{\F} \Gm,y}$ (which induces a corresponding action on $T$,
hence on $T_\Adist$ and $T_\Akfree$, and hence on $\Spec \Bdist$ and
$\Spec \Bkfree$).   
Thus we may form the corresponding quotient stacks $[\Spec \Bdist / \Gm\times_{\F} \Gm]$ and  
$[X / \Gm\times_{\F} \Gm],$ each of which admits a natural morphism to~$\cC^{\dd,1}$. 

\begin{rem}
  \label{rem: potential confusion of two lots of Gm times Gm}Note that
  we are making use of two independent copies of $\Gm\times_{\F}\Gm$; one
  parameterises the different unramified twists of $\gM$ and $\gN$, and the other the
  automorphisms of (the pullbacks of) $\gM$ and $\gN$. 
\end{rem}
\begin{defn}
  \label{defn: strict situations}We say that the pair $(\gM,\gN)$ is
  \emph{strict} if $\Spec\Adist=\Gm\times_\F\Gm$.
\end{defn}

Before stating and proving the main result of this subsection,
we prove some lemmas (the first two of which amount to recollections
of standard --- and simple --- facts).

\begin{lem}
  \label{lem: fibre products of finite type}If $\cX\to\cY$ is a
  morphism of stacks over~$S$, with $\cX$ algebraic and of
  finite type over~$S$, and $\cY$ having diagonal which is
  representable by algebraic spaces and of finite type, then
  $\cX\times_{\cY}\cX$ is an algebraic stack of finite type
  over~$S$. 
\end{lem}
\begin{proof}
	The fact that $\cX\times_{\cY} \cX$ is an algebraic 
	stack follows
	from~\cite[\href{http://stacks.math.columbia.edu/tag/04TF}{Tag 04TF}]{stacks-project}.
  Since composites of morphisms of finite type are of finite type, in
  order to show that $\cX\times_{\cY} \cX$ is of finite type over $S$,
  it suffices to show that the natural morphism $\cX\times_{\cY} \cX
  \to \cX\times_S \cX$ is of finite type.  Since this
  morphism is the base-change of the diagonal morphism $\cY \to \cY\times_S \cY,$
  this follows by assumption.
\end{proof}

\begin{lem}\label{lem:fibres of kExt}
The following conditions are equivalent:\\
(1) $\kExt^1_{\K{\kappa(\mathfrak m)}}\bigl( \gM_{\kappa(\mathfrak m),\xbar},
\gN_{\kappa(\mathfrak m),\ybar}\bigr)= 0$
for all maximal ideals $\mathfrak m$ of $\Akfree$.\\
(2) $U_\Akfree = 0$.\\
(3) $\Spec C^{\kfree}$ is the trivial
vector bundle over $\Spec A^{\kfree}$. 
\end{lem}
\begin{proof}
  Conditions~(2) and~(3) are equivalent by definition. Since the formation of
$\kExt^1_{\K{\Akfree}}(\gM_{\Akfree,x},\gN_{\Akfree,y})$ is compatible with base change,
and since $\Akfree$ is Jacobson, (1) is equivalent to
the assumption that
$$\kExt^1_{\K{\Akfree}}(\gM_{\Akfree,x},\gN_{\Akfree,y})=0,$$
i.e.\ that~$U_\Akfree=0$, as required.
\end{proof}

\begin{lemma}
	\label{lem:C to R, with maps to Akfree}
If the equivalent conditions of Lemma~{\em \ref{lem:fibres of kExt}} hold,
	then the natural morphism
	\begin{multline*}
	\Spec B^{\kfree} \times_{\Spec A^{\kfree} \times_\F \cC^{\dd,1}}
		\Spec B^{\kfree}
		\\
		\to
	\Spec B^{\kfree} \times_{\Spec A^{\kfree} \times_\F \cR^{\dd,1}}
		\Spec B^{\kfree}
		\end{multline*}
		is an isomorphism.
\end{lemma}
\begin{proof}
  Since $\cC^{\dd,1} \to \cR^{\dd,1}$ is separated (being
  proper) and representable, the diagonal morphism
  $\cC^{\dd,1} \to \cC^{\dd,1}\times_{\cR^{\dd,1}} \cC^{\dd,1}$
  is a closed immersion, and hence the morphism in the statement
  of the lemma is
  a closed immersion.   Thus, in order to show that it is an
  isomorphism, it suffices to show that it induces a surjection
  on $R$-valued points, for any $\F$-algebra $R$.  Since 
the source and target are of finite type
over $\F$, by Lemma~\ref{lem: fibre products of finite type}, 
we may in fact restrict attention to finite type $\F$-algebras.

A morphism $\Spec R\to \Spec\Bkfree
  \times_{\Spec A^{\kfree} \times_{\F} \cC^{\dd,1}}\Spec\Bkfree$
  corresponds to an isomorphism class
of tuples $(\alpha,\beta:\gE\to\gE',\iota,\iota',\pi,\pi')$, where
\begin{itemize}
\item $\alpha$ is a  morphism
$\alpha:\Spec R\to\Spec \Akfree$,
\item  $\beta:\gE\to\gE'$ is an isomorphism of Breuil--Kisin modules
with descent data and coefficients in $R$, 
\item $\iota:\alpha^* \gN \to \gE$, $\iota':\alpha^* \gN \to \gE'$, $\pi:\gE \to
  \alpha^* \gM$ and $\pi':\gE' \to
  \alpha^* \gM$ are morphisms
with the properties that $0 \to \alpha^*\gN \buildrel \iota
\over \to \gE \buildrel \pi  \over \to \alpha^* \gM \to 0$ and  $0 \to \alpha^*\gN \buildrel \iota'
\over \to \gE' \buildrel \pi'  \over \to \alpha^* \gM \to 0$ are both
short exact. 
\end{itemize}

Similarly,
a morphism $\Spec R\to \Spec\Bkfree
  \times_{\Spec A^{\kfree} \times_{\F} \cR^{\dd,1}}\Spec\Bkfree$
  corresponds to an isomorphism class
of tuples $(\alpha,\gE,\gE',\beta,\iota,\iota',\pi,\pi')$, where
\begin{itemize}
\item $\alpha$ is a  morphism
$\alpha:\Spec R\to\Spec \Akfree$,
\item  $\gE$ and $\gE'$ are Breuil--Kisin modules
with descent data and coefficients in $R$, 
and $\beta$ is an isomorphism
 $\beta:\gE[1/u]\to\gE'[1/u]$ of etale $\varphi$-modules with
 descent data and coefficients in $R$,
\item $\iota:\alpha^* \gN \to \gE$, $\iota':\alpha^* \gN \to \gE'$, $\pi:\gE \to
  \alpha^* \gM$ and $\pi':\gE' \to
  \alpha^* \gM$ are morphisms
with the properties that $0 \to \alpha^*\gN \buildrel \iota
\over \to \gE \buildrel \pi  \over \to \alpha^* \gM \to 0$ and  $0 \to \alpha^*\gN \buildrel \iota'
\over \to \gE' \buildrel \pi'  \over \to \alpha^* \gM \to 0$ are both
short exact. 
\end{itemize}

Thus to prove the claimed surjectivity, we have to show that, given
a tuple $(\alpha,\gE,\gE',\beta,\iota,\iota',\pi,\pi')$ associated
to a morphism $\Spec R\to \Spec\Bkfree
  \times_{\Spec A^{\kfree} \times_{\F} \cR^{\dd,1}}\Spec\Bkfree$,
  the isomorphism $\beta$ restricts to an isomorphism
  $\gE \to \gE'.$ 

By 
Lemma~\ref{lem:fibres of kExt}, the
  natural map
  $\Ext^1(\alpha^*\gM,\alpha^*\gN)\to\Ext^1_{\K{R}}(\alpha^*\gM[1/u],\alpha^*\gN[1/u])$
  is injective; so the Breuil--Kisin modules $\gE$ and $\gE'$ are isomorphic. Arguing as in the proof of Corollary~\ref{cor:
    monomorphism to Spec A times C}, we see that~$\beta$ is equivalent
  to the data of an $R$-point of~$\Gm\times_{\cO}\Gm$,
  corresponding to the automorphisms of $\alpha^*\gM[1/u]$ and
  $\alpha^*\gN[1/u]$ that it induces. These restrict to
  automorphisms of $\alpha^*\gM$ and $\alpha^*\gN$, so that
   (again by the proof of Corollary~\ref{cor:
    monomorphism to Spec A times C}) 
  $\beta$ indeed restricts to an
  isomorphism $\gE\to\gE'$, as required.
\end{proof}
	
We now present the main result of this subsection.

\begin{prop}\label{prop: construction of family monomorphing to C and R}
{\em (1)} The morphism
			$\xi^{\dist}$ induces a morphism 
\numequation
\label{eqn:unramified morphism}
[\Spec \Bdist / \Gm \times_{\F} \Gm ] \to
\cC^{\dd,1},
\end{equation}which is representable by algebraic spaces, of finite type,
			and unramified,
whose fibres over finite type points  are of degree $\leq 2$. 
In the strict case, this induced morphism is in fact a monomorphism,
while in general, the restriction $\xi_X$ of $\xi^{\dist}$
induces a finite type monomorphism 
  \numequation
  \label{eqn:monomorphism}
[X / \Gm \times_{\F} \Gm ] \hookrightarrow
\cC^{\dd,1}.
\end{equation}

{\em (2)}
If
$\kExt^1_{\K{\kappa(\mathfrak m)}}\bigl( \gM_{\kappa(\mathfrak m),\xbar},
\gN_{\kappa(\mathfrak m),\ybar}\bigr)=0$
for all maximal ideals $\mathfrak m$ of $\Akfree$, 
then the composite morphism
\numequation
\label{eqn:second unramified morphism}
[\Spec B^{\kfree}/\Gm\times_\F\Gm]\to \cC^{\dd,1}\to\cR^{\dd,1}
\end{equation}
is a representable by algebraic spaces, of finite type,
and unramified,
with fibres of degree $\leq 2.$
In the strict case, this induced morphism is in fact a monomorphism,
while in general,
the composite morphism
\numequation
\label{eqn:second monomorphism}
[X/\Gm\times_\F\Gm]\hookrightarrow \cC^{\dd,1}\to\cR^{\dd,1}
\end{equation}
is a finite type monomorphism. 
\end{prop} 

\begin{remark}\label{rem:explain-hypotheses}
The failure of~\eqref{eqn:unramified morphism} to be a monomorphism in
general is due, effectively,  to the possibility that an extension $\gE$ of some 
$\gM_{R,x}$ by $\gN_{R,y}$ and an extension $\gE'$ of some
$\gM_{R,x'}$ by $\gN_{R,y'}$ might be isomorphic as
Breuil--Kisin modules while nevertheless $(x,y)\neq (x',y')$. As we
will see in the proof,
whenever this happens the map $\gN_{\Lambda,y} \to \gE\to \gE' 
\to \gM_{\Lambda,x'}$ is nonzero, and then
$\gE' \otimes_R \kappa(\frakm)[1/u]$ is split for some maximal ideal $\frakm$
of $R$. This explains why, to obtain a monomorphism,
we can restrict either to the strict case or to the substack of extensions
that are non-split after inverting $u$.
\end{remark}

\begin{remark}
	We have stated this proposition in the strongest form that we
	are able to prove, but in fact its full strength is not required
	in the subsequent applications. 
              In particular, we don't need the precise bounds on the
	degrees of the fibres. 
\end{remark}
\begin{proof}[Proof of Proposition~{\ref{prop:
		construction of family monomorphing to C and R}}]
	By Corollary~\ref{cor: monomorphism to Spec A times C}
  (which we can apply because Assumption~\ref{assumption:vanishing} is
  satisfied, by Lemma~\ref{lem: generically no Homs})
  the natural morphism $[\Spec \Bdist/ \Gm \times_{\F} \Gm ] \to
  \Spec \Adist\times_{\F} \cC^{\dd,1}$ is a finite type monomorphism, 
  and hence so is its restriction to the open substack
  $[X/\Gm\times_{\F} \Gm]$ of its source.

Let us momentarily write $\cX$ to denote either $[\Spec B^{\dist}/
\Gm\times_{\F} \Gm]$  or $[X/\Gm\times_{\F} \Gm]$.  To show that
the finite type morphism
$\cX \to \cC^{\dd,1}$ is representable by algebraic spaces,
resp.\ unramified, resp.\ a
monomorphism,
it suffices to show that the corresponding diagonal morphism
$\cX \to \cX \times_{\cC^{\dd,1}} \cX$ is a monomorphism, resp.\ \'etale, resp.\ 
an isomorphism.

Now since $\cX \to \Spec A^{\dist} \times_{\F} \cC^{\dd,1}$ is a monomorphism,
the diagonal morphism $\cX \to \cX \times_{\Spec A^{\dist}\times_{\F} \cC^{\dd,1}}
\cX$ {\em is} an isomorphism, 
and so it is equivalent to show that the morphism of products
$$\cX \times_{\Spec A^{\dist}\times_{\F} \cC^{\dd,1}} \cX \to
\cX\times_{\cC^{\dd,1}} \cX$$
is a monomorphism, resp.\ \'etale, resp.\ an isomorphism.
This is in turn equivalent to showing the corresponding properties
for the morphisms
\numequation
\label{eqn:first closed immersion}
  \Spec\Bdist\times_{\Spec \Adist\times \cC^{\dd,1}}\Spec\Bdist \to
 \Spec\Bdist \times_{\cC^{\dd,1}}\Spec\Bdist 
 \end{equation}
 or
\numequation
\label{eqn:second closed immersion}
  X \times_{\Spec \Adist\times \cC^{\dd,1}}X \to
 X \times_{\cC^{\dd,1}} X.
 \end{equation}
 Now each of these morphisms is a base-change of the diagonal
 $\Spec A^{\dist}\to \Spec A^{\dist} \times_{\F} \Spec A^{\dist},$
 which is a closed immersion (affine schemes being separated),
 and so is itself a closed immersion.   In particular,
 it is a monomorphism, and so we have proved the representability
 by algebraic spaces
 of each of~(\ref{eqn:unramified morphism}) and~(\ref{eqn:monomorphism}).
Since the source and target of each of these monomorphisms
is of finite type
over~$\F$, by Lemma~\ref{lem: fibre products of finite type}, 
in order to show that either of these monomorphisms is
an isomorphism,
 it suffices to show that it induces a surjection on 
$R$-valued points, for arbitrary finite type $\F$-algebras $R$.
Similarly, to check that the
closed immersion~(\ref{eqn:first closed immersion}) is \'etale,
it suffices to verify that it is formally smooth,
and for this it suffices to verify that it satisfies the
infinitesimal lifting property 
with respect to square zero thickenings of finite type
$\F$-algebras.

A morphism $\Spec R\to \Spec\Bdist
  \times_{\cC^{\dd,1}}\Spec\Bdist$ corresponds to an isomorphism class
of tuples $(\alpha,\alpha',\beta:\gE\to\gE',\iota,\iota',\pi,\pi')$, where
\begin{itemize}
\item $\alpha,\alpha'$ are morphisms
$\alpha,\alpha':\Spec R\to\Spec \Adist$,
\item  $\beta:\gE\to\gE'$ is an isomorphism of Breuil--Kisin modules
with descent data and coefficients in $R$, 
\item $\iota:\alpha^* \gN \to \gE$, $\iota':(\alpha')^* \gN \to \gE'$, $\pi:\gE \to
  \alpha^* \gM$ and $\pi':\gE' \to
  (\alpha')^* \gM$ are morphisms
with the properties that $0 \to \alpha^*\gN \buildrel \iota
\over \to \gE \buildrel \pi  \over \to \alpha^* \gM \to 0$ and  $0 \to (\alpha')^*\gN \buildrel \iota'
\over \to \gE' \buildrel \pi'  \over \to (\alpha')^* \gM \to 0$ are both
short exact. 
\end{itemize}

We begin by proving that~(\ref{eqn:first closed immersion}) satisfies
the infinitesimal lifting criterion (when $R$ is a finite type $\F$-algebra).
Thus we assume given a square-zero ideal $I \subset R$, 
such that the induced morphism
$$\Spec R/I \to \Spec B^{\dist} \times_{\cC^{\dd,1}} \Spec B^{\dist}$$
factors through $\Spec B^{\dist}\times_{\Spec A^{\dist} \times_{\F} \cC^{\dd,1}} 
\Spec B^{\dist}$. In terms of the data
$(\alpha,\alpha',\beta:\gE\to\gE',\iota,\iota',\pi,\pi')$, 
we are assuming that $\alpha$ and $\alpha'$ coincide when restricted
to~$\Spec R/I$, and 
we must show that $\alpha$ and $\alpha'$ themselves coincide.

To this end, we consider the composite
\numequation
\label{eqn:key composite}
\alpha^*\gN\stackrel{\iota}{\to}\gE\stackrel{\beta}{\to}\gE'\stackrel{\pi'}{\to}(\alpha')^*\gM. 
\end{equation}
If we can show the vanishing of this morphism,
then by reversing the roles of $\gE$ and $\gE'$,
we will similarly deduce the vanishing of
$\pi \circ \beta^{-1} \circ \iota'$,  
from which we can conclude that $\beta$ induces an isomorphism between
$\alpha^*\gN$ and $(\alpha')^*\gN$. Consequently, it also induces an
isomorphism between~$\alpha^*\gM$ and~$(\alpha')^*\gM$, so it follows from
Lemma~\ref{lem: isomorphic twists are the same twist} that  $\alpha=\alpha'$,
as required. 

We show the vanishing of~(\ref{eqn:key composite}).
Suppose to the contrary that it doesn't vanish,
so that we have a non-zero morphism
$\alpha^*\gN\to (\alpha')^*\gM.$
It follows from Proposition~\ref{prop:vanishing of homs} that,
for some maximal ideal $\m$ of $R$, there exists a non-zero morphism
\[\alpha^*(\gN)\otimes_R\kappa(\mathfrak m) 
  {\to}(\alpha')^*(\gM)\otimes_R\kappa(\mathfrak
  m).\] 
By assumption $\alpha$ and $\alpha'$ coincide modulo $I$.  Since $I^2 = 0$,
there is an inclusion $I \subset \mathfrak m$,
and so in particular we find that
$$(\alpha')^*(\gM) \otimes_R \kappa(\mathfrak m)
\iso \alpha^*(\gM)\otimes_R \kappa(\mathfrak m).$$
Thus there exists a non-zero morphism
\[\alpha^*(\gN)\otimes_R\kappa(\mathfrak m) 
  {\to}\alpha^*(\gM)\otimes_R\kappa(\mathfrak m).\] 
Then, by Lemma~\ref{lem:rank one isomorphism over a field},
  after inverting~$u$ we obtain an isomorphism
\[\alpha^*(\gN)\otimes_R\kappa(\mathfrak m) [1/u]
  {\iso}\alpha^*(\gM)\otimes_R\kappa(\mathfrak m)[1/u],\] 
contradicting the assumption that $\alpha$ maps $\Spec R$
into $\Spec A^{\dist}$.
This completes the proof that~(\ref{eqn:first closed immersion})
is formally smooth, and hence that~(\ref{eqn:unramified morphism})
is unramified.

We next show that, in the strict case,
the closed immersion~(\ref{eqn:first closed immersion}) 
is an isomorphism, and thus that~(\ref{eqn:unramified morphism})
is actually a monomorphism. As noted above, 
it suffices to show that~(\ref{eqn:first closed immersion})
induces a surjection on $R$-valued points for finite type $\F$-algebras $R$,
which in terms of the data
$(\alpha,\alpha',\beta:\gE\to\gE',\iota,\iota',\pi,\pi')$, 
amounts to showing that necessarily $\alpha = \alpha'$.
Arguing just as we did above,
it suffices show the vanishing of~(\ref{eqn:key composite}).

Again, we suppose for the sake of contradiction that~(\ref{eqn:key composite})
does not vanish. It then follows
from Proposition~\ref{prop:vanishing of homs} that
for some maximal ideal $\m$ of $R$ there exists a non-zero
morphism \[\alpha^*(\gN)\otimes_R\kappa(\mathfrak m) 
  {\to}(\alpha')^*(\gM)\otimes_R\kappa(\mathfrak
  m).\] 
Then, by Lemma~\ref{lem:rank one isomorphism over a field},
  after inverting~$u$ we obtain an isomorphism
  \numequation
  \label{eqn:key isomorphism}
  \alpha^*(\gN)\otimes_R\kappa(\mathfrak
  m)[1/u]
  \iso (\alpha')^*(\gM)\otimes_R\kappa(\mathfrak m)[1/u].
  \end{equation}
In the strict case, such an isomorphism cannot exist by assumption,
and thus~(\ref{eqn:key composite}) must vanish.

We now turn to proving that~(\ref{eqn:second closed immersion}) 
is an isomorphism. Just as in the preceding arguments,
it suffices to show that~(\ref{eqn:key composite}) vanishes, and
if not 
then we obtain an isomorphism~(\ref{eqn:key isomorphism}).
Since 
we are considering points of $X\times X$,
we are given that the
induced extension $\gE'\otimes_R\kappa(\mathfrak m)[1/u]$ is non-split,
so that the base change of the morphism~(\ref{eqn:key composite})
from $R[[u]]$ to $\kappa(\mathfrak m)((u))$ must vanish.  Consequently
the composite $\beta\circ \iota$ induces a non-zero morphism
$\alpha^*(\gN)\otimes_R\kappa(\mathfrak m)[1/u] \to (\alpha')^*(\gN)
\otimes_R\kappa(\mathfrak m)[1/u],$
which, by Lemma~\ref{lem:rank one isomorphism over a field},
must in fact be an isomorphism.  Comparing this isomorphism
with the isomorphism~(\ref{eqn:key isomorphism}),
we find that
$(\alpha')^*(\gN)\otimes_R\kappa(\mathfrak m)[1/u]$
and
$(\alpha')^*(\gM)\otimes_R\kappa(\mathfrak m)[1/u]$
are isomorphic, contradicting the fact that $\alpha'$ maps 
$\Spec R$ to $\Spec A^{\dist}$.
Thus in fact the composite~(\ref{eqn:key composite}) must vanish,
and we have completed the proof that~(\ref{eqn:monomorphism})
is a monomorphism.

To complete the proof of part~(1) of the proposition,
we have to show that the fibres of~(\ref{eqn:unramified morphism})
are of degree at most $2$. We have already observed that $[\Spec \Bdist/ \Gm \times_{\F} \Gm ] \to
  \Spec \Adist\times_{\F} \cC^{\dd,1}$ is a monomorphism, so it is
  enough to check that given a finite extension~$\F'/\F$ and an
  isomorphism class of tuples $(\alpha,\alpha',\beta:\gE\to\gE',\iota,\iota',\pi,\pi')$, where
\begin{itemize}
\item $\alpha,\alpha'$ are distinct morphisms
$\alpha,\alpha':\Spec \F'\to\Spec \Adist$,
\item  $\beta:\gE\to\gE'$ is an isomorphism of Breuil--Kisin modules
with descent data and coefficients in $\F'$, 
\item $\iota:\alpha^* \gN \to \gE$, $\iota':(\alpha')^* \gN \to \gE'$, $\pi:\gE \to
  \alpha^* \gM$ and $\pi':\gE' \to
  (\alpha')^* \gM$ are morphisms
with the properties that $0 \to \alpha^*\gN \buildrel \iota
\over \to \gE \buildrel \pi  \over \to \alpha^* \gM \to 0$ and  $0 \to (\alpha')^*\gN \buildrel \iota'
\over \to \gE' \buildrel \pi'  \over \to (\alpha')^* \gM \to 0$ are both
short exact. 
\end{itemize}
then~$\alpha'$ is determined by the data of~$\alpha$ and~$\gE$. To see
this, note that since we are assuming that~$\alpha'\ne\alpha$, the
arguments above show that~(\ref{eqn:key composite}) does not vanish,
so that (since~$\F'$ is a field), we have an
isomorphism~$\alpha^*\gN[1/u]\isoto(\alpha')^*\gM[1/u]$. Since we are
over~$\Adist$, it follows that~$\gE[1/u]\cong\gE'[1/u]$ is split, and
that we also have an
isomorphism~$\alpha^*\gM[1/u]\isoto(\alpha')^*\gN[1/u]$. Thus
if~$\alpha''$ is another possible choice for~$\alpha'$, we have
$(\alpha'')^*\gM[1/u]\isoto(\alpha')^*\gM[1/u]$ and
$(\alpha'')^*\gN[1/u]\isoto(\alpha')^*\gN[1/u]$,
whence~$\alpha''=\alpha'$ by Lemma~\ref{lem: isomorphic twists are the
  same twist}, as required.

We turn to proving~(2), and thus
assume that
$$\kExt^1_{\K{\kappa(\mathfrak m)}}\bigl( \gM_{\kappa(\mathfrak m),\xbar},
\gN_{\kappa(\mathfrak m),\ybar}\bigr)=0$$
for all maximal ideals $\mathfrak m$ of $\Akfree$.

Lemma~\ref{lem:C to R, with maps to Akfree} shows that
$$\Spec B^{\kfree} \times_{\Spec A^{\kfree}\times_{\F}
	\cC^{\dd,1}}\Spec B^{\kfree} \to
\Spec B^{\kfree} \times_{\Spec A^{\kfree}\times_{\F}
	\cR^{\dd,1}}\Spec B^{\kfree}$$
is an isomorphism, from which we deduce
that
$$[\Spec B^{\kfree}/\Gm\times_{\F} \Gm] \to \Spec A^{\kfree}\times_{\F}
\cR^{\dd,1}$$
is a monomorphism.
Using this as input, the claims of~(2) may be proved in an essentially identical
fashion to those of~(1). 
\end{proof}

 \begin{cor} 
  \label{cor: dimension of families of extensions}
The dimension of $\overline{\cC}(\gM,\gN)$ 
is equal to 
the rank of $T_\Adist$ as a projective $\Adist$-module.  
If 
$$\kExt^1_{\K{\kappa(\mathfrak m)}}\bigl( \gM_{\kappa(\mathfrak m),\xbar},
\gN_{\kappa(\mathfrak m),\ybar}\bigr)=0$$
for all maximal ideals $\mathfrak m$ of $A^{\kfree}$, 
then the dimension of $\overline{\cZ}(\gM,\gN)$ is also equal to this rank,
while 
if
$$\kExt^1_{\K{\kappa(\mathfrak m)}}\bigl( \gM_{\kappa(\mathfrak m),\xbar},
\gN_{\kappa(\mathfrak m),\ybar}\bigr) \neq 0$$
for all maximal ideals $\mathfrak m$ of $A^{\kfree}$, 
then the dimension of $\overline{\cZ}(\gM,\gN)$ is strictly less than this rank.
\end{cor}
\begin{proof} The dimension of $[\Spec\Bdist/\Gm\times_{\F} \Gm]$ is equal to
	the rank of $T_{A^{\dist}}$ (it is the quotient
	by a two-dimensional group of a vector bundle over a two-dimensional base of rank
	equal to the rank of $T_{A^{\dist}}$). By Lemma~\ref{lem: scheme theoretic images coincide},
	$\overline{\cC}(\gM,\gN)$ is the scheme-theoretic image
	of the morphism
	$[\Spec\Bdist/\Gm\times_{\F}\Gm] \to \cC^{\dd,1}$
	provided by
	Proposition~\ref{prop: construction of family monomorphing to
          C and R}(1), which (by that proposition) is representable 
  by algebraic spaces and unramified.
  Since such a morphism is locally quasi-finite
  (in fact, in this particular case, 
  we have shown that the fibres of this morphism have degree at
        most~$2$), \cite[\href{https://stacks.math.columbia.edu/tag/0DS6}{Tag 0DS6}]{stacks-project}
 ensures
	that $\overline{\cC}(\gM,\gN)$ has the claimed dimension.

	If 
$\kExt^1_{\K{\kappa(\mathfrak m)}}\bigl( \gM_{\kappa(\mathfrak m),\xbar},
\gN_{\kappa(\mathfrak m),\ybar}\bigr) = 0$
for all maximal ideals $\mathfrak m$ of $A^{\kfree}$, 
then an identical argument using Proposition~\ref{prop: construction of family monomorphing to
          C and R}(2) implies the claim regarding the dimension
of $\overline{\cZ}(\gM,\gN)$.

Finally, suppose that 
$$\kExt^1_{\K{\kappa(\mathfrak m)}}\bigl( \gM_{\kappa(\mathfrak m),\xbar},
\gN_{\kappa(\mathfrak m),\ybar}\bigr) \neq 0$$
for all maximal ideals $\mathfrak m$ of $A^{\kfree}$. 
Then the composite $[\Spec\Bkfree/\Gm\times_{\F} \Gm] \to \cC^{\dd,1} \to \cR^{\dd,1}$
has the property that for every point $t$ in the source, the fibre over the
image of $t$ has a positive dimensional fibre.  \cite[\href{https://stacks.math.columbia.edu/tag/0DS6}{Tag 0DS6}]{stacks-project} then implies the remaining
	claim of the present lemma. 
\end{proof}

\section{Extensions of rank one Breuil--Kisin modules}
\label{sec:extensions-of-rank-one}

\subsection{Rank one modules over finite fields, and their extensions}
\label{subsec: Diamond--Savitt}

We now wish to apply the results of the previous section to study
the geometry of our various moduli stacks. In order to do this, it
will be convenient for us to have an explicit description of the
rank one Breuil--Kisin modules of height at most one with descent data over a
finite field of characteristic $p$, and of their possible extensions. 
Many of the results in this section are proved (for $p>2$)  in~\cite[\S 1]{DiamondSavitt} in the context of
Breuil modules, and in those cases it
is possible simply to translate the relevant statements to the Breuil--Kisin module context.

Assume from now on that $e(K'/K)$ is divisible by $p^{f}-1$, so
  that we are in the setting of~\cite[Remark 1.7]{DiamondSavitt}.
  (Note that the parallel in \cite{DiamondSavitt} of our field
  extension $K'/K$, with ramification and inertial indices $e',f'$ and
  $e,f$ respectively, is the extension $K/L$ with indices $e,f$ and
  $e',f'$ respectively.) 

Let $\F$ be a finite subfield of $\Fpbar$ containing the image of some (so all)
embedding(s) $k'\into\Fpbar$. 
Recall that for each
$g\in\Gal(K'/K)$ we write $g(\pi')/\pi'=h(g)$ with $h(g)\in
\mu_{e(K'/K)}(K') \subset W(k')$. We abuse notation and
denote the image of $h(g)$ in $k'$ again by $h(g)$, so that we obtain
a map 
$\hchar \colon \Gal(K'/K) \to (k')^{\times}$. 
Note that~$\hchar$ restricts to a character on the inertia
subgroup $I(K'/K)$, and is itself a character when $e(K'/K) = p^f-1$.

\begin{lem}
  \label{lem:rank one Kisin modules with descent data}Every rank one
  Breuil--Kisin module of height at most one with descent data and $\F$-coefficients is isomorphic
  to one of the modules $\gM(r,a,c)$ defined by: 
  \begin{itemize}
  \item $\gM(r,a,c)_i=\F[[u]]\cdot m_i$,
  \item $\Phi_{\gM(r,a,c),i}(1\otimes m_{i-1})=a_{i} u^{r_{i}} m_{i}$,
  \item $\ghat(\sum_i m_i)=\sum_i h(g)^{c_i} m_i$ for all $g\in\Gal(K'/K)$, 
  \end{itemize}
where  $a_i\in\F^\times$,  $r_i\in\{0,\dots,e'\}$ and $c_i\in\Z/e(K'/K)$ are sequences
satisfying 
$pc_{i-1}\equiv c_{i}+r_{i}\pmod{e(K'/K)}$, the sums in the third
bullet point run from $0$ to $f'-1$, and the $r_i,c_i,a_i$ are periodic with
period dividing $f$. 

Furthermore, two such modules $\gM(r,a,c)$ and $\gM(s,b,d)$ are
isomorphic if and only if $r_i=s_i$ and $c_i=d_i$ for all $i$, and $\prod_{i=0}^{f-1}a_i=\prod_{i=0}^{f-1}b_i$.
\end{lem}
\begin{proof}
The proof is elementary; see e.g.\ \cite[Thm.~2.1,
Thm.~3.5]{SavittRaynaud} for proofs of analogous results.
\end{proof}

We will sometimes refer to the element $m = \sum_i m_i \in
\gM(r,a,c)$ as the standard generator of $\gM(r,a,c)$.

\begin{rem}
 When $p > 2$
many of the results in this section (such as the above) can
be obtained by translating \cite[Lem.\ 1.3, Cor.\
1.8]{DiamondSavitt} from the Breuil module context to the Breuil--Kisin module context.
We briefly recall the dictionary between these two categories
(\emph{cf.}\ \cite[\S 1.1.10]{kis04}). If $A$ is a finite local
$\Zp$-algebra, write $S_A = S \otimes_{\Zp} A$, where $S$ is Breuil's
ring. We regard $S_A$ as a $\gS_A$-algebra via
$u\mapsto u$, and we let $\varphi:\gS_A\to S_A$ be the composite of this map with
$\varphi$ on $\gS_A$. Then given a Breuil--Kisin module of height at most~$1$ with descent data $\gM$, 
we
set $\cM:=S_A\otimes_{\varphi,\gS_A}\gM$. We have a map $1\otimes\varphi_\gM:S_A\otimes_{\varphi,\gS_A}\gM\to S_A \otimes_{\gS_A}\gM$,
and we set \[\Fil^1\cM:=\{x\in\cM\ :\
(1\otimes \varphi_\gM)(x)\in\Fil^1S_A\otimes_{\gS_A}\gM\subset
S_A\otimes_{\gS_A}\gM\}\]and define $\varphi_1:\Fil^1\cM\to\cM$ as the
composite \[\Fil^1\cM\overset{1\otimes\varphi_\gM}{\longrightarrow}\Fil^1S_A\otimes_{\gS_A}\gM\overset{\varphi_1\otimes
  1}{\longrightarrow}S_A\otimes_{\varphi,\gS_A}\gM=\cM.\]Finally, we define $\hat{g}$ on $\cM$
via $\hat{g}(s\otimes m)=g(s)\otimes \hat{g}(m)$. One checks without difficulty
that this makes $\cM$ a strongly divisible module with descent data 
(\emph{cf.}\ the
proofs of~\cite[Proposition 1.1.11, Lemma 1.2.4]{kis04}).

  In the correspondence described above, the Breuil--Kisin module $\gM((r_i),(a_i),(c_i))$ corresponds to the Breuil module $\cM((e'-r_i),(a_i),(pc_{i-1}))$ 
of~\cite[Lem.\ 1.3]{DiamondSavitt}. 
\end{rem}

\begin{defn}
If $\gM = \gM(r,a,c)$ is a rank one Breuil--Kisin module as described in the 
preceding lemma, we set $\alpha_i(\gM)  :=   (p^{f'-1} r_{i-f'+1} + \cdots +
r_{i})/(p^{f'} - 1)$ (equivalently, $(p^{f-1} r_{i-f+1} +  \cdots 
+ r_{i})/(p^f-1)$). We may abbreviate $\alpha_i(\gM)$ simply as $\alpha_i$
when $\gM$ is clear from the context.

 It follows easily from the congruence $r_i 
\equiv pc_{i-1} - c_i \pmod{e(K'/K)}$ together with the hypothesis 
that $p^f-1 \mid e(K'/K)$  that $\alpha_i \in \Z$ for all $i$. Note
that the $\alpha_i$'s are the unique solution to the system of
equations $p \alpha_{i-1} - \alpha_i = r_i$ for all $i$. Note also
that $(p^f-1)(c_i-\alpha_i) \equiv 0 \pmod{e(K'/K)}$, so that
$\hchar^{c_i-\alpha_i}$ is a character with image in $k^{\times}$.
\end{defn}

\begin{lem}
  \label{lem: generic fibres of rank 1 Kisin modules} For any $i$ we have
  $T(\gM(r,a,c))=\left(\sigma_i\circ\hchar^{c_i-\alpha_{i}}\cdot\ur_{\prod_{i=0}^{f-1}a_i}\right)|_{G_{K_\infty}}$,
  where $\ur_\lambda$ is the unramified character of $G_K$ sending
  geometric Frobenius to $\lambda$. 
\end{lem}

\begin{proof}
 Set $\gN = \gM(0,(a_i),0)$, so that $\gN$ is effectively a Breuil--Kisin module without
 descent data. Then for $\gN$ this result follows from the second paragraph of the
 proof \cite[Lem.~6.3]{MR3164985}. (Note that the functor $T_{\gS}$ of
 \emph{loc.\ cit.} is dual to our functor $T$;\ \emph{cf}.~\cite[A\
 1.2.7]{MR1106901}. Note also that the fact that the base field is
 unramified in \emph{loc.\ cit.} does not change the calculation.) If
 $n = \sum n_i$ is the standard generator of $\gN$ as in Lemma~\ref{lem:rank one Kisin modules with descent data}, let
 $\gamma \in \Zp^{\un} \otimes_{\Zp} (k' \otimes_{\Fp} \F)$ be an element
 so that $\gamma  n \in
 (\cO_{\widehat{\cE^{\text{nr}}}}\otimes_{\gS[1/u]} \gN[1/u])^{\varphi
   = 1}$.

Now for $\gM$ as in the statement of the lemma it is straightforward to verify
that $$\gamma \sum_{i=0}^{f'-1} [\underline{\pi}']^{-\alpha_{i}} \otimes
m_i \in  (\cO_{\widehat{\cE^{\text{nr}}}}\otimes_{\gS[1/u]}
\gM[1/u])^{\varphi=1},$$ and the result follows.
\end{proof}

One immediately deduces the following.

\begin{cor}
  \label{cor: Kisin modules with the same generic fibre} Let $\gM=\gM(r,a,c)$ and
  $\gN=\gM(s,b,d)$ be rank one Breuil--Kisin modules with descent data as
  above.  We have $T(\gM)=T(\gN)$ if and only if $c_i - \alpha_i(\gM)
  \equiv d_i -  \alpha_i(\gN) \pmod{e(K'/K)}$ for some $i$ {\upshape(}hence for all $i${\upshape)} and $\prod_{i=0}^{f-1}a_i=\prod_{i=0}^{f-1}b_i$.
\end{cor}

\begin{lem}
  \label{lem: maps between rank 1 Kisin modules}  In the notation of the
  previous Corollary, there is a nonzero map $\gM\to\gN$
  \emph{(}equivalently, $\dim_{\F} \Hom_{\K{\F}}(\gM,\gN)=1$\emph{)} if
  and only if $T(\gM)=T(\gN)$ and $\alpha_i(\gM) \ge\alpha_i(\gN)$ for each $i$.
\end{lem}

\begin{proof}

  The proof is  essentially the same as that of \cite[Lem.\
  1.6]{DiamondSavitt}. (Indeed, when $p > 2$ this
  lemma can once again be proved by translating directly from
  \cite{DiamondSavitt} to the Breuil--Kisin module context.)
\end{proof}

 Using the material of Section~\ref{subsec:ext
  generalities}, 
one can compute $\Ext^1(\gM,\gN)$ for any pair of rank
one Breuil--Kisin modules $\gM,\gN$ of height at most one. We begin with the
following explicit description of the complex $C^{\bullet}(\gN)$ of  Section~\ref{subsec:ext
  generalities}.

\begin{defn}\label{notn:calh}
 We write  $\Czerofrac = \Czerofrac(\gM,\gN) \subset \F((u))^{\Z/f\Z}$ for 
the space of $f$-tuples $(\mu_i)$ such that each nonzero term of 
$\mu_i$ has degree congruent to $c_i - d_i \pmod{e(K'/K)}$, and set 
$\Czero = \Czerofrac \cap \F[[u]]^{\Z/f\Z}$. 

We further define $\Conefrac  = \Conefrac(\gM,\gN) \subset
\F((u))^{\Z/f\Z}$ to be
 the space of $f$-tuples $(h_i)$ such that each nonzero term of $h_i$
 has degree congruent to $r_i + c_i - d_i \pmod{e(K'/K)}$, and set
 $\Cone = \Conefrac \cap \F[[u]]^{\Z/f\Z}$.   There is a map $\mumap \colon \Czerofrac \to 
\Conefrac$ defined by
\[ \mumap(\mu_i) =  (-a_i u^{r_i} \mu_i + b_i \varphi(\mu_{i-1}) u^{s_i}) \] 
Evidently this restricts to a map
$\mumap \colon \Czero \to \Cone$.
\end{defn}

\begin{lemma}\label{lem:explicit-complex}
There is an isomorphism of complexes 
\[ [ \Czero \xrightarrow{\mumap} \Cone  ] \toisom C^{\bullet}(\gN)\]
in which $(\mu_i) \in \Czero$ is sent to the map $m_i \mapsto \mu_i n_i$
in $C^0(\gN)$, and $(h_i) \in \Cone$ is sent to the map $(1\otimes
m_{i-1}) \mapsto h_i n_i$ in $C^1(\gN)$.
\end{lemma}

\begin{proof}
 Each element of $\Hom_{\gS_{\F}}(\gM,\gN)$ has the form $m_i \mapsto \mu_i
 n_i$ for some $f'$-tuple $(\mu_i)_{i \in \Z/f'\Z}$ of elements of $\F[[u]]$.  
The condition that this map 
is $\Gal(K'/K)$-equivariant 
 is easily seen to be
equivalent to the conditions that $(\mu_i)$ is periodic with period dividing $f$, and
that each nonzero term of $\mu_i$ has degree congruent to 
$c_{i}-d_{i} \pmod{e(K'/K)}$. (For the former consider the action
of a lift to $g \in \Gal(K'/K)$ satisfying $h(g)=1$ of a generator of $\Gal(k'/k)$, and for the
latter consider the action of $I(K'/K)$;\ \emph{cf}.\ the proof of
\cite[Lem.~1.5]{DiamondSavitt}.)   It follows that the map $\Czero
\to C^0(\gN)$ in the statement of the Lemma is an isomorphism. An
essentially identical argument shows that the given map $\Cone \to
C^1(\gN)$ is an isomorphism. 

To conclude, it suffices to observe that if $\alpha \in C^0(\gN)$ is given by $m_i
\mapsto \mu_i n_i$ with $(\mu_i)_i \in \Czero$ then
$\delta(\alpha) \in C^1(\gN)$ is the map
given by $$(1\otimes m_{i-1}) \mapsto (-a_i u^{r_i} \mu_i + b_i
\varphi(\mu_{i-1}) u^{s_i}) n_i,$$ which follows by a direct calculation.
\end{proof}

It follows from Corollary~\ref{cor:complex computes Hom and Ext} 
that $\Ext^1_{\K{\F}}(\gM,\gN) \cong \coker \mumap$.
If $h \in \Cone$, we write $\gP(h)$ for the element of
 $\Ext^1_{\K{\F}}(\gM,\gN)$ represented by $h$ under this isomorphism.

\begin{remark}
  \label{prop: extensions of rank one Kisin modules}Let $\gM=\gM(r,a,c)$ and
  $\gN=\gM(s,b,d)$ be rank one Breuil--Kisin modules with descent data as in
Lemma~{\em \ref{lem:rank one Kisin modules with descent data}}. It
follows from the proof of Lemma~\ref{lem: C computes Ext^1}, and in
particular the description of the map~\eqref{eqn:explicit 
  embedding}  found there, that the extension $\gP(h)$  is given by
the formulas 
  \begin{itemize}
  \item  $\gP_i=\F[[u]]\cdot m_i + \F[[u]]\cdot n_i$,
  \item $\Phi_{\gP,i}(1\otimes n_{i-1})=b_{i} u^{s_{i}}n_{i}$,
    $\Phi_{\gP,i}(1\otimes m_{i-1})=a_{i}u^{r_{i}}m_{i}+h_{i}
    n_{i}$.
  \item $\ghat(\sum_i m_i)=\sum_i h(g)^{c_i}m_i$,
    $\ghat(\sum_i n_i)=\sum_i h(g)^{d_i}  n_i$  for all $g\in\Gal(K'/K)$.
  \end{itemize}
From this description it is easy to see that the extension $\gP(h)$
has  height at most $1$ if and only if  each $h_i$ is divisible by $u^{r_i+s_i-e'}$ whenever $r_i+s_i-e'$ is non-negative.
\end{remark}

\begin{thm}\label{thm: extensions of rank one Kisin modules}
The dimension of $\Ext^1_{\K{\F}}(\gM,\gN)$ is given by the formula
\[\Delta+\sum_{i=0}^{f-1}\#\biggl\{j\in[0,r_i):j\equiv r_i+c_{i}-d_{i}\pmod{e(K'/K)}\biggr\} \]  
where $\Delta =  \dim_{\F} \Hom_{\K{\F}}(\gM,\gN)$ is $1$ if there is a nonzero map $\gM\to\gN$ and $0$
otherwise, while the subspace consisting of extensions of height at most $1$
has dimension
\[\Delta+\sum_{i=0}^{f-1}\#\biggl\{j\in[\max(0,r_i+s_i-e'),r_i):j\equiv r_i+c_{i}-d_{i}\pmod{e(K'/K)}\biggr\} .\] 
\end{thm}

\begin{proof}
  When $p > 2$, this result (for extensions of height at most $1$) can be obtained by translating
  \cite[Thm.~1.11]{DiamondSavitt} from Breuil modules to Breuil--Kisin
  modules. 
  We argue in the same spirit as \cite{DiamondSavitt} using  the generalities of Section~\ref{subsec:ext generalities}. 

Choose~$N$ as in
Lemma~\ref{lem:truncation argument used to prove f.g. of Ext Q
  version}(2). 
For brevity we
write $C^{\bullet}$ in lieu of $C^{\bullet}(\gN)$. We now use the
description of~$C^{\bullet}$ provided by
Lemma~\ref{lem:explicit-complex}. 
As we have noted, $C^0$ consists of the maps $m_i \mapsto \mu_i
n_i$ with $(\mu_i) \in \Czero$.
Since $(\varphi^*_{\gM})^{-1}(v^N C^1)$ contains precisely the maps $m_i
\mapsto \mu_i n_i$ in $C^0$ such that $v^{N} \mid u^{r_i} \mu_i$, we
compute that $\dim_{\F} C^0/\bigl((\varphi^*_{\gM})^{-1}(v^N C^1)\bigr)$
is the quantity
$$ Nf - \sum_{i=0}^{f-1} \#\biggl\{ j \in [e(K'/K) N-r_i, e(K'/K) N) : j \equiv c_i-d_i
  \pmod{e(K'/K)}\biggr\}.$$ We have
$\dim_{\F} C^1/v^NC^1 = Nf$, so
our formula for the dimension of $\Ext^1_{\K{\F}}(\gM,\gN)$ now follows
from Lemma~\ref{lem:truncation argument used to prove f.g. of Ext Q
  version}.
\end{proof}

\begin{remark}\label{rem:representatives-for-ext}
One can show exactly as in \cite{DiamondSavitt} that each element of $\Ext^1_{\K{\F}}(\gM,\gN)$ can be written uniquely in
the form $\gP(h)$ for $h \in \Cone$
with $\deg(h_i) < r_i$, except that when there exists a nonzero morphism
$\gM\to\gN$, the polynomials $h_i$ for $f \mid i$ may also have a term of degree
$\alpha_0(\gM)-\alpha_0(\gN)+r_0$ in common. Since we will not need
this fact we omit the proof.
\end{remark}

\subsection{Extensions of profile \texorpdfstring{$J$}{J}}
\label{sec:extensions-profile-J}

We now begin the work of showing, for each non-scalar tame type $\tau$,
that $\cC^{\tau,\BT,1}$ has $2^f$ irreducible
components, indexed by the subsets $J$ of~$\{0,1,\dots,f-1\}$. We will
also   describe the irreducible
components of~$\cZ^{\tau,1}$. 
The proof of this hinges on examining the extensions considered in
Theorem~\ref{thm: extensions of rank one Kisin modules},
and then applying the results of Subsection~\ref{subsec:universal families}.
We will show that 
most of these families of extensions  have 
positive codimension in $\cC^{\tau,\BT,1}$, and are thus negligible from the
point of view of determining irreducible components.  By a base change
argument, we will also be able to show that we can neglect the irreducible
Breuil--Kisin modules. The rest of Section~\ref{sec: extensions of rank one Kisin modules} is devoted to establishing the
necessary bounds on the dimension of the various families of
extensions, and to studying the map from $\cC^{\tau,\BT,1}$ to
$\cR^{\dd,1}$. 

We now introduce notation that we will use for the remainder of the
paper. 
We fix a
tame inertial type $\tau=\eta\oplus\eta'$ with coefficients in $\Qpbar$.
 We allow the case of scalar
types (that is, the case $\eta=\eta'$). 
Assume that the subfield $\F$ of $\Fpbar$ is large enough so that the reductions modulo $\m_{\Zp}$ of $\eta$ and
$\eta'$ (which by abuse of notation we continue to denote $\eta,\eta'$) have image in $\F$. 
We also fix a uniformiser $\pi$ of~$K$. 

\begin{remark}\label{rk:ordering}
We stress that  when we 
write $\tau=\eta\oplus\eta'$, we are implicitly ordering 
$\eta,\eta'$. Much of the notation in this section depends on 
distinguishing $\eta,\eta'$, as do some of the constructions later in
paper (in particular, those using the  
map to the Dieudonn\'e stack of Section~\ref{sec:dieudonne-stacks}). 
\end{remark}

As in Subsection~\ref{sec:dieudonne-stacks}, we make
the following  ``standard choice'' for the extension~$K'/K$: if $\tau$ is a tame
principal series type, we take $K'=K(\pi^{1/(p^f-1)})$, while
if~$\tau$ is a tame cuspidal type, we let $L$ be an unramified
quadratic extension of~$K$, and set $K'=L(\pi^{1/(p^{2f}-1)})$. In
either case $K'/K$ is a Galois extension and $\eta, \eta'$ both factor through
$I(K'/K)$. In the principal series
case, we have $e'=(p^f-1)e$, $f'=f$, and in the cuspidal case we have
$e'=(p^{2f}-1)e$, $f'=2f$.   Either way, we have $e(K'/K) =
p^{f'}-1$.

In either case, it follows from Lemma~\ref{lem:rank one Kisin modules
  with descent data} that a Breuil--Kisin module of rank one with descent data
from $K'$ to $K$ is described by the data of the quantities $r_i,a_i,c_i$ for $0\le i\le
f-1$, and similarly from Lemma~\ref{lem:explicit-complex} that extensions between two such Breuil--Kisin modules are
described by the $h_i$ for $0\le i\le
f-1$. This common description will enable us to treat the principal
series and cuspidal cases largely in parallel.

The character
$\hchar |_{I_K}$ of Section~\ref{subsec: Diamond--Savitt}  is identified via the Artin map $\cO_L^\times \to 
I_L^{\ab} = I_K^{\ab} $ with the reduction map 
$\cO_L^{\times} \to (k')^{\times}$. Thus for each $\sigma \in 
\Hom(k',\Fpbar)$ the map $\sigma \circ 
\hchar|_{I_L}$ is the fundamental character $\omega_{\sigma}$ defined 
in Section~\ref{subsec: notation}. 
Define  $k_i,k'_i\in \Z/(p^{f'}-1)\Z$ for all $i$ by the formulas
$\eta=\sigma_i\circ\hchar^{k_i}|_{I(K'/K)}$ and
$\eta'=\sigma_i\circ\hchar^{k'_i}|_{I(K'/K)}$. In particular we have
$k_i=p^ik_0$, $k'_i=p^ik'_0$ for all $i$.

\begin{defn} Let $\gM = \gM(r,a,c)$ and $\gN=\gM(s,b,d)$ be Breuil--Kisin
  modules of rank one with descent data. We say that the pair
  $(\gM,\gN)$ has \emph{type $\tau$} provided that for all $i$:
  \begin{itemize}
  \item the multisets $\{c_i,d_i\}$ and $\{k_i,k'_i\}$ are equal, and
  \item $r_i + s_i = e'$.
  \end{itemize}
  \end{defn}
  \begin{lemma} The following are equivalent.
    \begin{enumerate}
    \item   The pair $(\gM,\gN)$ has type~$\tau$. 
    \item  Some element of $\Ext^1_{\K{\F}}(\gM,\gN)$ of height at
      most one  satisfies the strong determinant condition and is of type~$\tau$. 
  \item Every element of $\Ext^1_{\K{\F}}(\gM,\gN)$ has height at most
    one, 
satisfies the strong determinant condition, and is of type~$\tau$. 
    \end{enumerate}
(Accordingly, we will sometimes refer to the condition that
$r_i+s_i=e'$ for all~$i$ as the determinant condition.) 
  \end{lemma}

  \begin{proof}
    Suppose first that the pair $(\gM,\gN)$ has type $\tau$. The last
    sentence of Remark~\ref{prop: extensions of rank one Kisin
      modules} shows that every element of $\Ext^1_{\K{\F}}(\gM,\gN)$
    has height at most one.  Let $\gP$ be such an element. The
    condition on the multisets $\{c_i,d_i\}$ guarantees that $\gP$ has
    type $\tau$ (an \emph{unmixed type} in the sense of \cite[Def.~3.3.2]{cegsB}). By \cegsBtypesareunmixed\
    we see that $\dim_{\F} (\im_{\gP,i}/E(u)\gP_i)_{\xi}$ is
    independent of the character $\xi : I(K'/K) \to \F^{\times}$. From the condition that $r_i+s_i=e'$ we
    know that the sum over all $\xi$ of these dimensions is equal to
    $e'$; since they are all equal, each is equal to $e$, and
    \cegsBdeterminantconditionexplicit\
 tells
    us that $\gP$ satisfies the
    strong determinant condition. This proves that (1) implies (3). 

Certainly (3)
    implies (2), so it remains to check that (2) implies (1). Suppose
    that  
$\gP \in \Ext^1_{\K{\F}}(\gM,\gN)$ has height at most one, 
satisfies the strong determinant condition, and has type $\tau$. The
condition that $\{c_i,d_i\}=\{k_i,k'_i\}$ follows from $\gP$ having
type $\tau$, and the condition that $r_i+s_i=e'$ follows from the last
part of  \cegsBdeterminantconditionexplicit.
  \end{proof}

\begin{df}
\label{df:extensions of profile $J$}
If $(\gM,\gN)$ is a pair of type $\tau$ (resp.\ $\gP$ is an extension 
of type~$\tau$), we define the {\em profile} of $(\gM,\gN)$ (resp.\ of $\gP$) to
be the subset $J := \{ i  \, | \, c_i = k_i\} \subseteq \Z/f'\Z$,
unless $\tau$ is scalar, in which case we define the profile to be the
subset~$\varnothing$. 
(Equivalently, $J$ is in all cases the complement in
$\Z/f'\Z$ of the set  $\{i \, | \, c_i = k'_i\}.$)

Observe that in the cuspidal case the equality $c_i = c_{i+f}$ means
that $i \in J$ if and only if $i+f \not\in J$, so that the set $J$ is
determined by its intersection with any $f$ consecutive integers
modulo $f' = 2f$.  

In the cuspidal case we will say that a subset $J \subseteq \Z/f'\Z$ is a
  profile
if it satisfies $i \in J$ if and only if $i+f\not\in J$; in the
principal series case, we may refer to any subset $J \subseteq \Z/f'\Z$
as a profile.

We define the {\em refined profile} of the pair $(\gM,\gN)$ (resp.\ of $\gP$) to consist of its profile $J$,
together with the $f$-tuple of invariants $r:= (r_i)_{i = 0}^{f-1}$.
If $(J,r)$ is a refined profile that arises from some pair (or extension)
of type $\tau$, then we refer to $(J,r)$ as a refined profile for $\tau$.

We say the pair $(i-1,i)$ is a \emph{transition} for $J$ if $i-1 \in
J$, $i \not\in J$ or vice-versa. (In the first case we sometimes say
that the pair $(i-1,i)$ is a transition out of $J$, and in the latter case a transition
into $J$.) Implicit in many of our arguments below
is the observation that in the cuspidal case $(i-1,i)$ is a transition
if and only if $(i+f-1,i+f)$ is a transition.
\end{df}

\subsubsection{An explicit description of refined profiles}
\label{subsubsec:explicitly refined}
The refined profiles for $\tau$ admit an explicit description.
If $\gP$ is of profile $J$, for some fixed $J \subseteq \Z/f'\Z$
then, since $c_i$, $d_i$ are fixed, we see that the $r_i$ and $s_i$
appearing in $\gP$ are determined
modulo $e(K'/K)=p^{f'}-1$. Furthermore, we see that $r_i+s_i\equiv
0\pmod{p^{f'}-1}$, so that these values are consistent with the 
determinant condition; conversely, if we make any choice of the~$r_i$
in the given residue class modulo $(p^{f'}-1)$, then the $s_i$ are
determined by the determinant condition, and the imposed values
are consistent with the descent data. There are of course only
finitely many choices for the~$r_i$, and so there are only finitely 
many possible refined profiles for $\tau$.

To make this precise, recall that we have the congruence
$$r_i \equiv
pc_{i-1} - c_i \pmod{p^{f'}-1}.$$ We will write $[n]$ for the
least non-negative residue class of $n$ modulo $e(K'/K) =
p^{f'}-1$.

If both $i-1$ and $i$ lie in $J$,
then we have $c_{i-1} = k_{i-1}$ and  $c_i = k_i$. The first of these
implies that $pc_{i-1} = k_i$, and therefore $r_i \equiv 0
\pmod{p^{f'}-1}$. The same conclusion holds if neither $i-1$ and $i$
lie in $J$. Therefore if $(i-1,i)$ is not a transition we may write 
\[ r_i =
  (p^{f'}-1)y_i \quad \text{ and } \quad s_i = (p^{f'}-1)(e-y_i).\]  
with $0 \le y_i \le e$.

Now suppose instead that $(i-1,i)$ is a transition. (In
particular the type $\tau$ is not scalar.)  This
time $pc_{i-1} = d_i$ (instead of $pc_{i-1} = c_i$), so that $r_i
\equiv d_i - c_i \pmod{p^{f'}-1}$. In this case we write
 \[ r_i =
  (p^{f'}-1)y_i - [c_i-d_i] \quad \text{ and } \quad s_i = (p^{f'}-1)(e+1-y_i) -
  [d_i-c_i]\]
with $1 \le y_i \le e$.

Conversely, for fixed profile $J$ one checks that each choice of integers $y_i$ in the ranges
described above gives rise to a refined profile for $\tau$.

If $(i-1,i)$ is not a transition 
and $(h_i) \in \Conefrac(\gM,\gN)$ then non-zero terms of $h_i$ have
degree congruent to $r_i + c_i - d_i \equiv c_i - d_i \pmod{p^{f'}-1}$.
If instead $(i-1,i)$ is a transition 
and $(h_i) \in \Conefrac(\gM,\gN)$ then non-zero terms
of $h_i$  have degree congruent to $r_i + c_i - d_i \equiv 0 \pmod{p^{f'}-1}$.
In either case, comparing with the preceding paragraphs we see that $\#\{ j \in 
  [0,r_i) : j \equiv  r_i + c_i - d_i \pmod{e(K'/K)}\}$ is exactly~$y_i$. 

By Theorem~\ref{thm: extensions of rank one Kisin modules}, we conclude
that for a fixed choice of the~$r_i$ 
the dimension of the
corresponding~$\Ext^1$ is $\Delta + \sum_{i=0}^{f-1} y_i$ (with
$\Delta$ as in the statement of \emph{loc.\ cit.}).
We say that the refined profile $\bigl(J, (r_i)_{i =0}^{f-1}\bigr)$
is \emph{maximal} if the $r_i$ are chosen to be maximal subject to the
above conditions, or equivalently if the $y_i$ are all chosen to be $e$; for each
profile~$J$, there is a unique maximal refined profile~$(J,r)$.

\subsubsection{The sets $\cP_{\tau}$}
\label{sec:sets-cp_tau}

To each tame type $\tau$ we now associate a set $\cP_{\tau}$, which
will be a subset of the set of profiles in $\Z/f'\Z$.  (In Appendix~\ref{sec: appendix on tame
  types} we will recall, following~\cite{MR2392355}, that the 
set $\cP_{\tau}$ parameterises the Jordan--H\"older factors of the
reduction mod~$p$ of 
$\sigma(\tau)$.)

We write $\eta (\eta')^{-1} =
\prod_{j=0}^{f'-1} (\sigma_j \circ \hchar)^{\gamma_j}$ for uniquely defined integers $0
\le \gamma_j \le p-1$ not all equal to $p-1$, so that
\begin{equation}\label{eq:k-gamma}
[k_i - k'_i] = \sum_{j=0}^{f'-1} p^{j} \gamma_{i-j} 
\end{equation}
with subscripts taken modulo $f'$.

If~$\tau$ is scalar then
we set $\cP_\tau=\{\varnothing\}$. Otherwise we let 
$\cP_{\tau}$ be the collection of profiles $J \subseteq \Z/f'\Z$ 
satisfying the conditions:
\begin{itemize}
\item if $i-1\in J$ and $i\notin J$ then $\gamma_{i}\ne p-1$, and
\item if $i-1\notin J$ and $i\in J$ then $\gamma_{i}\ne 0$.
\end{itemize}
When $\tau$ is a cuspidal type, so that $\eta' = \eta^q$, the integers 
$\gamma_j$ satisfy $\gamma_{i+f} = p-1-\gamma_i$ for all $i$; thus the
condition that if $(i-1,i)$ is a transition out of $J$ then $\gamma_i
\neq p-1$ translates exactly into the condition that if $(i+f-1,i+f)$ is a
transition into $J$ then $\gamma_{i+f} \neq 0$.

\subsubsection{Moduli stacks of extensions}\label{subsubsection: stacks of extensions}

We now apply the constructions of stacks and topological spaces of
Definitions~\ref{def:scheme-theoretic images}
and~\ref{df:constructible images} to the families of
extensions considered in Section~\ref{sec:extensions-profile-J}. 

\begin{df}\label{defn: M(j,r)}
If $(J,r)$ is a refined profile for $\tau$, then 
we let $\gM(J,r) := \gM(r,1,c)$ and let
$\gN(J,r) := \gM(s,1,d),$ where $c$, $d$, and $s$ are determined
from $J$, $r$, and $\tau$ according to the 
discussion of~(\ref{subsubsec:explicitly refined}); for instance we
take $c_i
= k_i$ when $i \in J$ and $c_i = k'_i$ when $i \not\in J$.
For the unique
maximal profile~$(J,r)$ refining~$J$, we write simply $\gM(J)$ and $\gN(J)$.
\end{df}

\begin{df}
If $(J,r)$ is a refined profile for $\tau$, then following
Definition~\ref{def:scheme-theoretic images}, we may construct the reduced
closed substack $\overline{\cC}\bigl(\gM(J,r),\gN(J,r)\bigr)$ of $\cC^{\tau,\BT,1}$,
as well as the reduced closed substack $\overline{\cZ}\bigl(\gM(J,r),
\gN(J,r)\bigr)$ of $\cZ^{\tau,1}.$
We introduce the notation $\overline{\cC}(J,r)$ and $\overline{\cZ}(J,r)$
for these two stacks, and note that (by definition) $\overline{\cZ}(J,r)$
is the scheme-theoretic image of $\overline{\cC}(J,r)$ under
the morphism $\cC^{\tau,\BT,1} \to \cZ^{\tau,1}$.
\end{df}

\begin{remark}\label{rem-all-pts} As noted in the final sentence of
  Definition~\ref{def:scheme-theoretic images}, Lemma~\ref{lem: scheme theoretic
    images coincide} shows that $\overline{\cC}(J,r)$ contains all
  extensions of refined profile $(J,r)$ over extensions of~$\F$, and not
  only those corresponding to a maximal ideal of $A^{\dist}$.  
  \end{remark}

\begin{thm}\label{thm: dimension of refined profiles} 
If $(J,r)$ is any refined profile for $\tau$,
then $\dim \overline{\cC}(J,r) \leq [K:\Qp],$ with equality if and only if $(J,r)$ is
maximal.
\end{thm}
\begin{proof}
  This follows from  Corollary~\ref{cor:
    dimension of families of extensions}, Theorem~\ref{thm:
    extensions of rank one Kisin modules}, and Proposition~\ref{prop:base-change for exts}.  (See also the discussion
  following Definition~\ref{df:extensions of profile $J$}, and note that over
  $\Spec \Adist$, we have $\Delta=0$ by definition.)
\end{proof}

\begin{df} If $J \subseteq \Z/f'\Z$ is a profile, and if $r$ is chosen so that
$(J,r)$ is a maximal refined profile for $\tau$,
then we write $\overline{\cC}(J)$ to denote the closed substack $\overline{\cC}(J,r)$
of $\cC^{\tau,\BT,1}$,
and $\overline{\cZ}(J)$ to denote the closed substack
$\overline{\cZ}(J,r)$ of $\cZ^{\tau,1}$.
Again, we note that by definition
$\overline{\cZ}(J)$ is the scheme-theoretic
image of~$\overline{\cC}(J)$ in~$\cZ^{\tau,1}$. 
\end{df}

We will see later that the~$\overline{\cC}(J)$ are precisely the
irreducible components of~$\cC^{\tau,\BT,1}$; in particular, their
finite type points can correspond to irreducible Galois
representations. While we do not need it in the sequel, we note the
following definition and result, describing the underlying topological
spaces of the loci of reducible Breuil--Kisin modules of fixed refined profile.
\begin{defn}
  For each refined type~$(J,r)$, we write~$|\cC(J,r)^\tau|$ for the
  constructible subset~$|\cC(\gM(J,r),\gN(J,r))|$
  of~$|\cC^{\tau,\BT,1}|$ of Definition~\ref{df:constructible images}
  (where $\gM(J,r)$, $\gN(J,r)$ are the Breuil--Kisin modules
  of Definition~\ref{defn: M(j,r)}). We
  write~$|\cZ(J,r)^\tau|$ for the image of~$|\cC(J,r)^\tau|$
  in~$|\cZ^{\tau,1}|$ (which is again a constructible
  subset). 
\end{defn}

\begin{lem}
  \label{lem: closed points of C(J,r)}The $\Fpbar$-points
  of~$|\cC(J,r)^\tau|$ are precisely the reducible Breuil--Kisin modules
  with $\Fpbar$-coefficients
  of type~$\tau$ and refined profile~$(J,r)$.
\end{lem}
\begin{proof}
  This is immediate from the definition.
\end{proof}

\section{Components of Breuil--Kisin and Galois moduli stacks}
\label{sec:Components}

Now that we have constructed the morphisms $\overline{\cC}(J) \to \overline{\cZ}(J)$ for each $J$, we can begin our study of the components of the stacks $\cC^{\tau,\BT,1}$ and $\cZ^{\tau,1}$. The first step in Subsection~\ref{subsec:vertical-comps} is to determine precisely for which $J$ the scheme-theoretic image $\overline{\cZ}(J)$ has dimension smaller than $[K:\Qp]$, and hence is \emph{not} a component of $\cZ^{\tau,1}$. In Section~\ref{subsec:irreducible} we  study the irreducible locus in $\cC^{\tau,\BT,1}$ and prove that it lies in a closed substack of positive codimension. We are then ready to establish our main results in Subsections~\ref{subsec: irred components} and~\ref{subsec: map to
  Dieudonne stack}.

\subsection{\texorpdfstring{$\kExt^1$}{ker-Ext} and vertical
  components}\label{subsec:vertical-comps}
In this section we will establish some basic facts about $\kExt^1_{\K{\F}}(\gM,\gN)$,
and use these results to study the images of our
irreducible components in $\cZ^{\tau,1}$. 
 Let $\gM = \gM(r,a,c)$ and $\gN = \gM(s,b,c)$ be Breuil--Kisin 
modules as in Section~\ref{subsec:
  Diamond--Savitt}.

Recall from \eqref{eqn: computing kernel of Ext groups} that the
dimension of $\kExt_{\K{\F}}(\gM,\gN)$ is bounded above by the
dimension of $\Hom_{\K{\F}}(\gM,\gN[1/u]/\gN)$;\ more precisely, by
Lemma~\ref{lem: Galois rep is a functor if A is actually finite local}
we find in this setting that 
\numequation\label{eq:ker-ext-formula}\begin{split} \dim_{\F}  \kExt^1_{\K{\F}}(\gM,\gN) = \dim_{\F}
\Hom_{\K{A}}(\gM,\gN[1/u]/\gN)  \\- (\dim_{\F} \Hom_{\F[G_K]}(T(\gM),T(\gN)) -
\dim_{\F} \Hom_{\K{A}}(\gM,\gN)).
\end{split}
\end{equation}

A map $f : \gM \to \gN[1/u]/\gN$ has the form
$f(m_i) = \mu_i n_i$ for some $f'$-tuple of elements $\mu_i \in
\F((u))/\F[[u]]$. By the same argument as in the first paragraph of the proof of
Lemma~\ref{lem:explicit-complex}, such a map belongs to
$C^0(\gN[1/u]/\gN)$ (i.e., it is $\Gal(K'/K)$-equivariant) if and only
if the $\mu_i$ are periodic with
period dividing $f$, and each nonzero term of $\mu_i$ has degree
congruent to $c_i-d_i \pmod{e(K'/K)}$.  One computes that
$\delta(f)(1\otimes m_{i-1}) = (u^{s_i} \varphi(\mu_{i-1}) - u^{r_i}
\mu_i)n_i$ and so $f \in C^0(\gN[1/u]/\gN)$ lies in  $\Hom_{\K{\F}}(\gM,\gN[1/u]/\gN)$
precisely when
\numequation\label{eq:phi-commute-ker-ext} a_i
u^{r_i} \mu_i = b_i \varphi(\mu_{i-1}) u^{s_i}\end{equation} for all
$i$.

\begin{remark}\label{rem:explicit-ker-ext}
Let $f \in  \Hom_{\K{\F}}(\gM,\gN[1/u]/\gN)$ be given as above. Choose
any lifting $\tilde{\mu}_i$ of $\mu_i$ to $\F((u))$. Then (with
notation as in~Definition~\ref{notn:calh}) the tuple
$(\tilde{\mu}_i)$ is an element of $\Czerofrac$,  and we define $h_i =
\mumap(\tilde{\mu}_i)$. Then
$h_i$ lies in $\F[[u]]$ for all $i$, so that $(h_i) \in \Cone$, and a
comparison with the construction of~\eqref{eqn: computing kernel of Ext groups} and 
the proof of Lemma~\ref{lem:explicit-complex} shows that $f$ maps
to the extension class in $\kExt^1_{\K{\F}}(\gM,\gN)$ represented by $\gP(h)$.
\end{remark}

Recall that Lemma~\ref{lem: bound on torsion in kernel of Exts}
 implies that  
nonzero terms appearing in $\mu_i$ have degree at least $-\lfloor
e'/(p-1) \rfloor$. From this we obtain the following trivial bound on $\kExt$.

\begin{lemma}\label{cor:bounds-on-ker-ext}
 We have $\dim_{\F} \kExt^1_{\K{\F}}(\gM,\gN) \le \lceil e/(p-1)
 \rceil f$. 
\end{lemma}

\begin{proof} The degrees of nonzero terms of $\mu_i$ all lie in a single
  congruence class modulo $e(K'/K)$, and are bounded below by
  $-e'/(p-1)$. Therefore
  there are at most $\lceil e/(p-1) \rceil$ nonzero terms, and since
  the $\mu_i$ are periodic with period dividing $f$ the lemma follows.
\end{proof}

\begin{remark}\label{rem:half}
It follows directly from Corollary~\ref{cor:bounds-on-ker-ext} that if
$p > 3$ and $e\neq 1$ then we have  $\dim_{\F}
\kExt^1_{\K{\F}}(\gM,\gN) \le [K:\Qp]/2$, for then $\lceil e/(p-1)
\rceil \le e/2$. Moreover these inequalities are strict if $e >
2$. 
\end{remark}

We will require a more precise computation of
$\kExt^1_{\K{\F}}(\gM,\gN)$ in the setting of
Section~\ref{sec:extensions-profile-J} where the pair $(\gM,\gN)$ has 
maximal refined profile $(J,r)$.  We now return to that  setting and its
notation.

Let $\tau$ be a tame type. We
will find the following notation to be helpful. We let $\gamma_i^* =
\gamma_i$ if $i-1 \not\in J$, and $\gamma_i^* = p-1-\gamma_i$
if $i-1 \in J$. (Here the integers $\gamma_i$ are as in
Section~\ref{sec:sets-cp_tau}. In the case of scalar types this means
that we have $\gamma^*_i = 0$ for all $i$.)
Since $p[k_{i-1}-k'_{i-1}] - [k_i - k'_i] = (p^{f'}-1)\gamma_i$, 
 an elementary but useful calculation
shows that
\numequation\label{eq:gammastar}
p[d_{i-1}-c_{i-1}] - [c_i-d_i] = \gamma_i^* (p^{f'}-1),
\end{equation}
when $(i-1,i)$ is a transition, and that in this case $\gamma_i^*
=0$ if and only if $[d_{i-1}-c_{i-1}] < p^{f'-1}$. Similarly, if
$\tau$ is not a scalar type and $(i-1,i)$ is not a transition then
\numequation\label{eq:gammastar-2}
p[d_{i-1}-c_{i-1}] + [c_i-d_i] = (\gamma_i^*+1) (p^{f'}-1).
\end{equation}

The main computational result of this section is the following.

\begin{prop}\label{prop:ker-ext-maximal}
Let $(J,r)$ be any maximal refined profile for $\tau$, and suppose that
the pair $(\gM,\gN)$ has refined profile $(J,r)$.    Then  
$\dim_{\F} \kExt^1_{\K{\F}} (\gM,\gN)$ is equal to 
\[\# \{ 0 \le i < f \, : \, \text{the pair } (i-1,i) \text{ is a
  transition and } \gamma_i^* = 0 \},\]
except that when $e=1$, $\prod_i a_i = \prod_i b_i$,  and the quantity displayed above is $f$,
 then 
the dimension of $\kExt^1_{\K{\F}} (\gM,\gN)$ is equal to $f-1$.
\end{prop}

\begin{proof}
 The argument has two parts. First we show that $\dim_{\F}
 \Hom_{\K{\F}}(\gM,\gN[1/u]/\gN)$ is precisely the displayed quantity
 in the statement of the Proposition;\ then we check that
 $\dim_{\F} \Hom_{\F[G_K]} (T(\gM),T(\gN)) - \dim_{\F} \Hom_{\K{\F}} (\gM,\gN)$ is
 equal to $1$ in the exceptional case of the statement, and $0$
 otherwise. The result then follows from~\eqref{eq:ker-ext-formula}.

Let  $f : m_i \mapsto \mu_i n_i$ be an element of 
$C^0(\gN[1/u]/\gN)$. 
Since $u^{e'}$ kills $\mu_i$, and all nonzero terms of $\mu_i$ have
degree congruent to
$c_i-d_i \pmod{p^{f'}-1}$, certainly all nonzero terms of $\mu_i$ have
degree at least $-e' + [c_i-d_i]$. On the other hand since the profile
$(J,r)$ is maximal we have $r_i = e' - [c_i-d_i]$ when $(i-1,i)$ is a
transition and $r_i = e'$ otherwise. In either case $u^{r_i}$ kills
$\mu_i$, so that \eqref{eq:phi-commute-ker-ext} becomes simply the
condition that $u^{s_i}$ kills $\varphi(\mu_{i-1})$.

If $(i-1,i)$ is not a transition then $s_i=0$, and we conclude that
$\mu_{i-1}=0$. Suppose instead that $(i-1,i)$ is a transition, so
that $s_i = [c_i-d_i]$. Then all nonzero terms of $\mu_{i-1}$ have
degree at least $-s_i/p > -p^{f'-1} > -e(K'/K)$. Since those terms must have
degree congruent to $c_{i-1}-d_{i-1}  \pmod{p^{f'}-1}$, it follows
  that $\mu_{i-1}$ has at most one nonzero term (of degree
  $-[d_{i-1}-c_{i-1}]$), and this only if $[d_{i-1}-c_{i-1}] <
  p^{f'-1}$, or equivalently $\gamma_i^* = 0$ (as noted
  above). Conversely if $\gamma_i^*=0$ then 
\[ u^{s_i} \varphi(u^{-[d_{i-1}-c_{i-1}]}) = u^{[c_i-d_i] -
    p[d_{i-1}-c_{i-1}]} = u^{-\gamma_i^* (p^{f'}-1)}\] vanishes in
$\F((u))/\F[[u]]$. We conclude that $\mu_{i-1}$ may have a single nonzero term if
and only if $(i-1,i)$ is a transition and $\gamma_i^*=0$, and this
completes the first part of the argument.

Turn now to the second part. 
Looking at Corollary~\ref{cor: Kisin
  modules with the same generic fibre} and Lemma~\ref{lem: maps
  between rank 1 Kisin modules}, to compare
$\Hom_{\F[G_K]}(T(\gM),T(\gN))$ and $\Hom_{\K{\F}}(\gM,\gN)$ we need
to compute the quantities $\alpha_i(\gM)-\alpha_i(\gN)$. By definition
this quantity is equal
to 
\numequation\label{eq:alpha-difference-1}\frac{1}{p^{f'}-1} \sum_{j=1}^{f'}  p^{f'-j} \left( r_{i+j} - s_{i+j}
\right).\end{equation}
Suppose first that $\tau$ is non-scalar. When $(i+j-1,i+j)$ is a transition, we have $r_{i+j}-s_{i+j} = (e-1)(p^{f'}-1) +
[d_{i+j}-c_{i+j}] - [c_{i+j}-d_{i+j}]$, and otherwise we
have $r_{i+j}-s_{i+j} = e( p^{f'}-1)=  (e-1)(p^{f'}-1) + [d_{i+j}-c_{i+j}] +
[c_{i+j}-d_{i+j}]$. Substituting these expressions into
\eqref{eq:alpha-difference-1}, adding and subtracting $\frac{1}{p^{f'}-1}
p^{f'} [d_i-c_i]$, and regrouping gives
\[ -[d_i-c_i] + (e-1) \cdot \frac{p^{f'}-1}{p-1} + \frac{1}{p^{f'}-1}\sum_{j=1}^{f'}
p^{f'-j} \left( p[d_{i+j-1} - c_{i+j-1}] \mp [c_{i+j} - d_{i+j}]
\right),\]
where the sign is $-$ if $(i+j-1,i+j)$ is a transition and $+$ if
not. Applying the formulas~\eqref{eq:gammastar}
and~\eqref{eq:gammastar-2} we conclude that
\numequation\label{eq:alpha-difference}
\alpha_i(\gM)-\alpha_i(\gN) = -[d_i-c_i] +  (e-1) \cdot
\frac{p^{f'}-1}{p-1} + \sum_{j=1}^{f'} p^{f'-j} \gamma^*_{i+j} +
\sum_{j \in S_i} p^{f'-j}
\end{equation}
where the set $S_i$ consists of $1 \le j \le f$ such that $(i+j-1,i+j)$ is not a transition.
Finally, a moment's inspection shows that the same formula still holds
if $\tau$ is scalar (recalling that $J = \varnothing$ in that case).

Suppose that we are in the exceptional case of the proposition, so
that $e=1$, $\gamma_i^*=0$ for all $i$, and every pair $(i-1,i)$ is a
transition. The formula~\eqref{eq:alpha-difference} gives
$\alpha_i(\gM)-\alpha_i(\gN) = -[d_i-c_i]$. Since also $\prod_i a_i =
\prod_i b_i$  the conditions of
Corollary~\ref{cor: Kisin modules with the same generic fibre} are
satisfied, so that $T(\gM)=T(\gN)$;\ but on the other hand
$\alpha_i(\gM) < \alpha_i(\gN)$, so that by Lemma~\ref{lem: maps
  between rank 1 Kisin modules} there are no nonzero maps $\gM\to
\gN$, and $\dim_{\F} \Hom_{\F[G_K]} (T(\gM),T(\gN)) - \dim_{\F}
\Hom_{\K{\F}} (\gM,\gN) = 1.$

If instead we are not in the exceptional case of the
proposition, then either $\prod_i a_i \neq \prod_i b_i$, or else~\eqref{eq:alpha-difference}
gives $\alpha_i(\gM) - \alpha_i(\gN) > -[d_i-c_i]$ for all
$i$. Suppose that $T(\gM) \cong T(\gN)$. It follows from
Corollary~\ref{cor: Kisin modules with the same generic fibre} that
$\alpha_i(\gM) - \alpha_i(\gN) \equiv -[d_i-c_i]
\pmod{e(K'/K)}$. Combined with the previous inequality we deduce that
$\alpha_i(\gM) - \alpha_i(\gN)  > 0$, and Lemma~\ref{lem: maps between
  rank 1 Kisin modules} guarantees the existence of a nonzero map $\gM
\to \gN$. We deduce that in any event $\dim_{\F} \Hom_{\F[G_K]}
(T(\gM),T(\gN)) = \dim_{\F} \Hom_{\K{\F}} (\gM,\gN)$, completing the proof.
\end{proof}

\begin{cor}\label{cor:ker-ext-maximal-nonzero}
Let $(J,r)$ be any maximal refined profile for $\tau$, and suppose that
the pair $(\gM,\gN)$ has refined profile $(J,r)$.     If $J \in
\cP_{\tau}$ then $\dim_{\F} \kExt^1_{\K{\F}} (\gM,\gN) = 0$. Indeed
this is an if and only if except possibly when $K=\Qp$, the type
$\tau$ is cuspidal, and $T(\gM(J,r)) \cong T(\gN(J,r))$.
\end{cor}

\begin{proof}
  The first statement is immediate from Proposition~\ref{prop:ker-ext-maximal},
  comparing the definition of $\gamma_i^*$ with the defining
 condition on elements of $\cP_{\tau}$;\  in fact this gives an if and
 only if unless we are in the exceptional case in
 Proposition~\ref{prop:ker-ext-maximal} and $f-1=0$. In that case
 $e=f=1$, so $K=\Qp$. In the principal series case for $K=\Qp$ there
 can be no transitions, so the type is cuspidal. Then $\gamma_i^*=0$
 for $i=0,1$ and an elementary
 analysis of \eqref{eq:gammastar} shows that there exists  $x \in \Z/(p-1)\Z$
 such that $c_i = 1 +x(p+1)$, $d_i = p+x(p+1)$ for $i=0,1$. Then $r_i
 = p-1$ and $s_i = p(p-1)$, and Lemma~\ref{lem: generic fibres of rank
   1 Kisin modules} gives $T(\gM(J,r)) \cong T(\gN(J,r))$.
\end{proof}

Recall 
that
$\overline{\cZ}(J)$ is by definition the scheme-theoretic image of~$\overline{\cC}(J)$ in
$\cZ^{\tau,1}$. In the remainder of this section, we show that the $\overline{\cZ}(J)$
with $J\in\cP_{\tau}$ are pairwise distinct irreducible components
of~$\cZ^{\tau,1}$. In Section~\ref{subsec: irred components} below we
will show that they in fact exhaust the irreducible components of~$\cZ^{\tau,1}$.
\begin{thm}\label{thm: identifying the vertical components} 
$\overline{\cZ}(J)$ has dimension at most $[K:\Qp]$, with equality occurring if
and only if~$J\in\cP_{\tau}$. Consequently, the~$\overline{\cZ}(J)$ with
$J\in\cP_{\tau}$ are irreducible components of~$\cZ^{\tau,1}$.
\end{thm}
\begin{proof}The first part is immediate from Corollary~\ref{cor: dimension of
    families of extensions}, Proposition~\ref{prop:base-change for
    exts}, 
  Corollary~\ref{cor:ker-ext-maximal-nonzero} and Theorem~\ref{thm:
    dimension of refined profiles} (noting that the exceptional case of
  Corollary~\ref{cor:ker-ext-maximal-nonzero}  occurs away from
  $\textrm{max-Spec}\, 
  A^{\dist}$). Since~$\cZ^{\tau,1}$ is
  equidimensional of dimension~$[K:\Qp]$ by Theorem~\ref{prop:
    dimensions of the Z stacks}, and the~$\overline{\cZ}(J)$ are closed and
  irreducible by construction, the second part follows from the first
  together with~\cite[\href{https://stacks.math.columbia.edu/tag/0DS2}{Tag 0DS2}]{stacks-project}.
\end{proof}

We also note the following result.

\begin{prop}
\label{prop:C to Z mono}
If~$J\in\cP_{\tau}$, then there is a dense open substack $\cU$ of 
$\overline{\cC}(J)$ such that the canonical morphism $\overline{\cC}(J) 
\to \overline{\cZ}(J)$ restricts to an open immersion on $\cU$.
\end{prop}
\begin{proof}
This follows from 
Proposition~\ref{prop: construction of family monomorphing to C and R}
and Corollary~\ref{cor:ker-ext-maximal-nonzero}.
\end{proof}

For later use, we note the following computation. Recall that we write
$\gN(J) = \gN(J,r)$ for the maximal profile $(J,r)$ refining $J$, and that~$\tau=\eta\oplus\eta'$.

\begin{prop}\label{prop:char-calculation}
For each profile $J$ we have 
\[ T(\gN(J)) \cong \eta \cdot \left(\prod_{i=0}^{f'-1} (\sigma_i \circ
  \hchar)^{t_i}\right)^{-1} |_{G_{K_{\infty}}} \] 
where 
\[ t_i =
\begin{cases}
  \gamma_i + \delta_{J^c}(i) & \text{if } i-1 \in J \\
  0 & \text{if } i-1\not\in J.
\end{cases}\]
Here $\delta_{J^c}$ is the characteristic function of the complement
of $J$ in $\Z/f'\Z$, and we are abusing notation by writing $\eta$ for
the function
$\sigma_i \circ \hchar^{k_i}$, which agrees with $\eta$ on $I_K$.

In particular the map $J \mapsto T(\gN(J))$ is injective on $\cP_{\tau}$.
\end{prop}

\begin{remark}\label{rk:cuspidal-char-niveau-1}
In the cuspidal case it is not \emph{a priori} clear that the formula in Proposition~\ref{prop:char-calculation} gives
a character of $G_{K_{\infty}}$ (rather than a character only when
restricted to $G_{L_{\infty}}$), but this is an elementary (if
somewhat painful) calculation using the
definition of the $\gamma_i$'s and the relation $\gamma_i +
\gamma_{i+f} = p-1$.
\end{remark}

\begin{proof}
 We begin by explaining how the final statement follows from the rest
 of the Proposition. First observe that if $J \in \cP_\tau$ then  $0 \le t_i \le p-1$ for all
 $i$. Indeed the only possibility for a contradiction would be if
 $\gamma_i = p-1$ and $i \not\in J$, but then the definition of
 $\cP_\tau$ requires that we cannot have $i-1 \in J$. Next, note that
 we never have $t_i = p-1$ for all $i$. Indeed, this would require $J
 = \Z/f'\Z$ and $\gamma_i=p-1$ for all~$i$, but by definition the
 $\gamma_i$ are not all equal to $p-1$.

The observations in the previous paragraph imply that (for $J \in \cP_\tau$) the character
$T(\gN(J))$ uniquely determines the integers $t_i$, and so it remains
to show that the integers $t_i$ determine the set $J$. If $t_i = 0$
for all $i$, then either $J = \varnothing$ or $J = \Z/f'\Z$ (for
otherwise there is a transition out of $J$, and $\delta_{J^c}(i) \neq
0$ for some $i-1 \in J$). But if $J = \Z/f'\Z$ then $\gamma_i = 0$ for
all $i$ and $\tau$ is scalar;\ but for scalar types we have $\Z/f'\Z
\not\in \cP_\tau$, a contradiction. Thus $t_i =0$ for all $i$ implies $J =
\varnothing$. 

For the rest of this part of the argument, we may therefore suppose $t_i \neq 0$
for some $i$, which forces $i-1 \in J$. The entire set $J$ will then
be determined by recursion if we can show that knowledge of $t_i$ along with
whether or not $i \in J$, determines whether or not $i-1 \in J$. Given
the defining formula for $t_i$, the only possible ambiguity is if $t_i
= 0$ and $\gamma_i +
\delta_{J^c}(i) = 0$, so that $\gamma_i = 0$ and $i \in J$. But the definition of $\cP_{\tau}$ requires $i-1
\in J$ in this case. This completes the proof.

We now turn to proving the formula for $T(\gN(J))$. We will use
Lemma~\ref{lem: generic fibres of rank 1 Kisin modules} applied at
$i=0$, for which we have to compute $\alpha_0 - d_0$ writing $\alpha_0
= \alpha_0(\gN)$. Recall
that we have already computed $\alpha_0(\gM(J)) - \alpha_0(\gN(J))$ in
the proof of Proposition~\ref{prop:ker-ext-maximal}. Since
$\alpha_0(\gM(J)) + \alpha_0(\gN(J)) = e(p^{f'}-1)/(p-1)$, taking the
difference between these formulas gives 
\[ 2 \alpha_0 =  [d_0 - c_0] - \sum_{j=1}^{f'} p^{f'-j} \gamma_j^* +
  \sum_{j \in S_0^c} p^{f'-j} \]
where $S_0^c$ consists of those $1\le j \le f$ such that $(j-1,j)$ is
a transition. Subtract $2[d_0]$ from both sides, and add the expression
$-[k_0-k'_0] + \sum_{j=1}^{f'} p^{f'-j} \gamma_j$ (which vanishes by definition) to the
right-hand side. Note that $[d_0 - c_0] - [k_0-k'_0] - 2[d_0]$ is
equal to $-2[k_0]$ if $0 \not\in J$, and to
$e(K'/K)-2[k_0-k_0']-2[k_0']$ if $0 \in J$.
Since $\gamma_j -
\gamma_j^* = 2\gamma_j - (p-1)$ if $j-1 \in J$ and is $0$ otherwise,
the preceding expression rearranges to give (after dividing by $2$)
\[ \alpha_0 - [d_0]  = -\kappa_0 + \sum_{j-1\in J} p^{f'-j}
  \gamma_j  + \sum_{j-1\in J, j\not\in J} p^{f'-j} = -\kappa_0 +
  \sum_{j=1}^{f'}  p^{f'-j} t_j\]
where $\kappa_0 = [k_0]$ if $0 \not\in J$ and $\kappa_0 = [k_0 - k'_0]
+[k'_0]$ if $0 \in J$. Since in either case $\kappa_0 \equiv k_0
\pmod{e(K'/K)}$ the result now follows from Lemma~\ref{lem: generic fibres of rank 1 Kisin modules}.
\end{proof}

\begin{defn}Let $\rbar:G_K\to\GL_2(\F')$ be representation. Then we
  say that a Breuil--Kisin module~$\gM$ with $\F'$-coefficients is a
  \emph{Breuil--Kisin model of~$\rbar$ of type~$\tau$} if~$\gM$ is an
  $\F'$-point of~$\cC^{\tau,\BT,1}$, and
  $T_{\F'}(\gM)\cong\rbar|_{G_{K_\infty}}$.
\end{defn}

Recall that for each continuous representation $\rbar:G_K\to\GL_2(\Fpbar)$, there
is an associated (nonempty) set of Serre weights~$W(\rbar)$ whose
precise definition is recalled in Appendix~\ref{sec: appendix on
  tame types}.

\begin{thm}
  \label{thm: unique serre weight}The~$\overline{\cZ}(J)$, with $J\in\cP_\tau$, are pairwise distinct closed
  substacks of $\cZ^{\tau,1}$. For each $J\in\cP_\tau$, there is a dense set of finite type
  points of $\overline{\cZ}(J)$ with the property that the corresponding Galois
  representations have $\sigmabar(\tau)_J$ as a Serre weight, and which
  furthermore admit a unique Breuil--Kisin model of type~$\tau$.
\end{thm}
\begin{proof}
Recall from Definition~\ref{def:scheme-theoretic images}
that $\overline{\cZ}(J)$ is defined to be the scheme-theoretic image
of a morphism $\Spec \Bdist \to \cZ^{\dd,1}.$  As in the proof of
Lemma~\ref{lem:ext images}, since the source
and target of this morphism are of finite presentation over $\F$, 
its image is a dense constructible subset of its scheme-theoretic image, 
and so contains a dense open subset, which we may interpret as 
a dense open substack $\cU$ of $\overline{\cZ}(J).$
From the definition of $\Bdist,$
the finite type points of~$\cU$ 
correspond to
reducible Galois representations admitting a model of type~$\tau$
and refined profile~$(J,r)$, for which~$(J,r)$ is maximal. 

That the~$\overline{\cZ}(J)$ are pairwise distinct is immediate  
from the above and~Proposition~\ref{prop:char-calculation}.  
 Combining this observation with Theorem~\ref{thm: dimension of refined profiles}, 
we see that by deleting the intersections of~$\overline{\cZ}(J)$ with
the~$\overline{\cZ}(J',r')$ for all refined profiles~$(J',r')\ne (J,r)$,
we obtain a dense open substack~$\cU'$ whose finite type points have
the property that every Breuil--Kisin model of type~$\tau$ of the
corresponding Galois representation has profile~$(J,r)$. The unicity of such a Breuil--Kisin model then follows 
from Corollary~\ref{cor:ker-ext-maximal-nonzero}.

It remains to show that every such Galois representation $\rbar$ has $\sigmabar(\tau)_J$ as a
Serre weight. 
Suppose first that $\tau$ is a principal series type. We claim that (writing $\sigmabar(\tau)_J =
\sigmabar_{\vec{t},\vec{s}} \otimes (\eta' \circ \det)$ as
in Appendix~\ref{sec: appendix on tame
  types}) we
have \[T(\mathfrak{N}(J))|_{I_K}=\eta'|_{I_K} \prod_{i=0}^{f-1}\omega_{\sigma_i}^{s_i+t_i}.\]To see this, note that by
Proposition~\ref{prop:char-calculation} it is enough to show that
$\eta|_{I_K}=\eta'|_{I_K} \prod_{i=0}^{f-1}\omega_{\sigma_i}^{s_i+2t_i}$, which
follows by comparing the central characters of $\sigmabar(\tau)_J$
and $\sigmabar(\tau)$ (or from a direct computation with the
quantities $s_i,t_i$).

Since~$\det\rbar|_{I_K}=\eta\eta'\varepsilonbar^{-1}$, we
have \[\rbar|_{I_K}\cong \eta'|_{I_K} \otimes \begin{pmatrix}
      \prod_{i=0}^{f-1}\omega_{\sigma_i}^{s_i+t_i} &*\\ 0 & \varepsilonbar^{-1}\prod_{i=0}^{f-1}\omega_{\sigma_i}^{t_i}
    \end{pmatrix}.\] The result then
follows from Lemma~\ref{lem: explicit Serre weights with our
  normalisations}, using Lemma~\ref{lem: list of things we need to
  know about Serre weights}(2) 
and the fact that the fibre of the morphism $\cC^{\tau,\BT,1} \to \cR^{\dd,1}$
above $\rbar$ is nonempty to see that $\rbar$ is not tr\`es ramifi\'ee.

The argument in the cuspidal case proceeds analogously, noting
that if the character $\theta$ (as in Appendix~\ref{sec: appendix on tame
  types}) corresponds to $\widetilde{\theta}$ under local class field theory
then $\widetilde{\theta} |_{I_K} = \eta'
\prod_{i=0}^{f'-1} \omega_{\sigma'_i}^{t_i}$, and that from central
characters we have 
$\eta\eta' = (\widetilde{\theta} |_{I_K})^2 
\prod_{i=0}^{f-1} \omega_{\sigma_i}^{s_i}$.
\end{proof}

\begin{rem}
  \label{rem: unique Serre weight isn't proved here but could be}With
  more work, we could use the results of~\cite{gls13} and our
  results on dimensions of families of extensions to strengthen
  Theorem~\ref{thm: unique serre weight}, showing that there is a
  dense set of finite type points of~$\cZbar(J)$ with the property that the corresponding Galois
  representations have $\sigmabar(\tau)_J$ as their \emph{unique} non-Steinberg Serre
  weight. In fact, we will prove this as part of our work on
  the geometric Breuil--M\'ezard conjecture in \cite{cegsA} 
  (which uses Theorem~\ref{thm: unique serre
    weight} as an input). 
\end{rem}
\subsection{Irreducible Galois representations} 
\label{subsec:irreducible}
We now show
that the points of $\cC^{\tau,\BT,1}$ which are irreducible (that
is, cannot be written as an extension of rank one Breuil--Kisin modules) lie
in a closed substack of positive codimension. This, in essence, amounts to bounding the dimensions of the Kisin varieties (for Breuil--Kisin modules of height at most $1$ with descent data of type $\tau$) corresponding to irreducible representations. We begin with the
following useful observation.

\begin{lem}
  \label{lem: closed points of irred}The rank two Breuil--Kisin modules with
  descent data  and $\Fpbar$-coefficients which are irreducible (that
  is, which cannot be written as an extension of rank~$1$ Breuil--Kisin
  modules with descent data)  are
  precisely those whose corresponding \'etale $\varphi$-modules are
  irreducible, or equivalently whose corresponding
  $G_K$-representations are irreducible.
\end{lem}
\begin{proof}Let~$\gM$ be a Breuil--Kisin module with descent data
  corresponding to a finite type point of~$\cC^{\tau,\BT,1}_{\dd}$,
  let~$M=\gM[1/u]$, and let~$\rhobar$ be the $G_K$-representation
  corresponding to~$M$. As noted in the proof of Lemma~\ref{lem: restricting to K_infty doesn't lose information about
    rbar}, $\rhobar$ is reducible if and only
  if~$\rhobar|_{G_{K_\infty}}$ is reducible, and by Lemma~\ref{lem:
    Galois rep is a functor if A is actually finite local}, this is
  equivalent to~$M$ being reducible. That this is in turn
  equivalent to~$\gM$ being reducible may be proved in 
  the same way as~\cite[Lem.\ 5.5]{MR3164985}. 
\end{proof}

Recall that~$L/K$ denotes the unramified quadratic extension; then the
irreducible representations~$\rhobar:G_K\to\GL_2(\Fpbar)$ are all
induced from characters of~$G_L$. Bearing in mind Lemma~\ref{lem:
  closed points of irred}, this means that we can study
irreducible Breuil--Kisin modules via a consideration of base-change of Breuil--Kisin modules
from $K$ to~$L$, and our previous
study of reducible Breuil--Kisin modules.
Since this will require us to consider Breuil--Kisin modules (and moduli
stacks thereof) over both $K$ and $L$, we will have to introduce
additional notation in order to indicate over which of the two fields
we might be working.   We do this simply by adding a subscript `$K$'
or `$L$' to our current notation.  We will also
omit other decorations which are being held fixed
throughout the present discussion. Thus we write $\cC_{K}^{\tau}$
to denote the moduli stack that was previously denoted $\cC^{\tau,\BT,1}$,
and $\cC_{L}^{\tau_{|L}}$ to denote the corresponding
moduli stack for Breuil--Kisin modules over $L$, with the type taken
to be the restriction $\tau_{|L}$ of $\tau$ from $K$ to $L$.
(Note that whether $\tau$ is principal series or cuspidal,
the restriction $\tau_{| L}$ is principal series.)

As usual we fix
a uniformiser~$\pi$ of $K$, which we also take to be our fixed
uniformiser of $L$.
Also, throughout
this section we take $K' =
L(\pi^{1/(p^{2f}-1)})$, so that $K'/L$ is the standard choice of
extension for $\tau$ and $\pi$ regarded as a type and uniformiser for~$L$. 

Write $\phi : W(k')[[u]] \to W(k')[[u]]$ for the automorphism fixing $u$ and acting on $W(k')$ via
the automorphism induced by the non-trivial automorphism of $k'/k$.
If~$\gP$ is a Breuil--Kisin module with descent data from~$K'$ to~$L$, then we
let
$\gP^{(f)}$ be the Breuil--Kisin module given by the pullback of $\gP$ along $\phi$,
and such that the descent data on $\gP^{(f)}$ is given by $\ghat(s
\otimes m) = \ghat(s) \otimes \hat{g}^{p^f}(m)$ for $s\in
W(k')[[u]]$ and $m \in \gP$.
In particular, we have $\gM(r,a,c)^{(f)} =
\gM(r',a',c')$ where $r'_i = r_{i+f}$, $a'_i=a_{i+f}$, 
and $c'_i=c_{i+f}$.

We let $\sigma$ denote the non-trivial automorphism of $L$ over $K$,
and write $G :=  \Gal(L/K)= \langle \sigma \rangle$, a cyclic group
of order two.
There is an action $\alpha$ of $G$ on 
$\cC_{L}$
defined via $\alpha_{\sigma}: \gP \mapsto \gP^{(f)}$. 
More precisely, this induces an action
of $G:= \langle \sigma \rangle$ on 
$\cC_{L}^{\tau_{|L}}$
in the strict\footnote{From a $2$-categorical perspective,
	it is natural to relax the notion
of group action on a stack so as to allow natural transformations,
rather than literal equalities, when relating multiplication
in the group to the compositions of the corresponding 
equivalences of categories arising in the definition of an action. 
An action
in which actual equalities hold is then called {\em strict}.  Since
our action is strict, we are spared from having to consider
the various $2$-categorical aspects of the situation that would
otherwise arise.}
sense that
$$\alpha_{\sigma} \circ \alpha_{\sigma} =
\id_{\cC_L^{\tau_{|L}}}.$$

We now define the fixed point stack for this action.

\begin{df}
	\label{def:fixed points}
	We let the fixed point stack
$(\cC_{L}^{\tau_{|L}})^G$ denote the stack
whose $A$-valued points consist of an $A$-valued point $\gM$
of $\cC_L^{\tau_{|L}}$, together with an isomorphism
$\imath: \gM \iso \gM^{(f)}$ which satisfies the cocycle condition
that the composite
$$\gM \buildrel \imath \over \longrightarrow \gM^{(f)} 
\buildrel \imath^{(f)} \over \longrightarrow (\gM^{(f)})^{(f)} = \gM$$
is equal to the identity morphism $\id_{\gM}$.
\end{df}

We now give another description of $(\cC_L^{\tau_{|L}})^G$, in terms 
of various fibre products, which is technically useful. 
This alternate description involves two steps.  In the first step,
we define fixed points of the automorphism $\alpha_{\sigma}$,
without imposing the additional condition that the fixed point data
be compatible with the relation $\sigma^2~=~1$ in~$G$.  Namely, we define
$$
(\cC_{L}^{\tau_{|L}})^{\alpha_{\sigma}}
:= 
\cC_{L}^{\tau_{|L}}
	\underset
{\cC_{L}^{\tau_{|L}}
	\times 
\cC_{L}^{\tau_{|L}}
}
{\times}
\cC_{L}^{\tau_{|L}}
	$$where the first morphism $\cC_{L}^{\tau_{|L}}\to\cC_{L}^{\tau_{|L}}
	\times 
\cC_{L}^{\tau_{|L}}$ is the diagonal, and the second is $\id\times\alpha_\sigma$.
	Working through the definitions,
	one finds
	that an $A$-valued point of $(\cC_L^{\tau_{|L}})^{\alpha_{\sigma}}$
	consists of a pair $(\gM,\gM')$ of objects of $\cC_L^{\tau_{|L}}$
	over $A$, equipped with isomorphisms $\alpha: \gM \iso \gM'$
	and $\beta: \gM \iso (\gM')^{(f)}$.  The morphism 
	$$(\gM,\gM',\alpha,\beta) \mapsto (\gM, \imath),$$
	where $\imath := (\alpha^{-1})^{(f)} \circ \beta: \gM \to \gM^{(f)}$,
	induces an isomorphism between $(\cC_L^{\tau_{|L}})^{\alpha_{\sigma}}$
		and the stack classifying points $\gM$ of $\cC_L^{\tau_{|L}}$
		equipped with an isomorphism 
		$\imath: \gM \to \gM^{(f)}$.
		(However, no cocycle condition has been imposed on $\imath$.)

	Let $I_{\cC_L^{\tau_{|L}}}$ denote the inertia stack
	of $\cC_L^{\tau_{|L}}.$
We define a morphism
$$(\cC_L^{\tau_{|L}})^{\alpha_{\sigma}} \to I_{\cC_L^{\tau_{|L}}}$$
via $$(\gM,\imath) \mapsto (\gM, \imath^{(f)}\circ \imath),$$
where, as in Definition~\ref{def:fixed points},
we regard the composite $\imath^{(f)}\circ \imath$ as an automorphism
of $\gM$ via the identity $(\gM^{(f)})^{(f)} = \gM.$
Of course, we also have the identity
	section $e: \cC_L^{\tau_{|L}} \to I_{\cC_L^{\tau_{|L}}}$.
	We now define
	$$(\cC_L^{\tau_{|L}})^G :=
	(\cC_L^{\tau_{|L}})^{\alpha_{\sigma}}
       	\underset{I_{\cC_L^{\tau_{|L}}}}{\times} 
	\cC_L^{\tau_{|L}}.
	$$
	If we use the description of $(\cC_L^{\tau|_{L}})^{\alpha_{\sigma}}$
	as classifying
	pairs $(\gM,\imath),$ then (just unwinding definitions)
	this fibre product classifies tuples $(\gM,\imath,\gM',\alpha)$,
	where $\alpha$ is an isomorphism $\gM \iso \gM'$ which furthermore
	identifies $\imath^{(f)}\circ \imath$
	with $\id_{\gM'}$.  Forgetting $\gM'$ and $\alpha$ then induces
	an isomorphism between~$(\cC_L^{\tau|_L})^G$, as defined
	via the above fibre product, and the stack defined in 
	Definition~\ref{def:fixed points}.

To compare this fixed point stack to $\cC^\tau_K$, we make the
following observations. Given a Breuil--Kisin module with descent
data from~$K'$ to~$K$, we obtain a Breuil--Kisin module with descent data
from~$K'$ to~$L$ via the obvious forgetful map. Conversely, given a
Breuil--Kisin module~$\gP$ with  descent data
from~$K'$ to~$L$, the additional data required to enrich this to a
Breuil--Kisin module with descent data from~$K'$ to~$K$ can be described as follows as follows:  let $\theta \in \Gal(K'/K)$ denote the unique element which fixes
$\pi^{1/(p^{2f}-1)}$ and acts nontrivially on $L$. Then to enrich
the descent data on $\gP$ to descent data from $K'$ to $K$, it is
necessary and 
sufficient to give an additive map $\hat\theta : \gP \to \gP$ satisfying
$\hat\theta(sm) = \theta(s)\hat\theta(m)$ for all $s \in \gS_{\F}$ and
$m \in \gP$, and such that $\hat\theta
\hat g \hat \theta = \hat g^{p^f}$ for all $g \in \Gal(K'/L)$.

In turn, the data of the additive map $\hat\theta:\gP\to\gP$ is
equivalent to giving the data of the map
$\theta(\hat\theta):\gP\to\gP^{(f)}$ obtained by
composing~$\hat\theta$ with the Frobenius on~$L/K$. The defining
properties of $\hat\theta$ are equivalent to asking that
this map is an isomorphism of Breuil--Kisin modules with descent data
satisfying the cocycle condition of Definition~\ref{def:fixed points};
accordingly, we have a natural morphism  $\cC_K^{\tau} \to
	(\cC_L^{\tau_{|L}})^G$, and a restriction morphism
        \numequation\label{eqn: restriction morphism}\cC_K^{\tau} \to
	\cC_L^{\tau_{|L}}. \end{equation}

The following simple lemma summarises the basic facts about
base-change in the situation we are considering.

\begin{lemma}\label{lem: fixed point stack iso}
	There is an isomorphism $\cC_K^{\tau} \iso
	(\cC_L^{\tau_{|L}})^G$.
\end{lemma}
\begin{proof}
 This follows immediately from the preceding discussion.	
\end{proof}

\begin{rem}
  \label{rem: the R version of the fixed point stack}In the proof of
  Theorem~\ref{thm: irreducible Kisin modules can be ignored} we
  will make use of the following analogue of Lemma~\ref{lem: fixed
    point stack iso} for \'etale $\varphi$-modules. Write~$\cR_K$,
  $\cR_L$ for the moduli stacks of Definition~\ref{defn: R^dd}, i.e.\
  for the moduli stacks of rank~$2$ \'etale $\varphi$-modules with
  descent data respectively to~$K$ or to~$L$. Then we have an action
  of~$G$ on~$\cR_L$ defined
  via~$M\mapsto M^{(f)}$, with $M^{(f)}$ defined analogously to $\gP^{(f)}$, and we
  define the fixed point stack~$(\cR_L)^G$ exactly as in
  Definition~\ref{def:fixed points}: namely an $A$-valued point
  of~$(\cR_L)^G$ consists of an $A$-valued point~$M$ of $\cR_L$,
  together with an isomorphism $\iota:M\isoto M^{(f)}$ satisfying the
  cocycle condition. The preceding discussion goes through in this
  setting, and shows that there is an isomorphism
  $\cR_K\isoto (\cR_L)^G$.

We also note that 
the morphisms $\cC_K^\tau \to \cC_L^{\tau_{|L}}$ and
$\cC_K^\tau \to \cR_K$ 
induce a monomorphism
\numequation\label{eqn:C into R mono base change}  \cC_K^{\tau} \hookrightarrow \cC_L^{\tau_{|L}} \times_{\cR_L}
\cR_K\end{equation}
One way to see this is to rewrite this morphism (using the previous discussion) 
as a morphism
$$(\cC_L^{\tau_{|L}})^G \to \cC_L^{\tau_{|L}} \times_{\cR_L} (\cR_L)^G,$$
and note that the descent data via $G$ on an object classified by
the source of this morphism is determined by the induced descent data on its
image in $(\cR_L)^G$.
\end{rem}

We now use the Lemma~\ref{lem: fixed point stack iso} to study the locus of finite type points
of $\cC_K^{\tau}$ which correspond to irreducible Breuil--Kisin modules. 
Any irreducible Breuil--Kisin module over $K$ becomes reducible when restricted to $L$,
and so may be described as an extension
$$0 \to \gN \to \gP \to \gM \to 0,$$
where $\gM$ and $\gN$ are 
Breuil--Kisin modules of rank one with descent data from $K'$ to $L$,
and $\gP$ is 
additionally
equipped with an isomorphism $\gP \cong \gP^{(f)}$,
satisfying the cocycle condition of Definition~\ref{def:fixed
  points}. 

Note that the characters
$T(\gM)$, $T(\gN)$ of $G_{L_\infty}$ are distinct and cannot be
extended to characters of $G_K$. Indeed, this condition is plainly
necessary for an extension~$\gP$ to arise as the base change of an irreducible
Breuil--Kisin module 
(see the
proof of  Lemma~\ref{lem: restricting to K_infty doesn't lose information about
    rbar}). 
Conversely, if $T(\gM)$, $T(\gN)$ of $G_{L_\infty}$ are distinct and cannot be
extended to characters of $G_K$, then  for any $\gP \in \Ext^1_{\K{\F}}(\gM,\gN)$
whose descent data can be enriched to give descent data from $K'$ to $K$, this enrichment is necessarily irreducible. 
In  particular, the existence of such a~$\gP$ implies that 
the descent data
on $\gM$ and $\gN$ cannot be enriched to give descent data from~$K'$ to~$K$. 

We additionally have the following observation. 

\begin{lemma}\label{lem:nonempty-then-map}
If $\gM,\gN$ are such that there is an extension \[0 \to \gN \to \gP
  \to \gM \to 0\] whose descent data can be enriched to give an irreducible
Breuil--Kisin module over~$K$, then there exists a
nonzero map $\gN \to \gM^{(f)}$.
\end{lemma}

\begin{proof}
The composition $\gN
\to \gP \xrightarrow{\hat\theta}  \gP \to \gM$, in which first and
last arrows are the natural inclusions and projections, must be
nonzero (or else $\hat\theta$ would give descent data on $\gN$ from
$K'$ to $K$). It is not itself a map of Breuil--Kisin modules, because $\hat\theta$
is semilinear, but is a map of Breuil--Kisin modules when viewed as a map $\gN \to \gM^{(f)}$.
\end{proof}

We now consider (for our fixed~$\gM$, $\gN$, and working over~$L$
rather than over~$K$) the scheme $\Spec B^{\dist}$ as in
Subsection~\ref{subsec:universal families}. Following
Lemma~\ref{lem:nonempty-then-map}, we assume that there exists a
nonzero map $\gN \to \gM^{(f)}$.  The observations made above show
that we are in the strict case, and thus that
$\Spec A^{\dist} = \Gm\times \Gm$ and that furthermore we may (and do)
set $V = T$.  We consider the fibre product with the restriction morphism~\eqref{eqn: restriction morphism}
$$Y(\gM,\gN):=\Spec B^{\dist} \times_{\cC_L^{\tau_{|L}}}\cC_K^{\tau}.$$

Let $\Gm \hookrightarrow \Gm\times\Gm$ be the diagonal closed immersion,
and let $(\Spec B^{\dist})_{|\Gm}$ denote the pull-back of $\Spec B^{\dist}$
along this closed immersion.
By Lemma~\ref{lem:nonempty-then-map}, the projection $Y(\gM,\gN) \to \Spec B^{\dist}$
factors through
$(\Spec B^{\dist})_{|\Gm},$
and combining this with Lemma~\ref{lem: fixed point stack iso} we see
that  $Y(\gM,\gN)$ may also be described as the fibre product
$$(\Spec B^{\dist})_{|\Gm} \times_{\cC_L^{\tau_{|L}}} (\cC_L^{\tau_{|L}})^G.$$

Recalling the warning of Remark~\ref{rem: potential confusion of two lots of Gm times
  Gm}, Proposition~\ref{prop: construction of family monomorphing to C
  and R} 
now shows that there is a monomorphism
$$[ (\Spec B^{\dist})_{|\Gm} / \Gm\times\Gm] \hookrightarrow \cC_L^{\tau_{|L}},$$
and thus, by
Lemma~\ref{lem: morphism from quotient stack is a monomorphism},
that there is an isomorphism
$$
(\Spec B^{\dist})_{|\Gm}
\times_{\cC_L^{\tau_{|L}}}
(\Spec B^{\dist})_{|\Gm}
\iso
(\Spec B^{\dist})_{|\Gm} \times \Gm\times \Gm.$$
(An inspection of the proof of Proposition~\ref{prop: construction of family monomorphing to C and R} shows that in fact
this result is more-or-less proved directly,
as the key step in proving the proposition.)
An elementary manipulation with fibre products then shows that there
is an isomorphism
$$Y(\gM,\gN) \times_{(\cC_L^{\tau_{|L}})^G} Y(\gM,\gN)
\iso Y(\gM,\gN)\times \Gm\times\Gm,$$
and thus, by another application of 
Lemma~\ref{lem: morphism from quotient stack is a monomorphism},
we find that there is a monomorphism
\numequation\label{eqn: mono from Y}[Y(\gM,\gN)/\Gm\times\Gm] \hookrightarrow (\cC_L^{\tau_{|L}})^G.\end{equation}

We define $\cC_{\irred}$ to be the union over all such pairs $(\gM,\gN)$ of
the scheme-theoretic images of the various projections
$Y(\gM,\gN) \to (\cC_L^{\tau_{|L}})^G$.  Note that this
image 
depends on $(\gM,\gN)$ up to simultaneous 
unramified twists of $\gM$ and $\gN$, and there are only
finitely many such pairs $(\gM,\gN)$ up to such unramified twist. By
definition, $\cC_{\irred}$ is a closed substack of
$\cC^{\tau}_K$ 
which contains every finite
type point of $\cC^{\tau}_K$ corresponding to an irreducible Breuil--Kisin
module.

The following is the main result of this section.

\begin{thm}
  \label{thm: irreducible Kisin modules can be ignored} The
  closed substack  $\cC_{\irred}$ of $\cC_K^{\tau}=\cC^{\tau,\BT,1}$, which contains
 every finite type point of $\cC^{\tau}_K$ corresponding
 to an irreducible Breuil--Kisin module,
  has dimension strictly less than $[K:\Qp]$.
\end{thm}
\begin{proof}
  As noted above, there are only finitely many pairs $(\gM,\gN)$ up to
  unramified twist, so it is enough to show that for each of them, the
  scheme-theoretic image of the monomorphism~\eqref{eqn: mono from Y}
  has dimension less than $[K:\Q_p]$.

By \cite[\href{https://stacks.math.columbia.edu/tag/0DS6}{Tag 0DS6}]{stacks-project},
to prove the present theorem,
it then suffices to show that 
$\dim Y(\gM,\gN) \leq [K:\Q_p] + 1$
(since $\dim \Gm\times\Gm = 2$).
To establish this, it suffices to show, for each point
$x \in \Gm(\F'),$ where $\F'$ is a finite extension of~$\F$, 
that the dimension of the fibre $Y(\gM,\gN)_x$ is bounded by
$[K:\Q_p]$. After relabelling, as we may, the field $\F'$ as $\F$ and
the Breuil--Kisin modules $\gM_x$ and $\gN_x$ as $\gM$ and $\gN$, we may
suppose that in fact $\F'=\F$ and~$x=1$.

Manipulating
the fibre product appearing in the definition of~$Y(\gM,\gN)$, we find
that
\numequation
\label{eqn:Y fibre blahblah}
Y(\gM,\gN)_1 = 
\Ext^1_{\K{\F}}(\gM,\gN) \times_{\cC_L^{\tau_{|_N}}} \cC_K^{\tau}, 
\end{equation}
where the fibre product is taken with respect to the
morphism $\Ext^1_{\K{\F}}(\gM,\gN) \to \cC_L^{\tau}$ that 
associates the corresponding
rank two extension to an extension
of rank one Breuil--Kisin modules,
and the restriction
morphism~\eqref{eqn: restriction morphism}.

In order to bound the dimension of~$Y(\gM,\gN)_1$, it will be easier
to first embed it into another,
larger, fibre product, which we now introduce. Namely, the
monomorphism~\eqref{eqn:C into R mono base change} 
induces a monomorphism
$$Y(\gM,\gN)_1 \hookrightarrow Y'(\gM,\gN)_1 := 
\Ext^1_{\K{\F}}(\gM,\gN) \times_{\cR_L} \cR_K.$$
Any finite type point of this fibre product lies over a fixed isomorphism
class of finite type points
of $\cR_K$ (corresponding to some fixed irreducible Galois
representation); we let $P$ be a choice of such a point.  The
restriction of $P$ then lies in a fixed isomorphism class of finite
type points of $\cR_L$ (namely, the isomorphism
class of the direct sum
$\gM[1/u]\oplus \gN[1/u] \cong \gM[1/u] \oplus \gM^{(f)}[1/u]$).
Thus the projection $Y'(\gM,\gN)_1 \to \cR_K$ factors through
the residual gerbe of $P$, while the morphism $Y'(\gM,\gN)_1
\to \cR_L$ factors through the residual gerbe 
of 
$\gM[1/u]\oplus \gN[1/u] \cong \gM[1/u] \oplus \gM^{(f)}[1/u]$.
Since $P$ corresponds via
Lemma~\ref{lem: Galois rep is a functor if A is actually finite local}
to an irreducible Galois representation,
we find that $\Aut(P) = \Gm$.
Since 
$\gM[1/u]\oplus \gN[1/u] $  corresponds  via
Lemma~\ref{lem: Galois rep is a functor if A is actually finite local}
to the direct sum of two non-isomorphic Galois characters, we find
that $\Aut(\gM[1/u]\oplus \gN[1/u] ) = \Gm \times \Gm$.  

Thus we obtain monomorphisms
\nummultline
\label{eqn:Y fibre}
Y(\gM,\gN)_1 \hookrightarrow 
Y'(\gM,\gN)_1
\\
\hookrightarrow 
\Ext^1_{\K{\F}}(\gM,\gN)
\times_{[\Spec F'//\Gm\times \Gm]} [\Spec F'//\Gm]
\cong
\Ext^1_{\K{\F}}(\gM,\gN)
\times \Gm.
\end{multline}
In Proposition~\ref{prop:irred-bound} we obtain a description of the image of $Y(\gM,\gN)_1$
under this monomorphism which allows us to bound its dimension
by~$[K:\Qp]$, as required.
\end{proof}

We now prove the bound on the dimension of~$Y(\gM,\gN)_1$ 
that we used in the proof of Theorem~\ref{thm: irreducible Kisin
  modules can be ignored}. Before establishing this bound, we make some further remarks.
To begin with, we remind the reader 
that we are working with Breuil--Kisin modules, \'etale $\varphi$-modules, etc.,
over $L$ rather than $K$, so that e.g.\
the structure parameters of $\gM, \gN$ are periodic modulo $f' = 2f$
(not modulo $f$), and the pair $(\gM,\gN)$ has type $\tau|_L$.
We will readily apply various pieces of notation that were
introduced above in the
context of the field $K$, adapted in the obvious manner to the context
of the field $L$.
(This applies in particular to the notation $\Conefrac$, $\Czerofrac$, etc.\
introduced in Definition~\ref{notn:calh}.)

We write
$m, n$ for the standard generators of $\gM$ and
$\gN$. 
The existence of the nonzero map $\gN \to \gM^{(f)}$ implies that
$\alpha_i(\gN) \ge \alpha_{i+f}(\gM)$ for all $i$, and also that
$\prod_i a_i = \prod_i b_i$. Thanks to the latter we will lose no
generality by assuming that $a_i = b_i =1 $ for all $i$. Let
$\tilde m$ be the standard generator for $\gM^{(f)}$.  The map
$\gN \to \gM^{(f)}$ will (up to a scalar) have the form
$n_i \mapsto u^{x_i} \tilde m_{i}$ for integers $x_i$ satisfying
$px_{i-1}-x_i = s_i - r_{i+f}$ for all $i$; thus
$x_i = \alpha_{i}(\gN) - \alpha_{i+f}(\gM)$ for all $i$. Since the
characters $T(\gM)$ and $T(\gN)$ are conjugate we must have
$x_i \equiv d_i - c_{i+f} \pmod{p^{f'}-1}$ for all $i$ (\emph{cf}.\
Lemma~\ref{lem: generic fibres of rank 1 Kisin modules}).  Moreover,
the strong determinant condition $s_i + r_i = e'$ for all $i$ implies
that $x_i = x_{i+f}$.

We stress that we make no claims about the
optimality of the following result; we merely prove ``just what we
need'' for our applications. Indeed the estimates of
\cite[Thm.~1.1]{MR2562792} and \cite[Thm.~1]{CarusoKisinVar} suggest that improvement should
be possible by the methods of those papers.

\begin{prop}\label{prop:irred-bound} 
We have  $\dim 
Y(\gM,\gN)_1 \le [K:\Qp]$. 

\end{prop}

\begin{remark}
  Since the image of $Y(\gM,\gN)_1$ in $\Ext^1_{\K{\F}}(\gM,\gN)$ lies
  in $\kExt^1_{\K{\F}}(\gM,\gN)$ with fibres that can be seen to  have dimension at most one,
  many cases of Proposition~\ref{prop:irred-bound} will  already follow from
  Remark~\ref{rem:half} (applied with $L$ in place of~$K$).
\end{remark}

\begin{proof}[Proof of Proposition~\ref{prop:irred-bound}]

Let $\gP = \gP(h)$ be an element of $\Ext^1_{\K{\F}}(\gM,\gN)$ whose
descent data can be enriched to give descent data from $K'$ to $K$, and
let $\tgP$ be such an enrichment.
By Lemma~\ref{lem:nonempty-then-map} (and the discussion preceding
that lemma) 
the \'etale $\varphi$-module $\gP[\frac 1u]$ is
isomorphic to $\gM[\frac 1u] \oplus \gM^{(f)}[\frac 1u]$. All
extensions of the $G_{L_{\infty}}$-representation $T(\gM[\frac 1u] \oplus
\gM^{(f)}[\frac 1u])$ to a representation of $G_{K_{\infty}}$ are
  isomorphic (and given by the induction of $T(\gM[\frac 1u])$ to~$G_{K_\infty}$), 
  so the same is true of the \'etale $\varphi$-modules with
descent data from $K'$ to $K$ that enrich the descent data on
$\gM[\frac 1u] \oplus \gM^{(f)}[\frac 1u]$. One such enrichment, which
we denote $P$, has $\hat\theta$ that interchanges $m$ and
$\tilde m$. Thus $\tgP[\frac 1u]$ is isomorphic to $P$.

As in the proof of Lemma~\ref{lem:nonempty-then-map},  the hypothesis that $T(\gM) \not\cong
T(\gN)$ implies that any non-zero map (equivalently, isomorphism) of
\'etale $\varphi$-modules with descent data $\lambda : \tgP[\frac 1u] \to P$ takes the submodule $\gN[\frac
1u]$ to $\gM^{(f)}[\frac 1u]$. We may scale the map $\lambda$ so that
it  restricts to the map $n_i \to u^{x_i} \tilde m_i$ on $\gN$.  
Then there is an element $\xi \in \F^\times$ so that
$\lambda$ induces multiplication by $\xi$ on the common quotients $\gM[\frac 1u]$.
That is,  the map $\lambda$ may be assumed to have the form
\numequation\label{eq:lambdamap}
\begin{pmatrix}
n_{i} \\ m_{i}
\end{pmatrix} \mapsto
\begin{pmatrix}
u^{x_i} & 0 \\ \nu_i & \xi   
\end{pmatrix}
\begin{pmatrix}
\tilde m_{i} \\ m_{i}
\end{pmatrix}
\end{equation}
for some $(\nu_i) \in \F((u))^{f'}$. The condition that the map
$\lambda$ commutes with the descent data from $K'$ to $L$ is seen to be
equivalent to the condition that nonzero terms in $\nu_i$ have degree
congruent to $c_i -d_i + x_i \pmod{p^{f'}-1}$; or equivalently, if we
define $\mu_i := \nu_i u^{-x_i}$ for all $i$, that the tuple $\mu = (\mu_i)$
is an element of the set $\Czerofrac = \Czerofrac(\gM,\gN)$
of Definition~\ref{notn:calh}.

The condition that $\lambda$ commutes with $\varphi$ can be checked to give
\begin{equation*}
   \varphi \begin{pmatrix}
n_{i-1} \\ m_{i-1}
\end{pmatrix}
=  \begin{pmatrix}
    u^{s_i} & 0 \\ \varphi(\nu_{i-1}) u^{r_{i+f}-x_i} - \nu_i u^{r_i-x_i}
    & u^{r_i} 
  \end{pmatrix}\begin{pmatrix}
n_{i} \\ m_{i}
\end{pmatrix}.
\end{equation*}
The extension $\gP$ is of the form $\gP(h)$, for some $h \in \Cone$
as in Definition~\ref{notn:calh}.
The lower-left entry of the first matrix on the right-hand side of the
above equation must then be $h_i$. Since $r_{i+f}-x_i = s_i - px_{i-1}$,
the resulting condition can be rewritten  as
\[ h_i= \varphi(\mu_{i-1}) u^{s_i} - \mu_i  u^{r_i},\]
or equivalently that $h = \mumap(\mu)$. Comparing with
Remark~\ref{rem:explicit-ker-ext}, we recover  the fact that the
extension class of $\gP$ is an element
of~$\kExt^1_{\K{\F}}(\gM,\gN)$, and the tuple  $\mu$ determines an 
element of the space $\Hzero$ defined as follows.

\begin{defn}\label{defn:calw} The map $\mumap \colon \Czerofrac \to \Conefrac$ induces a map
  $\Czerofrac/\Czero \to \Conefrac/\mumap(\Czero)$, which we also
  denote
  $\mumap$. We let  $\Hzero \subset \Czerofrac/\Czero$
  denote the subspace consisting of elements $\mu$ such that
$\mumap(\mu) \in \Cone/\mumap(\Czero)$. 
\end{defn}

By the discussion following 
Lemma~\ref{lem:explicit-complex}, an element $\mu \in \Hzero$ determines an 
extension $\gP(\mumap(\mu))$. Indeed,
Remark~\ref{rem:explicit-ker-ext} and the proof of \eqref{eqn:
  computing kernel of Ext groups} taken together show that there is a natural isomorphism,
in the style of Lemma~\ref{lem:explicit-complex}, between the morphism
$\mumap : \Hzero \to \Cone/\mumap(\Czero)$ and the connection map
$\Hom_{\K{\F}}(\gM,\gN[1/u]/\gN) \to \Ext^1_{\K{\F}}(\gM,\gN)$, with
$\im\mumap$ corresponding to $\kExt^1_{\K{\F}}(\gM,\gN)$.

Conversely, let $h$ be an element of $\mumap(\Czerofrac) \cap \Cone$,
and set $\nu_i = u^{x_i} \mu_i$. The condition that there is a Breuil--Kisin module
$\tgP$ with descent data from $K'$ to $K$ and $\xi \in \F^{\times}$ such that $\lambda : \tgP[\frac1u]
\to P$ defined as above is an isomorphism is precisely the condition that 
the map $\hat\theta$ on $P$ pulls back via $\lambda$ to a map that
preserves $\gP$. One computes that this pullback is
\begin{equation*}
  \hat\theta \begin{pmatrix}
n_{i} \\ m_{i}
\end{pmatrix}
= \xi^{-1} \begin{pmatrix}
     -\nu_{i+f}  & u^{x_i} \\
 (\xi^2-\nu_i \nu_{i+f}) u^{-x_i} & \nu_i
   \end{pmatrix}
  \begin{pmatrix}
  n_{i+f} \\ m_{i+f}
  \end{pmatrix}
\end{equation*}
recalling that $x_i =x_{i+f}$. 

We deduce that $\hat\theta$ preserves $\gP$
precisely when the $\nu_i$ are
integral and $\nu_i \nu_{i+f} \equiv \xi^2 \pmod{u^{x_i}}$ for
all~$i$.  For~$i$ with $x_i=0$
the latter condition is automatic given the former, which is
equivalent to the condition that $\mu_i$ and $\mu_{i+f}$ are both
integral. If instead $x_i > 0$, then we have the nontrivial
condition $\nu_{i+f} \equiv \xi^2 \nu_{i}^{-1} \pmod{u^{x_i}}$; in other
words that $\mu_i, \mu_{i+f}$ have $u$-adic valuation exactly $-x_i$,
and their principal parts determine one another via the equation
$\mu_{i+f}  \equiv \xi^2 (u^{2x_i} \mu_i )^{-1}
\pmod{1}$. 

 Let
$\mathbf{G}_{m,\xi}$ be the multiplicative group with parameter
$\xi$. We now (using the notation of Definition~\ref{defn:calw}) define  $\Hzero' \subset \Czerofrac/\Czero \times
\mathbf{G}_{m,\xi}$ to be the subvariety 
consisting of the pairs
$(\mu,\xi)$ with exactly the preceding properties; that is, we
regard~$\Czerofrac/\Czero$ as an Ind-affine space in the obvious way, and
define~$\Hzero'$ to be the pairs $(\mu,\xi)$ satisfying 
\begin{itemize}
\item if $x_i=0$ then $\val_i \mu = \val_{i+f} \mu =\infty$,
	and
\item if $x_i >0$ then $\val_i \mu = \val_{i+f} \mu = -x_i$ and $\mu_{i+f}
  \equiv \xi^2  (u^{2x_i} \mu_i)^{-1} \pmod{u^0}$
\end{itemize} 
where we write $\val_i \mu$ for the $u$-adic valuation of $\mu_i$, putting $\val_i \mu = \infty$ when $\mu_i$ is integral.

Putting all this together with~\eqref{eqn:Y fibre blahblah}, we find that the map 
\[ \Hzero' \cap (\Hzero \times \mathbf{G}_{m,\xi}) \to
  Y(\gM,\gN)_1 \] 
sending $(\mu,\xi)$ to the pair $(\gP,\tgP)$ is a well-defined
surjection, 
where $\gP =
\gP(\mumap(\mu))$, $\tgP$ is the enrichment of $\gP$ to a Breuil--Kisin
module with descent data from $K'$ to $K$ in which $\hat\theta$ is
pulled back to  $\gP$ from $P$ via the map $\lambda$ as in
\eqref{eq:lambdamap}. 
(Note that~$Y(\gM,\gN)_1$ is reduced and
of finite type, for
example by~\eqref{eqn:Y fibre}, so the surjectivity can be checked on
$\Fpbar$-points.)
In particular $\dim Y(\gM,\gN)_1 \le \dim \Hzero'.$

Note that $\Hzero'$ will be empty if for some $i$ we have $x_i > 0$ but
$x_i + c_i-d_i \not\equiv 0 \pmod{p^{f'}-1}$ (so that $\nu_i$ cannot
be a $u$-adic unit). 
Otherwise, the dimension of $\Hzero'$ is easily computed to be
$D =  1+\sum_{i=0}^{f-1} \lceil x_i/(p^{f'}-1) \rceil$ (indeed if~$d$ is the number of nonzero
$x_i$'s, then $\Hzero' \cong \Gm^{d+1} \times \Ga^{D-d-1}$), 
 and since
$x_i \le e'/(p-1)$ we find that $\Hzero'$ has dimension at most  $1 + \lceil e/(p-1)  \rceil f$.
This establishes  the bound  $\dim 
Y(\gM,\gN)_1 \le 1 + \lceil e/(p-1) \rceil f$. 

Since $p > 2$ this bound already establishes the theorem when $e
> 1$. 
 If instead $e=1$ the above bound gives  $\dim Y(\gM,\gN) \le [K:\Qp]
 + 1$. Suppose for the sake of
 contradiction that equality holds. This is only possible if $\Hzero'
 \cong \Gm^{f+1}$, $\Hzero' \subset \Hzero \times \mathbf{G}_{m,\xi}$,
 and $x_i = [d_i - c_i] > 0$ for all
 $i$. 
 Define $\mu^{(i)}
 \in \Czerofrac$ to be the element such that $\mu_{i} =
 u^{-[d_{i}-c_{i}]}$, and $\mu_j = 0$ for $j \neq i$. Let $\F''/\F$ be
 any finite extension such that $\#\F'' > 3$.  For each nonzero $z \in \F''$
 define $\mu_z = \sum_{j \neq i,i+f} \mu^{(i)}  + z \mu^{(i)} + z^{-1}
\mu^{(i+f)}$, so that  $(\mu_z, 1)$ is an element of  $\Hzero'(\F'')$.
Since  $\Hzero' \subset \Hzero \times \mathbb{G}_{m,\xi}$ and $\Hzero$ is
 linear, the differences between the $\mu_z$ for varying $z$ lie in
 $\Hzero(\F'')$, and (e.g.\ by considering $\mu_1 - \mu_{-1}$ and $\mu_1 -
 \mu_{z}$ for any $z \in \F''$ with $z\neq z^{-1}$) we deduce that each
 $\mu^{(i)}$ lies in $\Hzero$. In particular 
each 
$\mumap(\mu^{(i)})$ lies in $\Cone$.

If $(i-1,i)$ were not a transition then (since $e=1$) we would have
either $r_i =0 $ or $s_i = 0$. The former would contradict
$\mumap(\mu^{(i)}) \in \Cone$ (since the $i$th component of
$\mumap(\mu^{(i)})$ would be $u^{-[d_i-c_i]}$, of negative degree),
and similarly the latter would contradict $\mumap(\mu^{(i-1)}) \in
\Cone$. Thus $(i-1,i)$ is a transition for all $i$. In fact the same
observations show more precisely that $r_i \ge  x_i = [d_i-c_i]$ and $s_i
\ge p x_{i-1} = p [d_{i-1}-c_{i-1}]$.  Summing these inequalities and subtracting
$e'$ we obtain $0 \ge p [d_{i-1}-c_{i-1}] - [c_i-d_i]$, and comparing
with  \eqref{eq:gammastar} 
shows that we must also have $\gamma_i^*=0$ for
all $i$. Since $e=1$ and $(i-1,i)$ is a transition for all $i$ the refined profile of the pair $(\gM,\gN)$ is
automatically maximal;\ but then we are in the exceptional case of
Proposition~\ref{prop:ker-ext-maximal}, 
which (recalling the proof of that Proposition) implies that $T(\gM) \cong T(\gN)$. 
This is the desired contradiction.
\end{proof}

\subsection{Irreducible components}\label{subsec: irred components}
We can now use our results on families of extensions of characters to
classify the irreducible components of the stacks~$\cC^{\tau,\BT,1}$
and~$\cZ^{\tau,1}$. In the article \cite{cegsA} 
we will combine
these results with results coming from Taylor--Wiles patching (in
particular the results of~\cite{geekisin,emertongeerefinedBM}) 
to describe the
closed points of each irreducible component of~$\cZ^{\tau,1}$ in terms
of the weight part of Serre's conjecture.
\begin{cor}
  \label{cor: the C(J) are the components}Each irreducible component
  of~$\cC^{\tau,\BT,1}$ is of the form~$\overline{\cC}(J)$ for some~$J$;
  conversely, each~$\overline{\cC}(J)$ is an irreducible component of~$\cC^{\tau,\BT,1}$. 
\end{cor}
\begin{rem}
  \label{rem: haven't yet proved the C(J) are distinct}Note that at
  this point we have not established that different sets~$J$ give
  distinct irreducible components~$\overline{\cC}(J)$; we will prove this in Section~\ref{subsec: map to
  Dieudonne stack} below by a consideration of Dieudonn\'e
modules. 
\end{rem}
\begin{proof}[Proof of Corollary~{\ref{cor: the C(J) are the
    components}}]By~Theorem~\ref{cor: Kisin moduli consequences of local models}(2), $\cC^{\tau,\BT,1}$ is
  equidimensional of dimension~$[K:\Qp]$. By construction, the~$\overline{\cC}(J)$ are irreducible
  substacks of~$\cC^{\tau,\BT,1}$, and by Theorem~\ref{thm:
    dimension of refined profiles} they also have dimension~$[K:\Qp]$, so  they are in fact
  irreducible components by~\cite[\href{https://stacks.math.columbia.edu/tag/0DS2}{Tag 0DS2}]{stacks-project}. 
  
  By Theorem~\ref{thm: irreducible Kisin
    modules can be ignored} and Theorem~\ref{thm: dimension of refined
    profiles}, we see that there is a closed substack
  $\cC_{\mathrm{small}}$ of~$\cC^{\tau,\BT,1}$ of dimension strictly
  less than~$[K:\Qp]$, with the property that every finite type point
  of~$\cC^{\tau,\BT,1}$ is a point of at least one of the~$\overline{\cC}(J)$
  or of~$\cC_{\mathrm{small}}$ (or both). Indeed, every extension of
  refined profile~$(J,r)$ lies
  on~$\overline{\cC}(J,r)$, by Remark~\ref{rem-all-pts}, so we can
  take~$\cC_{\mathrm{small}}$ to be the union of the
  stack~$\cC_{\mathrm{irred}}$ of Theorem~\ref{thm: irreducible Kisin
    modules can be ignored} and the stacks~$\overline{\cC}(J,r)$ for
  non-maximal profiles~$(J,r)$. 
  Since  $\cC^{\tau,\BT,1}$ is
  equidimensional of dimension~$[K:\Qp]$, it follows 
  that the~$\overline{\cC}(J)$ exhaust
  the irreducible components of~$\cC^{\tau,\BT,1}$, as required. 
\end{proof}

We now deduce a classification of the
irreducible components of $\cZ^{\tau,1}$. In the paper \cite{cegsA}
we will give 
a considerable refinement of
this, giving a precise description of the finite type points of the
irreducible components in terms of the weight part of Serre's conjecture.
  \begin{cor}
    \label{cor: components of Z are exactly the Z(J)}The irreducible
    components of~$\cZ^{\tau,1}$ are precisely the~$\overline{\cZ}(J)$
    for~$J\in\cP_\tau$, and if $J\ne J'$ then~$\overline{\cZ}(J)\ne\overline{\cZ}(J')$.
  \end{cor}
  \begin{proof}
    By Theorem~\ref{thm: identifying the vertical components}, if
    $J\in\cP_\tau$ then~$\overline{\cZ}(J)$ is an irreducible component of
    ~$\cZ^{\tau,1}$. 
    Furthermore, these~$\overline{\cZ}(J)$ are pairwise
    distinct by Theorem~\ref{thm: unique serre weight}.

    Since the morphism
    $\cC^{\tau,\BT,1}\to\cZ^{\tau,1}$ is scheme-theoretically
    dominant, it follows from Corollary~\ref{cor: the C(J) are the
      components} that each irreducible component of $\cZ^{\tau,1}$
    is dominated by some~$\overline{\cC}(J)$.  Applying Theorem~\ref{thm:
      identifying the vertical components} again, we see that if
    $J\notin\cP_\tau$ then~$\overline{\cC}(J)$ does not dominate an irreducible
    component, as required.
  \end{proof}

\subsection{Dieudonn\'e modules and the morphism to the gauge stack}\label{subsec: map to
  Dieudonne stack} 
We now study the images of the irreducible components  $\overline{\cC}(J)$
in the gauge stack $\cG_\eta$; 
this amounts to computing
the Dieudonn\'e modules and Galois
representations associated to the extensions of Breuil--Kisin modules that we
considered in Section~\ref{sec: extensions of rank one Kisin
  modules}. 
Suppose throughout this subsection that $\tau$ is a non-scalar type,
and that $(J,r)$ is a maximal refined profile. 
Recall that in the cuspidal case this entails that $i \in J$ if and
only if $i + f \not\in J$.

\begin{lemma}
	\label{lem:Dieudonne modules}
Let $\gP \in \Ext^1_{\K{\F}}(\gM,\gN)$ be an extension of type $\tau$
and refined profile $(J,r)$. Then for $i \in \Z/f'\Z$ we have $F=0$ on $D(\gP)_{\eta,i-1}$ if $i\in
J$, while $V=0$ on $D(\gP)_{\eta,i}$ if
$i\notin J$.
\end{lemma}
\begin{proof}
Recall that $D(\gP) = \gP/u\gP$. Let $w_i$ be the image of $m_i$ in
$D(\gP)$ if $i \in J$, and let $w_i$ be the image of $n_i$ in $D(\gP)$
if $i \not\in J$.  It follows easily from the
definitions that $D(\gP)_{\eta,i}$ is generated over~$\F$ by $w_i$.

Recall that the actions of $F,V$ on $D(\gP)$ are as specified in
Definition~\ref{def: Dieudonne module formulas}. In particular $F$ is
induced by $\varphi$, while $V$ is $\czero^{-1} \mathfrak{V}$ mod $u$ where $\mathfrak{V}$ is the
unique map on $\gP$ satisfying $\mathfrak{V} \circ \varphi =
E(u)$, and $\czero = E(0)$.  For the Breuil--Kisin module $\gP$, we have
\[\varphi(n_{i-1}) = b_i u^{s_i} n_i,\qquad \varphi(m_{i-1}) = a_i u^{r_i}
m_i + h_i n_i,\] and so
 one checks (using that $E(u) = u^{e'}$ in $\F$)
that 
$$\mathfrak{V}(m_i) = a_i^{-1} u^{s_i} m_{i-1} - a_{i}^{-1} b_i^{-1}
h_i n_{i-1} , \qquad \mathfrak{V}(n_i) = b_i^{-1} u^{r_i} n_{i-1}.$$

From Definition~\ref{df:extensions of profile $J$} and the
discussion immediately following it, we recall that if $(i-1,i)$ is not a transition
then $r_i = e'$,
$s_i=0$, and $h_i$ is divisible by $u$ (the latter because nonzero
terms of $h_i$ have degrees congruent to $r_i+c_i-d_i
\pmod{p^{f'}-1}$, and $c_i \not\equiv d_i$ since $\tau$ is non-scalar).
On the other hand if $(i-1,i)$ is a transition, then $r_i , s_i >0$,
and nonzero terms of $h_i$ have degrees divisible by
$p^{f'}-1$; in that case we write $h_i^0$ for the constant coefficient
of $h_i$, and we remark that $h_i^0$ does not vanish identically on $\Ext^1_{\K{\F}}(\gM,\gN)$.

Suppose, for instance, that $i-1 \in J$ and $i \in J$. Then
$w_{i-1}$ and $w_i$ are the images in $D(\gP)$ of $m_{i-1}$ and
$m_{i}$.  From the above formulas we see that $u^{r_i} = u^{e'}$ and
$h_i$ are both divisible by $u$, while on the other hand $u^{s_i} = 1$. We
deduce that $F(w_{i-1}) = 0$ and $V(w_i) = \czero^{-1} a_i^{-1}
w_{i-1}$.  Computing along similar lines,  it is easy to check the following four
cases.

\begin{enumerate}
\item $i-1\in J,i\in J$. Then  $F(w_{i-1}) = 0$ and $V(w_i) = \czero^{-1} a_i^{-1}
w_{i-1}$.

\item $i-1\notin J,i\notin J$. Then $F(w_{i-1})=b_{i}w_{i}$, $V(w_{i})=0$.
\item\label{item: interesting case} $i-1\in J$, $i\notin J$. Then $F(w_{i-1})=h_{i}^0w_{i}$, $V(w_{i})=0$.
\item $i-1\notin J$, $i\in J$. Then $F(w_{i-1})=0$,
  $V(w_{i})=-\czero^{-1} a_{i}^{-1}b_{i}^{-1}h_{i}^0w_{i-1}$.
\end{enumerate}
In particular, if $i\in J$ then $F(w_i)=0$, while if
$i\notin J$ then $V(w_{i+1})=0$. 
\end{proof}

Since $\cC^{\tau,\BT}$ is flat over $\cO$ by Theorem~\ref{cor: Kisin
  moduli
  consequences of local models}, 
it follows from Lemma~\ref{lem: maps to gauge stack as Cartier
    divisors} that the natural morphism $\cC^{\tau,\BT} \to \cG_{\eta}$ 
is determined by an $f$-tuple of effective Cartier divisors $\{\cD_j\}_{0 \le j < f}$
lying in the special fibre $\cC^{\tau,\BT,1}$. 
Concretely, 
$\cD_j$ is the zero locus of~$X_j$, which is the zero locus
of~$F:D_{\eta,j}\to D_{\eta,j+1}$. 
The zero locus of $Y_j$ (which is the zero locus of~$V:D_{\eta,j+1}\to
D_{\eta,j}$) is another
Cartier divisor $\cD_j'$. 
Since $\cC^{\tau,\BT,1}$ is reduced,
we conclude that each of $\cD_j$ and $\cD_j'$ is simply a union of irreducible components
of  $\cC^{\tau,\BT,1}$, each component appearing precisely once in
precisely one of either $\cD_j$ or $\cD_j'$.

\begin{prop}
\label{prop:Dieudonne divisors}
$\cD_j$ is equal to the union of the irreducible components~$\overline{\cC}(J)$ of
$\cC^{\tau,\BT,1}$ for those $J$ that contain
$j+1$. 
\end{prop}
\begin{proof}
Lemma~\ref{lem:Dieudonne modules} shows
that if $j+1\in J$, then $X_j=0$, while
if $j+1\notin J$, then $Y_j=0$. In the latter case, by an inspection
of case~\eqref{item: interesting case} of the proof of Lemma~\ref{lem:Dieudonne modules}, we have
$X_j=0$ if and only if  $j\in J$ 
and
$h_{j+1}^0=0$. Since~$h_{j+1}^0$ does not vanish identically on an
irreducible component, we see that the irreducible components on which $X_j$
vanishes identically are precisely those for which $j+1\in J$, as
claimed. 
\end{proof}

\begin{thm}
  \label{thm: components of C}The algebraic stack~$\cC^{\tau,\BT,1}$
  has precisely $2^f$ irreducible components, namely the irreducible substacks~$\overline{\cC}(J)$. 
\end{thm}
\begin{proof}
By Corollary~\ref{cor: the C(J) are the components}, we need only show
that if~$J\ne J'$   then $\overline{\cC}(J)\ne\cC(J')$; but this is immediate
from Proposition~\ref{prop:Dieudonne divisors}.
\end{proof}

\renewcommand{\theequation}{\Alph{section}.\arabic{subsection}} 
\appendix
\section{Serre weights and tame types}\label{sec: appendix on tame
  types} 
We begin by recalling some results from~\cite{MR2392355} on the Jordan--H\"older factors
of the reductions modulo $p$ of lattices in principal series
and cuspidal representations of $\GL_2(k)$,
following~\cite[\S3]{emertongeesavitt} (but with slightly different
normalisations than those of \emph{loc.\ cit.}).

Let $\tau$ be a tame inertial type. Recall from Section~\ref{subsec:
  notation} that we associate a representation~$\sigma(\tau)$ of
$\GL_2(\cO_K)$ to~$\tau$ as follows: if $\tau
\simeq \eta \oplus \eta'$ is a tame principal series type, then we set
$\sigma(\tau) := \Ind_I^{\GL_2(\cO_K)} \eta'\otimes \eta$, while 
if $\tau=\eta\oplus\eta^q$ is a tame cuspidal type, then $\sigma(\tau)$ is the
inflation to $\GL_2(\cO_K)$ of the cuspidal representation of $\GL_2(k)$
denoted by~$\Theta(\eta)$ in~\cite{MR2392355}. (Here we have
identified~$\eta,\eta'$ with their composites with~$\Art_K$.)

 Write $\sigmabar(\tau)$ for the
semisimplification of the reduction modulo~$p$ of (a
$\GL_2(\cO_K)$-stable $\cO$-lattice in) $\sigma(\tau)$. The action
of~$\GL_2(\cO_K)$ on~$\sigmabar(\tau)$ factors through~$\GL_2(k)$, so
the Jordan--H\"older factors~$\JH(\sigmabar(\tau))$ of~$\sigmabar(\tau)$ are Serre weights.
By the results of~\cite{MR2392355}, these Jordan--H\"older factors of
$\sigmabar(\tau)$ are pairwise non-isomorphic, and are parametrised
by a certain  set $\cP_\tau$ 
that we now recall.

Suppose first that $\tau=\eta\oplus\eta'$ is a tame principal series 
type. Set $f'=f$ in this case.
We define $0 \le \gamma_i \le p-1$ (for $i \in \Z/f\Z$) 
to be the unique integers not all equal to $p-1$ such that 
$\eta (\eta')^{-1} = \prod_{i=0}^{f-1}
\omega_{\sigma_i}^{\gamma_i}$. If instead $\tau = \eta \oplus \eta'$
is a cuspidal type, set $f'=2f$. We define $0 \le \gamma_i \le p-1$ (for $i \in \Z/f'\Z$) 
to be the unique integers such that 
$\eta (\eta')^{-1} = \prod_{i=0}^{f'-1}
\omega_{\sigma'_i}^{\gamma_i}$. Here the $\sigma'_i$ are the embeddings
$l \to \F$, where $l$ is the quadratic extension of $k$, 
$\sigma'_0$ is a fixed choice of embedding extending $\sigma_0$, and
$(\sigma'_{i+1})^p = \sigma'_i$ for all $i$.

If~$\tau$ is scalar then
we set $\cP_\tau=\{\varnothing\}$. 
Otherwise we have $\eta\ne\eta'$, and 
we let
$\cP_{\tau}$ be the collection of subsets $J \subset \Z/f'\Z$ 
satisfying the conditions:
\begin{itemize}
\item if $i-1\in J$ and $i\notin J$ then $\gamma_{i}\ne p-1$, and
\item if $i-1\notin J$ and $i\in J$ then $\gamma_{i}\ne 0$
\end{itemize}
and, in the cuspidal case, satisfying the further condition that $i
\in J$ if and only if $i+f \not\in J$.

The Jordan--H\"older factors of $\sigmabar(\tau)$ are by definition
Serre weights, and are
parametrised by $\cP_{\tau}$ as follows (see~\cite[\S3.2, 3.3]{emertongeesavitt}). For any $J\subseteq \Z/f'\Z$, we let $\delta_J$ denote
the characteristic function of $J$, and if $J \in \cP_{\tau}$
we define $s_{J,i}$ by
\[s_{J,i}=\begin{cases} p-1-\gamma_{i}-\delta_{J^c}(i)&\text{if }i-1 \in J \\
  \gamma_{i}-\delta_J(i)&\text{if }i-1\notin J, \end{cases}\]
and we set $t_{J,i}=\gamma_{i}+\delta_{J^c}(i)$ if $i-1\in J$ and $0$
otherwise.  Write $\vec{s}$ for the tuple~$(s_{J,i})$,
suppressing the $J$ from the notation for readability,
and similarly write $\vec{t}$ for the tuple~$(t_{J,i})$.
Recall that the Serre weight $\sigmabar_{\vec{t},\vec{s}}$ is defined in \eqref{eq:serreweight}.

In the principal series case we let
$\sigmabar(\tau)_J
:=\sigmabar_{\vec{t},\vec{s}}\otimes\eta'\circ\det$ for each $J \in \cP_\tau$;
the $\sigmabar(\tau)_J$ are precisely the Jordan--H\"older factors of
$\sigmabar(\tau)$. 

In the cuspidal case, one checks that $s_{J,i} = s_{J,i+f}$ for all
$i$, and also that  the character $\eta' \cdot
\prod_{i=0}^{f'-1} (\sigma'_i)^{t_{J,i}} : l^{\times} \to
\F^{\times}$ factors as $\theta \circ N_{l/k}$ where $N_{l/k}$
is the norm map. We let $\sigmabar(\tau)_J
:=\sigmabar_{0,\vec{s}}\otimes\theta \circ\det$ for each $J \in \cP_\tau$;
the $\sigmabar(\tau)_J$ are precisely the Jordan--H\"older factors of
$\sigmabar(\tau)$. 

\begin{aremark}\label{arem: wtf were we thinking in EGS}
  The parameterisations above are easily deduced from those given in
  \cite[\S3.2, 3.3]{emertongeesavitt} for the Jordan--H\"older factors
  of the representations $\Ind_I^{\GL_2(\cO_K)} \eta'\otimes \eta$
  and~$\Theta(\eta)$. (Note that there is a minor mistake
  in~\cite[\S3.1]{emertongeesavitt}: since the conventions
  of~\cite{emertongeesavitt} regarding the inertial Langlands
  correspondence agree with those of~\cite{geekisin}, the explicit
  identification of~$\sigma(\tau)$ with a principal series or cuspidal
  type in ~\cite[\S3.1]{emertongeesavitt} is missing a dual. The
  explicit parameterisation we are using here is of course independent
  of this issue.

  This mistake has the unfortunate effect that various
  explicit formulae in~\cite[\S7]{emertongeesavitt} need to be
  modified in a more or less obvious fashion; note that since
  ~$\sigma(\tau)$ is self dual up to twist, all formulae can be fixed
  by making twists and/or exchanging~$\eta$ and~$\eta'$. In
  particular, the definition of the strongly divisible module
  before~\cite[Rem.\ 7.3.2]{emertongeesavitt} is incorrect as written,
  and can be fixed by either reversing the roles of~$\eta,\eta'$ or
  changing the definition of the quantity~$c^{(j)}$ defined there.)
\end{aremark}

\begin{aremark}
In the cuspidal case, write $\eta$ in the form 
$(\sigma'_0)^{(q+1)b+1+c}$ where $0\le b\le q-2$, $0\le c\le 
q-1$. 
Set $t'_{J,i} = t_{J,i+f}$ for integers $1\le i \le f$. Then one can check
that $\sigmabar(\tau)_J = \sigmabar_{\vec{t}',\vec{s}} \otimes
(\sigma_0^{(q+1)b + \delta_J(0)} \circ \det).$
\end{aremark}

We now recall some facts about the set of Serre weights~$W(\rbar)$
associated to a representation $\rbar:G_K\to\GL_2(\Fpbar)$. 

\begin{adefn}\label{def:de rham of type}
 We say that a crystalline representation $r : G_K \to
  \GL_2(\Qpbar)$ has \emph{type $\sigmabar_{\vec{t},\vec{s}}$}
  provided that for each embedding $\sigma_j : k \into \F$ there is an
  embedding $\widetilde{\sigma}_j : K \into \Qpbar$ lifting $\sigma_j$
  such that the $\widetilde{\sigma}_j$-labeled Hodge--Tate weights
  of~$r$ are $\{-s_j-t_j,1-t_j\}$, and the remaining $(e-1)f$ pairs of  Hodge--Tate weights
  of~$r$ are all $\{0,1\}$. \emph{(}In particular the representations of
   type $\sigmabar_{\vec{0},\vec{0}}$ \emph{(}the trivial weight\emph{)} are the same as those of Hodge type $0$.\emph{)}
\end{adefn}

\begin{adefn}\label{def:serre weights}
Given a representation $\rbar:G_K\to\GL_2(\Fpbar)$ we define $W(\rbar)$
to be the set of Serre weights $\sigmabar$ such that $\rbar$ has a
crystalline lift of type $\sigmabar$.
\end{adefn}

There are several definitions of the set $W(\rbar)$ in the literature,
which by the papers~\cite{blggu2,geekisin,gls13} are known to be
equivalent (up to normalisation).  While the preceding definition 
is perhaps the most compact, it is the description of $W(\rbar)$
via the Breuil--M\'ezard conjecture that appears to be the most
amenable to generalisation;\ see 
\cite{MR3871496} for much more discussion.

 Recall that~$\rbar$ is~\emph{tr\`es ramifi\'ee} if it is a
    twist of an extension of the trivial character by the mod~$p$
    cyclotomic character, and if furthermore the splitting field of its projective
    image is \emph{not} of the form
    $K(\alpha_1^{1/p},\dots,\alpha_s^{1/p})$ for some
    $\alpha_1,\dots,\alpha_s\in\cO_K^\times$. 

\begin{alemma}\label{lem: list of things we need to know about Serre weights}\leavevmode
  \begin{enumerate}
  \item If ~$\tau$ is a tame type, then $\rbar$ has a potentially
    Barsotti--Tate lift of type~$\tau$ if and only
    if $W(\rbar)\cap\JH(\sigmabar(\tau))\ne 0$. 
  \item The following conditions are equivalent:
    \begin{enumerate}
    \item $\rbar$ admits a potentially Barsotti--Tate lift of some tame type.
    \item $W(\rbar)$ contains a non-Steinberg Serre weight.
    \item $\rbar$ is not tr\`es    ramifi\'ee.
    \end{enumerate}
  \end{enumerate}
\end{alemma}

\begin{proof}
  This is  \cegsBtresramlemma.
\end{proof}

\begin{alem}
  \label{lem: explicit Serre weights with our normalisations} 
  Suppose that~$\sigmabar_{\vec{t},\vec{s}}$ is a non-Steinberg
  Serre weight. Suppose that~$\rbar:G_K\to\GL_2(\Fpbar)$ is a reducible
  representation satisfying \[\rbar|_{I_K}\cong
    \begin{pmatrix}
      \prod_{i=0}^{f-1}\omega_{\sigma_i}^{s_i+t_i} &*\\ 0 & \varepsilonbar^{-1}\prod_{i=0}^{f-1}\omega_{\sigma_i}^{t_i}
    \end{pmatrix}
    ,\] and that $\rbar$ is not tr\`es ramifi\'ee.
Then $\sigmabar_{\vec{t},\vec{s}}\in W(\rbar)$.
  \end{alem}
  \begin{proof} 
 Write $\rbar$ as an extension of characters $\overline{\chi}$ by $\overline{\chi}'$.
 It is straightforward from the classification of crystalline
 characters as in \cite[Lem.\ 5.1.6]{MR3871496} that there exist 
 crystalline lifts $\chi,\chi'$ of
 $\overline{\chi},\overline{\chi}'$ so that $\chi,\chi'$ have Hodge--Tate
 weights $1-t_j$ and $-s_j-t_j$ respectively at one embedding lifting each $\sigma_j$ and
 Hodge--Tate weights $1$ and $0$ respectively at the others. In the
 case that $\rbar$ is not the twist of an extension of
 $\varepsilonbar^{-1}$ by $1$ the result follows because the corresponding
    $H^1_f(G_K,\chi'\otimes\chi^{-1})$ agrees with the
    full~$H^1(G_K,\chi'\otimes\chi^{-1})$ (as a consequence of the
    usual dimension formulas for~$H^1_f$, \cite[Prop.\
    1.24]{nekovar}).  

If $\rbar$ is twist
of an extension of $\varepsilonbar^{-1}$ by $1$, the assumption that 
 $\sigmabar_{\vec{t},\vec{s}}$ is non-Steinberg implies $s_j = 0$
    for all $j$. The hypothesis that
$\rbar$ is not tr\`es ramifi\'ee guarantees that 
    $\rbar\otimes\prod_{i=0}^{f-1}\omega_{\sigma_i}^{-t_i} $ is finite
    flat, so has a Barsotti--Tate lift, and we deduce  that
    $\sigmabar_{\vec{t},\vec{0}} \in W(\rbar)$. 
\end{proof}

\section*{Conflicts of interest} None.

\section*{Financial support} The first author was supported in part by NSF grant DMS-1501064, 
  by a Royal Society University Research Fellowship,
  and by ERC Starting Grant 804176. The second author was supported in part by the
  NSF grants DMS-1303450, DMS-1601871,
  DMS-1902307, DMS-1952705, and DMS-2201242. 
The third author was 
  supported in part by a Leverhulme Prize, EPSRC grant EP/L025485/1, Marie Curie Career
  Integration Grant 303605, 
  ERC Starting Grant 306326, and a Royal Society Wolfson Research
  Merit Award. The fourth author was supported in part by NSF CAREER
  grant DMS-1564367 and  NSF grants 
  DMS-1702161 and DMS-1952566.

\bibliographystyle{amsalpha-custom}
\bibliography{dieudonnelattices}
\end{document}